\documentclass[a4paper,12pt]{amsart}

\usepackage{tikz-cd}

\usepackage[latin1]{inputenc} 
\usepackage{graphicx}
\usepackage[matrix,arrow,tips,curve]{xy}
\usepackage[english]{babel}
\usepackage{amsmath}
\usepackage{amssymb}

\usepackage{mathrsfs}

\usepackage{times}

\usepackage{enumerate,bm}

\newtheorem{thm}{Theorem}[section]
\newtheorem{lemma}[thm]{Lemma}
\newtheorem{thmdef}[thm]{Theorem - Definition}
\newtheorem{corollary}[thm]{Corollary}
\newtheorem{proposition}[thm]{Proposition}

\newtheorem{question}[thm]{Question}

\newtheorem*{thm*}{Theorem}

\theoremstyle{definition}
\newtheorem{definition}[thm]{Definition}

\newtheorem{example}[thm]{Example}
\newtheorem{remark}[thm]{Remark}
\newtheorem{remdefi}[thm]{Definition - Remark}
\newtheorem{prg}[thm]{}

\numberwithin{figure}{section}
\numberwithin{table}{section}

\newcommand{\ph}{\varphi}
\newcommand{\w}{\widetilde}
\newcommand{\ma}{\mathcal}
\newcommand{\la}{\longrightarrow}
\newcommand{\ol}{\mathcal{O}}
\newcommand{\wi}{\widehat}
\newcommand{\pr}{\mathbb{P}}
\newcommand{\Q}{\mathbb{Q}}
\newcommand{\C}{\mathbb{C}}
\newcommand{\R}{\mathbb{R}}
\newcommand{\Z}{\mathbb{Z}}
\newcommand{\N}{\mathcal{N}_1}
\newcommand{\Nu}{\mathcal{N}^1}

\newcommand{\dom}{\operatorname{dom}}

\newcommand{\Spec}{\operatorname{Spec}}
\newcommand{\Sing}{\operatorname{Sing}}
\newcommand{\Pic}{\operatorname{Pic}}

\newcommand{\im}{\operatorname{Im}}
\newcommand{\NE}{\operatorname{NE}}
\newcommand{\Exc}{\operatorname{Exc}}
\newcommand{\Bs}{\operatorname{Bs}}
\newcommand{\Supp}{\operatorname{Supp}}
\newcommand{\Lo}{\operatorname{Locus}}
\newcommand{\codim}{\operatorname{codim}}
\newcommand{\Eff}{\operatorname{Eff}}
\newcommand{\Nef}{\operatorname{Nef}}
\newcommand{\Chow}{\operatorname{Chow}}
\newcommand{\Mov}{\operatorname{Mov}}
\newcommand{\MCD}{\operatorname{MCD}}
\newcommand{\mov}{\operatorname{mov}}

\newcommand{\Bl}{\operatorname{Bl}}

\newcommand{\sm}{\text{\it sm}}
\newcommand{\pt}{\text{\it pt}}
\newcommand{\pts}{\text{\it pts}}
\newcommand{\reg}{\text{\it reg}}
\newcommand{\s}{\scriptscriptstyle}

\usepackage{etoolbox}
\patchcmd{\section}{\normalfont}{\normalfont\large}{}{}
\patchcmd{\subsection}{\bfseries}{\scshape\centering}{}{}
\patchcmd{\subsection}{-.5em}{.5em}{}{}

\makeatletter
\@namedef{subjclassname@2020}{\textup{2020} Mathematics Subject Classification}
\makeatother

\title[Fano $4$-folds with $b_2\geq 7$]{Towards the classification of Fano $4$-folds with $b_2\geq 7$}
\author{C.~Casagrande}
\address{Universit\`a di Torino,
Dipartimento di Matematica,
via Carlo Alberto 10,
10123 Torino - Italy}
\email{cinzia.casagrande@unito.it}

\calclayout 

\begin{document}
\begin{abstract}
  We study (smooth, complex) Fano $4$-folds $X$ with Picard number $\rho_X\geq 7$. We show that if $\rho_X>9$, then $X$ is a product of del Pezzo surfaces (Th.~\ref{sharpbound}), thus improving results in \cite{32,3folds}; the statement is now optimal.  In the range $\rho_X\in\{7,8,9\}$ we show that if $X$ is not a product of surfaces, and has no small elementary contraction, then it is the blow-up of a cubic $4$-fold along a special configuration of planes
  (Th.~\ref{finalintro}). When instead $\rho_X\geq 7$ and $X$ has a small elementary contraction, we study $X$ depending on its fixed prime divisors, giving explicit results on the geometry of $X$ in the framework of birational geometry.
  In particular for the boundary case $\rho_X=9$ we show that either $X$ is a product of surfaces, or $X$ belongs to two explicit families, or there is a sequence of flips $X\dasharrow X'$ where $X'$ is a smooth projective $4$-fold with an elementary contraction onto a $3$-fold (Th.~\ref{explicit}).

  In the paper we also give several results on rational contractions of fiber type of Fano $4$-folds, and more generally of Mori dream spaces; in particular we use some properties of del Pezzo surfaces over non-closed fields, applied to generic fibers.
\end{abstract}
\maketitle
\section{Introduction}
\noindent
Since the classification of  Fano $3$-folds in the 80's, there has been a lot of interest in the study of higher dimensional Fano varieties, starting from dimension $4$. With the introduction of the Lefschetz defect (see below), the author has developped a program to study (smooth, complex) Fano $4$-folds $X$ with large second Betti number $b_2(X)$
via birational geometry, in the framework of the Minimal Model Program.  Let us recall that, since $X$ is Fano, we have $b_2(X)=\rho_X$ where $\rho_X$ is the Picard number of $X$.

In this paper we focus on the explicit study and classification of  Fano $4$-folds $X$ with Picard number $\rho_X\geq 7$. 
Our first result is the following.
\begin{thm}[Cor.~\ref{explicitfinal}]\label{sharpbound}
Let $X$ be a smooth Fano $4$-fold with  $\rho_X>9$. Then $X\cong S_1\times S_2$, with $S_i$ del Pezzo surfaces.
\end{thm}
Let us note that the statement above is proven in \cite{32} for $\rho>12$, and in \cite[Cor.~1.4]{3folds} for $\rho=12$, therefore the theorem is new for the cases $\rho=10$ and $\rho=11$. It is now optimal, as we know one family of Fano $4$-folds with $\rho=9$ that are not products of surfaces.

We recall that, so far, there are only 5 known families of Fano $4$-folds with $\rho\in\{7,8,9\}$ that are not products. Three such families, with $\rho=7,8,9$, are given by the Fano models of the blow-up of $\pr^4$ at $6,7,8$ general points respectively (see \S\ref{fanomodel}), and two additional families with $\rho=7$ are constructed in \cite[Prop.~1.9]{3folds} (see \S\ref{newex}). Some candidates are also given in [\emph{ibid.}, Questions 7.6 and 7.25];
in this range of Picard numbers  we expect few families.

Besides Th.~\ref{sharpbound}, in this paper we perform a detailed study of Fano $4$-folds $X\not\cong S_1\times S_2$ with $\rho_X\in\{7,8,9\}$; we obtain many explicit results on the geometry of $X$, including some classification results, and new candidates (see Questions \ref{questionY}, \ref{questionquadric}, \ref{questioncubic}).
In particular we give a complete characterization of the case where $X$ does not have small elementary contractions, as follows.
\begin{thm}[Th.~\ref{final}]\label{finalintro}
         Let $X$ be a smooth Fano $4$-fold with $\rho_X\geq 7$, not isomorphic to a product of surfaces, and without small elementary contractions.
             
             Then $\rho_X\leq 9$ and there are a cubic $4$-fold $Z\subset \pr^5$ with at most isolated ordinary double points, and $s=\rho_X-1$ distinct planes  $A_1,\dotsc,A_s\subset Z$ intersecting pairwise in a point, such that $X$ is obtained from $Z$ by blowing-up one plane $A_i$ and successively the transforms of the other ones.\footnote{The resulting $4$-fold does not depend on the order of the blow-ups, see Lemma \ref{blowuporder}.} 
     \end{thm}
     We do not know whether a Fano $4$-fold as in the statement above does exist, see \S\ref{excubics} and Question \ref{questioncubic}, so on one hand we get more candidate families of Fano $4$-fold as blow-ups of cubic $4$-folds, on the other hand it may very well be that this case does not happen. Let us note that Th.~\ref{finalintro} builds on the study in \cite{32}, where under the same assumptions, the bound $\rho_X\leq 12$ was proven. We refer the reader to \ref{overviewfinal} for an overview of the proof of Th.~\ref{finalintro}.
     
\medskip
     
Consider now the remaining case of a Fano $4$-fold $X\not\cong S_1\times S_2$,   with $\rho_X\geq 7$, and having small elementary contractions.
Here we rely on the classification of {\bf fixed prime divisors}\footnote{A  prime divisor $D\subset X$ is fixed if  $D=\Bs|mD|$ for all $m\in\Z_{>0}$.} $D\subset X$ from \cite{eff} (see p.~\pageref{fixedprime} and Section \ref{prelfixed}). Given such $D$,
 there
 are a sequence of flips $X\dasharrow X'$ with $X'$ smooth, and an elementary divisorial contraction $\sigma\colon X'\to Y$, such that
$D$ is the transform of $\Exc(\sigma)$, and $Y$ is Fano. Moreover we have only four possibilities for $\sigma$: it is 
 the blow-up of a smooth point, or of a smooth irreducible curve, or
of
an isolated, locally factorial, terminal singularity with
$\Exc(\sigma)$ a quadric, or finally of an irreducible surface.
We call $D$ {\bf of type $(3,0)^{\sm}$, $(3,1)^{\sm}$, $(3,0)^Q$, or $(3,2)$} accordingly.

When $X$ has a small elementary contraction, it must contain some fixed prime divisor $D$ of type $(3,0)^{\sm}$, $(3,1)^{\sm}$, or $(3,0)^Q$ (Rem.~\ref{united}). 
We study $X$ depending on the type of fixed prime divisors that it contains; in particular in the first case, when $D$ is of type $(3,0)^{\sm}$, we show the following.
\begin{thm}[Th.~\ref{30}]\label{30intro}
Let $X$ be a smooth Fano $4$-fold with $\rho_X\geq 7$, having a fixed prime divisor of type $(3,0)^{\sm}$. Then  one of the following holds:
  \begin{enumerate}[$(i)$]
  \item $\rho_X\leq 9$ and $X$ has an elementary rational contraction onto a $3$-fold;\footnote{A \emph{rational contraction} of $X$ is a rational map $f\colon X\dasharrow Y$ that factors as a sequence of flips $X\dasharrow X'$ followed by a contraction  $f'\colon X'\to Y$, namely a morphism such that $f'_*\ol_{X'}=\ol_Y$ with $Y$ projective,
 see \S\ref{notation}; $f$ is \emph{elementary} if $\rho_X-\rho_Y=1$.}
  \item $\rho_X=7$,
$h^0(X,-K_X)\leq 15$,
and there is a sequence of flips $X\dasharrow X'$ such that $X'=\Bl_{6\mskip1mu\pts}Y$, $Y$ a smooth Fano $4$-fold with $\rho_Y=1$.
    \end{enumerate}
\end{thm}
Th.~\ref{30intro} refines the study in \cite{blowup}, where under the same assumptions, the bound $\rho_X\leq 12$ was proven. We refer the reader to \ref{overview30} for an overview of the proof of Th.~\ref{30intro}.

\medskip

The remaining possibility is that of a  Fano $4$-fold $X$ with $\rho_X\geq 7$, having a fixed prime divisor of type $(3,1)^{\sm}$ or $(3,0)^Q$, and no fixed prime divisor of type $(3,0)^{\sm}$; we study this case
by refining the results in \cite{small}. This allows to complete the proof of Th.~\ref{sharpbound}, and also
to give a partial result on the case of  Picard number $\rho=9$, as follows.
\begin{thm}[Cor.~\ref{explicitfinal}]\label{explicit}
  Let $X$ be a smooth Fano $4$-fold with $\rho_X= 9$.
Then one of the following holds:
\begin{enumerate}[$(i)$]
\item $X\cong S_1\times S_2$, with $S_i$ del Pezzo surfaces;
\item $X$ has an elementary rational contraction onto a $3$-fold;
\item  $X$ is the blow-up of $W$ along a normal surface $S$,
  where $W$ is the Fano model of $\Bl_{7\mskip1mu\pts}\pr^4$, and $S\subset W$ is the transform of a cubic scroll or
  a cone over a twisted cubic in
$\pr^4$, containing the blown-up points;
\item  $X$ is a blow-up of a cubic $4$-fold as in Th.~\ref{finalintro}.
     \end{enumerate}
   \end{thm}
   We note that our only example of Fano $4$-fold $X\not\cong S_1\times S_2$ with $\rho_X=9$ has an  elementary rational contraction onto a $3$-fold (see \S\ref{fanomodel}), while we do not know whether cases $(iii)$ and $(iv)$ do happen (see \cite[\S 7.2]{3folds} for $(iii)$, and \S\ref{excubics} for $(iv)$).

   \medskip
   
 Our general approach is to study Fano $4$-folds via their contractions and rational contractions, and
 there are a few tools and results that we use systematically and
 are crucial in proving the previous theorems; let us introduce them.

\medskip

\noindent{\bf Rational contractions of fiber type.} We recall that a \emph{rational contraction} is a rational map $f\colon X\dasharrow Y$ that factors as a sequence of flips $X\dasharrow X'$ followed by a \emph{contraction}  $f'\colon X'\to Y$, namely a morphism such that $f'_*\ol_{X'}=\ol_Y$, with $Y$ projective; $f$ is \emph{of fiber type} when $\dim Y<\dim X$.
We study in detail rational contractions of fiber type $f\colon X\dasharrow Y$ of Fano $4$-folds (more generally of Mori dream spaces),
see also \cite{fanos, fibrations}.

The case where $\dim Y=3$ is treated in \cite{3folds}, where a sharp bound on $\rho_X$ is given, together with a partial classification for the range $\rho_X\geq 7$, as follows.
\begin{thm}[\cite{3folds}, Theorems 1.3, 1.10, and 1.11]\label{CS}
    Let $X$ be a smooth Fano $4$-fold with $\rho_X\geq 7$, having a rational contraction onto a $3$-fold,
and    not isomorphic to a product of surfaces.
    Then $\rho_X\leq 9$, and one of the following holds:
    \begin{enumerate}[$(i)$]
\item $X$ has an elementary rational contraction onto a $3$-fold;
\item  $X$ is the blow-up of $W$ along a normal surface $S$,
where $W$ is the Fano model of $\Bl_{\rho_{\s X}-2\mskip2mu\pts}\pr^4$, and $S\subset W$ is the transform of a surface $A\subset\pr^4$ containing the blown-up points. The surface $A$ is either
 a cubic scroll, or
  a cone over a twisted cubic, or a sextic K3 surface with $A_1$ or $A_2$ singularities at the blown-up points; if $A$ is a K3 surface then $\rho_X=7$.
\end{enumerate}
 \end{thm}   
 We note that, in the setting of Th.~\ref{CS}, the case where we do not yet have a classification is when $X$ has an elementary rational contraction onto a $3$-fold.

 \medskip

 Consider now a rational contraction of fiber type $f\colon X\dasharrow Y$   with $\dim Y\in\{1,2\}$. We can choose a factorization of $f$ as $X\stackrel{\ph}{\dasharrow} X'\stackrel{f'}{\to} Y$ where $\ph$ is a sequence of flips, $X'$ is smooth,  and $f'$ is a $K$-negative contraction of fiber type. Let $F\subset X'$ be a general fiber; then $F$ is smooth and is either a del Pezzo surface, or a Fano $3$-fold.

Let $\N(X')$ be the vector space of one-cycles in $X'$, with $\R$-coefficients, modulo numerical equivalence, and
 let us consider $\N(F,X'):=\iota_*(\N(F))\subset\N(X')$, where $\iota\colon F\hookrightarrow X'$ is the inclusion. Then $d_f:= \dim\N(F,X')$ is an invariant of $f$ and does not depend on the choice of the resolution $f'$ (Def.-Rem.~\ref{dimfiber}).
 Moreover $d_f$ is equal to the Picard number of the generic fiber $X_\eta$ of $f'$, which is a smooth Fano variety over the non-closed field $K=\C(Y)$ of rational functions on $Y$ (Lemma \ref{genrho}).

 When $d_f$ is large enough, we show that $f$ factors through a rational contraction onto a $3$-fold, as follows.
  \begin{thm}[Th.~\ref{brasile}]\label{brasileintro}
    Let $X$ be a smooth Fano $4$-fold and $f\colon X\dasharrow Y$ a
    non-trivial$\,$\footnote{Namely  $Y\neq\{\pt\}$.} rational contraction  of fiber type.
  If $d_f\geq 5$, then
 $f$ can be factored as $X\stackrel{g}{\dasharrow} Z\stackrel{h}{\to} Y$, where $\dim Z=3$.
\end{thm}
To prove this, we use different strategies depending on the dimension of the base $Y$. When $\dim Y=2$, we consider the generic fiber $X_{\eta}$ of $f'\colon X'\to Y$ (notation as above),
which is a smooth del Pezzo surface  over the field $K=\C(Y)$.
When $d_f=\rho_{X_{\eta}}\geq 5$, there is a morphism $X_{\eta}\to C$ onto a curve over $K$ (Lemma \ref{cabofrio}), and we use results on spreading out to get the statement (see Lemmas \ref{spread} and \ref{induced}, and Section \ref{secbrasile}).

Instead, when $Y\cong\pr^1$, a general fiber $F\subset X'$ of $f'$ is a smooth, complex Fano $3$-fold with $\rho_F\geq d_f=\dim\N(F,X')$. When $d_f\geq 5$, we  use the monodromy action on $\Nu(F)$ and $\Nef(F)$ to extend a contraction  $F\to S$ with $\dim S=2$ to a relative contraction over some open subset of $Y=\pr^1$, see Prop.~\ref{P1}. 

\medskip

Using the results from \cite{3folds} on rational contractions onto $3$-folds, we deduce from Th.~\ref{brasileintro} the following result.
\begin{thm}[Th.~\ref{scala}]\label{scalaintro}
  Let $X$ be a smooth Fano $4$-fold with $\rho_X\geq 7$, not isomorphic to a product of surfaces, and $f\colon X\dasharrow Y$ a non-trivial rational contraction of fiber type.
 Then $d_f\leq 4$.
\end{thm}
\noindent The bound  $d_f\leq 4$ is sharp, see the discussion after Th.~\ref{scala}.

Theorem \ref{scalaintro} is very much related to how 
the faces of the cone of effective divisors can intersect the cone of movable divisors.
More precisely, let us consider $\Nu(X)$, the vector space of divisors with $\R$-coefficients up to  numerical equivalence, and in it the cones $\Eff(X)$ and $\Mov(X)$ respectively of effective and movable divisors (see Section \ref{notation}); we have $\Mov(X)\subset\Eff(X)$, and since $X$ is Fano, both cones are rational polyhedral.
If $\tau$ is a face of $\Eff(X)$, we say that $\tau$ is a {\bf movable face} if $\tau\cap\Mov(X)\neq\{0\}$, otherwise we say that $\tau$ is a {\bf fixed face} (Def.~\ref{fmface}).

Given a non-trivial rational contraction of fiber type $f\colon X\dasharrow Y$, we have that $\tau_f:=\langle [D]\in\Nu(X)\,|D$ is effective and $\Supp D$ does not dominate $Y\rangle$ is a movable face of $\Eff(X)$, of dimension $\rho_X-d_f$ (Def.-Rem.~\ref{dimfiber}, Rem.~\ref{tauf}). Conversely, given a proper, movable face $\tau$ of $\Eff(X)$, there exists a non-trivial rational contraction of fiber type $f\colon X\dasharrow Y$ with $d_f\geq\rho_X-\dim\tau$ (see Rem.~\ref{zumba}). This gives the following.
\begin{corollary}[Cor.~\ref{summary}]\label{summaryintro}
 Let $X$ be a smooth Fano $4$-fold with $\rho_X\geq 7$, not isomorphic to a product of surfaces.
  If $\tau$ is a movable face of $\Eff(X)$, then 
  $\dim\tau\geq\rho_X-4$. In particular the cone $\Eff(X)$ is generated by classes of fixed prime divisors.
\end{corollary}

\bigskip

 \noindent{\bf The Lefschetz defect.}
Another essential tool 
is the Lefschetz defect, an invariant of $X$ defined as follows. 
For any prime divisor $\iota\colon D\hookrightarrow X$, set $\N(D,X):=\iota_*(\N(D))\subset\N(X)$. Then we define the Lefschetz defect of $X$ as
$$\delta_X:=\max\{\codim\N(D,X)\,|\,D\text{ a prime divisor in }X\}.$$
We refer the reader to \cite{codim,delta3} for results on the Lefschetz defect of Fano varieties of arbitrary dimension; in dimension $4$ we have the following.
\begin{thm}[\cite{codim}, Cor.~1.3, \cite{3folds}, Th.~1.8]\label{starting}
Let $X$ be a smooth Fano $4$-fold with $\rho_X\geq 7$, not isomorphic to a product of surfaces. Then $\delta_X\leq 1$.
\end{thm}  
 
\medskip

\noindent{\bf Fixed prime divisors of type $(3,2)$.}
A fixed prime divisor  of type $(3,2)$ is  the exceptional divisor $E$ of an elementary divisorial contraction $\sigma\colon X\to Y$ such that $\dim \sigma(E)=2$ (see Th.-Def.~\ref{fixed}). If $X\not\cong S_1\times S_2$ and $\rho_X\geq 7$, we have $\delta_X\leq 1$ by Th.~\ref{starting}, therefore $\N(E,X)$ has  codimension at most one in $\N(X)$. We show the following.
\begin{thm}[Th.~\ref{sharp32}]\label{sharp32intro}
  Let $X$ be a smooth Fano $4$-fold with $\rho_X\geq 7$, not isomorphic to a product of surfaces, and $E\subset X$ a fixed prime divisor of type $(3,2)$ such that $\N(E,X)\subsetneq\N(X)$.
  Then $\rho_X\leq 9$ and
  $X$ has an elementary rational contraction onto a  $3$-fold.
\end{thm}
This improves the bound $\rho_X\leq 12$ shown in  \cite[Prop.~5.32]{blowup}. Besides the bound $\rho_X\leq 9$, in Th.~\ref{sharp32} we give several geometric properties of $X$, that could lead to an explicit classification of this case.

We refer the reader to \ref{overviewsharp32} for an overview of the proof of Th.~\ref{sharp32intro}.
 In particular,  we need  a generalization of Theorems \ref{brasileintro} and \ref{scalaintro} in the case of
rational contractions $f\colon X\dasharrow S$ with $\dim S=2$ that are {\bf quasi-elementary}, namely with $d_f=\rho_X-\rho_S$ (see Def.~\ref{defqes}), as follows.
\begin{thm}[Th.~\ref{zucca}]\label{zuccaintro}
  Let $X$ be a smooth Fano $4$-fold with $\rho_X\geq 7$, not isomorphic to a product of surfaces, and  $f\colon X\dasharrow S$ a quasi-elementary rational contraction onto a surface.   Then $\rho_X-\rho_S\leq 3$,
  and if $\rho_X-\rho_S>1$, then $f$ factors as $X\stackrel{g}{\dasharrow}Y\stackrel{h}{\dasharrow} S$ where $\dim Y=3$. 
\end{thm}
The bound $\rho_X-\rho_S\leq 3$ follows easily from the previous results, while the proof that, in the non-elementary case, $f$ factors through a $3$-fold,
 is rather long (see Section \ref{seczucca}). 
We work by contradiction: by assuming that the statement does not hold, we get an explicit description of $X$ as a sequence of flips and blow-ups from $\pr^2\times\pr^2$; finally we show that this construction never yields a Fano $4$-fold. See \ref{overviewzucca} for a more detailed overview of the proof.

\medskip

Let us conclude this Introduction by noting that the condition $\rho\geq 7$ arises naturally and is optimal for many of these results, like Th.~\ref{finalintro}, Th.~\ref{CS}, Th.~\ref{scalaintro}, Cor.~\ref{summaryintro},   Th.~\ref{starting}, and Th.~\ref{zuccaintro}; on the other hand
Fano $4$-folds with $\rho\leq 6$ exhibit different behaviours,
as we show in Examples \ref{es1}, \ref{es2}, and \ref{perone}, and in \S\ref{exX}.

\medskip

The structure of the paper is as follows.
In Section \ref{notation} we fix the notation and terminology, and we recall some preliminary results needed in the sequel.

Section \ref{ratfibertype} contains several preliminary notions and results on regular and rational contractions of fiber type $f\colon X\dasharrow Y$ of Mori dream spaces. In particular we define quasi-elementary and special rational contractions of fiber type and recall their main properties, and we introduce
the invariant $d_f$ mentioned above.

In Section \ref{secbrasile} we apply the previous results to Fano $4$-folds, and we prove Theorems \ref{brasileintro} and \ref{scalaintro} on rational contractions of fiber type, and Cor.~\ref{summaryintro} on movable faces of the cone of effective divisors.

 In Section \ref{prelfixed} we recall the classification of fixed prime divisors in Fano $4$-folds $X$ with 
 $\rho_X\geq 7$ in four types, and give many related properties that are used in the sequel.

 In Section \ref{seczucca} we prove Th.~\ref{zuccaintro} on quasi-elementary rational contractions onto a surface. In Section \ref{sec32} we prove Th.~\ref{sharp32intro} on fixed prime divisors of type $(3,2)$, and give some applications needed in the sequel.

 In Section \ref{cubic} 
  we prove 
  Th.~\ref{finalintro}, that characterizes Fano $4$-folds with no small contractions (with $\rho\geq 7$, and not products) as suitable blow-ups of cubic $4$-folds.

  In Section \ref{sec30} we prove
  Th.~\ref{30intro} on Fano $4$-folds with a fixed prime divisor of type $(3,0)^{\sm}$. Finally
in
 Section \ref{secsmall} we consider Fano $4$-folds with a fixed prime divisor of type $(3,1)^{\sm}$ or $(3,0)^Q$, and prove Theorems \ref{sharpbound} and \ref{explicit}; we also give a partial result on the case $\rho=8$ (Cor.~\ref{last}).

In Section \ref{secexamples} we collect several examples. We first recall the known examples of Fano $4$-folds that are not products, with $\rho\geq 7$ (\S\ref{fanomodel} and \S\ref{newex}), and then give other relevant explicit examples; in some cases we do not know whether these $4$-folds are actually Fano for large values of the Picard number.

\medskip

{\small \noindent{\bf Acknowledgements.}
I had the idea of using the properties of del Pezzo surfaces over non-closed fields at
the minicourse  ``Quotients of groups of birational transformations'' by Susanna Zimmermann, at the V Latin American School of Algebraic Geometry (V ELGA) in Brazil, August 2024;
I thank her and the organisers of the school.

I am also grateful to Brendan Hassett for explaining me that there exist families of cubic $4$-folds containing a configuration of planes as in Th.~\ref{finalintro}.

The author has been partially supported by PRIN 2022L34E7W ``Moduli spaces and birational geometry'', and is a member of GNSAGA, INdAM.}
 
{\small\tableofcontents}
    
    \section{Notation and preliminaries}\label{notation}
    \noindent If $\mathcal{N}$ is a finite-dimensional real vector space and $a_1,\dotsc,a_r\in \mathcal{N}$, $\langle a_1,\dotsc,a_r\rangle$ denotes the convex cone in $\mathcal{N}$ generated by $a_1,\dotsc,a_r$.
    Moreover, for every $a\neq 0$, $a^{\perp}$ is the hyperplane orthogonal to $a$ in the dual vector space $\ma{N}^*$.
    If $\alpha\in\ma{N}^*$ and $\tau:=\langle a\rangle$, we write $\alpha\cdot\tau >0$ (respectively $\alpha\cdot\tau <0$,  $\alpha\cdot\tau =0$) if
     $\alpha\cdot a>0$  (respectively $\alpha\cdot a<0$,  $\alpha\cdot a=0$).
    A facet of a cone is a face of codimension one. If $\ma{C}\subset\ma{N}$ is a convex polyhedral cone, we denote by $\ma{C}^{\vee}\subset\ma{N}^*$ its dual cone.
     \begin{lemma}\label{cones}
Let $\sigma$ be a convex polyhedral cone in a finite dimensional real vector
space $\ma{N}$. Let $\tau$ be a one-dimensional face of $\sigma$, and let $\alpha\in\ma{N}^*$ be such that
$\alpha\cdot\tau<0$ and $\alpha\cdot\tau'\geq 0$ for every one-dimensional face $\tau'\neq\tau$ of $\sigma$.
If $\eta$ is a face of $\sigma$ such that $\eta\subset \ker\alpha$, then $\tau+\eta$ is a face of $\sigma$.
\end{lemma}
\begin{proof}
This is proved in \cite[Lemma 4.7]{fibrations} when $\dim\eta=1$; the same proof applies to the general case.
  \end{proof}

\medskip

The \emph{index} of a locally factorial Fano variety $X$ is the divisibility of $-K_X$ in $\Pic(X)$.

  For the standard terminology and properties in birational geometry and about Mori dream spaces, we refer the reader to \cite{KollarMori,hukeel}, and we recall that smooth Fano varieties are Mori dream spaces by \cite[Cor.~1.3.2]{BCHM}.

 A {\bf contraction} is a projective, surjective morphism with connected fibers, between normal quasi-projective varieties. It can be birational or of fiber type. A contraction $f\colon X\to Y$ is $K$-negative if, for some $m\in\mathbb{N}$,
 $-mK_X$ is Cartier and $f$-ample.
 
 \medskip
 
  Let $X$ be a projective, normal, and $\Q$-factorial Mori dream space.
    As usual we denote by $\N(X)$ the real vector space of numerical equivalence classes of one-cycles in $X$ with real coefficients, and by $\Nu(X)$ the dual vector space of numerical equivalence classes of $\R$-divisors in $X$.
If $D$ is a divisor and $C$ a curve in $X$, we denote by $[D]\in\Nu(X)$ and $[C]\in\N(X)$ their numerical equivalence classes, and we  set $D^{\perp}:=[D]^{\perp}\subset\N(X)$ and $C^{\perp}:=[C]^{\perp}\subset\Nu(X)$.

A movable divisor is an effective divisor $D$ such that the stable
 base locus of the linear system $|D|$  has codimension $\geq 2$. 
We will consider the usual cones of divisors and of curves: $$\Nef(X)\subseteq\Mov(X)\subseteq\Eff(X)\subset\Nu(X)\quad\text{and}\quad
\mov(X)\subseteq \NE(X)\subset\N(X),$$ where 
all the notations are standard except maybe $\mov(X)$, which is the convex cone generated by classes of curves moving in a family covering $X$.
Since $X$ is a Mori dream space, all these cones are  rational polyhedral. An extremal ray $R$ of $\NE(X)$ is a one-dimensional face of this cone. Given  a divisor $D$ on $X$, we say that $R$ is $D$-negative (respectively $D$-positive,  $D$-trivial)  if $D\cdot R<0$ (respectively $D\cdot R>0$,  $D\cdot R=0$). We also write $\Lo(R)$ for the union in $X$ of all curves with class in $R$. 

Given a contraction $f\colon X\to Y$, we set $\NE(f):=(\ker f_*)\cap\NE(X)$, a face of $\NE(X)$ of dimension $\rho_X-\rho_Y$.
The contraction $f$ is {\bf elementary} if $\rho_X-\rho_Y=1$. An elementary contraction can be divisorial, small, or of fiber type.

An elementary contraction $f$ is {\bf of type $(a,b)$} if $\dim\Exc(f)=a$ and $\dim f(\Exc(f))=b$. If $n=\dim X$, we say that $f$ is {\bf of type $(n-1,b)^{\sm}$} if $f$ is the blow-up of a smooth, $b$-dimensional irreducible subvariety $A\subset Y_{\reg}$. Finally when $n=4$ we say that  $f$ is {\bf of type $(3,0)^{Q}$} if $\Exc(f)$ is an irreducible $3$-dimensional quadric with 
$\ma{N}_{\Exc(f)/X}\cong\ol_Q(-1)$.

A  flip is the flip of a small elementary contraction. Given  a divisor $D$ on $X$, a $D$-negative (respectively $D$-positive, $D$-trivial) flip is the flip of a small elementary contraction $f\colon X\to Y$ such that $D\cdot\NE(f)<0$
(respectively $D\cdot\NE(f)>0$, $D\cdot\NE(f)=0$). We do not assume that contractions or flips are $K$-negative, unless specified.

    A {\bf SQM} \emph{(small $\Q$-factorial modification)} is a birational map $\ph\colon X\dasharrow X'$ such that $X'$ is again projective, normal, and $\Q$-factorial, and $\ph$ is an isomophism in codimension one; since $X$ is a Mori dream space, $\ph$ always factors as a finite sequence of flips.

    A {\bf rational contraction} is a rational map $f\colon X\dasharrow Y$ that factors as $X\stackrel{\text{SQM}}{\dasharrow} X'\stackrel{f'}{\to} Y$, where $f'$ is a contraction; $f$ can be birational or of fiber type, and we say that $f$ is non-trivial if $Y\neq\{\pt\}$. We  call $f'$ a {\bf resolution} of $f$. We say that $f$ is elementary if $\rho_X-\rho_Y=1$; if $f$ is elementary and birational, then it can be divisorial or small, depending whether $\Exc(f')$ is a prime divisor or $\codim\Exc(f')>1$.

    There is a fan in $\Nu(X)$, called the {\bf Mori chamber decomposition} and denoted by $\MCD(X)$, supported on $\Mov(X)$, whose cones are in bijection with the set of rational contractions of $X$ (up to isomorphism of the target); the cone corresponding to $f\colon X\dasharrow Y$ is $f^*(\Nef(Y))$.
    \begin{remark}\label{factor}
Let $X$ be a projective, normal, and $\Q$-factorial Mori dream space, and $f\colon X\dasharrow Y$, $g\colon X\dasharrow Z$ two rational contractions. Then there exists a contraction $h\colon Z\to Y$ such that $f=h\circ g$ if and only if $f^*(\Nef(Y))\subset g^*(\Nef(Z))$.
    \end{remark}  
   A {\bf fixed prime divisor} is a prime divisor $D$ which is the stable base locus of the linear system
   $|D|$; then $[D]\in\Nu(X)$ generates a one-dimensional face of $\Eff(X)$.
   \begin{remark}[\cite{eff}, Rem.~2.19]\label{fixedprime}  Let $X$ be a projective, normal, and $\Q$-factorial Mori dream space.  The fixed prime divisors in $X$ are the exceptional divisors of divisorial elementary rational contractions of $X$.
     \end{remark}
   \begin{definition}\label{fmface}
Let $\tau$ be a face of $\Eff(X)$. We say that $\tau$ is a {\bf fixed face} if $\tau\cap\Mov(X)=\{0\}$, otherwise we say that $\tau$ is a {\bf movable face}.
\end{definition}
\begin{definition}\label{padj}
We say that two fixed prime divisors $D,E\subset X$ are {\bf adjacent} if $\langle [D],[E]\rangle$ is a fixed face of $\Eff(X)$.
   \end{definition}  
  
Let $Z\subset X$ be a closed subset. We set
$$\N(Z,X):=\iota_*(\N(Z))\subseteq\N(X),$$
where $\iota\colon Z\hookrightarrow X$ is the inclusion.
\begin{lemma}\label{easy}
Let $f\colon X\to Y$ be a contraction and $Z\subset X$ a closed subset. Then: 
$$\dim\N(Z,X)\leq \rho_X-\rho_Y+\dim\N\bigl(f(Z),Y\bigr).$$
\end{lemma}
\begin{proof}
The pushforward $f_*\colon\N(X)\to\N(Y)$ is a surjective linear map with kernel of dimension $\rho_X-\rho_Y$, and $f_*(\N(Z,X))=\N(f(Z),Y)$.
\end{proof}  
\begin{remark}\label{cassis}
  Let $Z\subset X$ a closed subset and $D\subset X$ a prime divisor such that $Z\cap D=\emptyset$. Then $\N(Z,X)\subset D^{\perp}\subsetneq\N(X)$.

  Indeed for every curve $C\subset Z$ we have $D\cdot C=0$, hence $[C]\in D^{\perp}$.
\end{remark}  
An {\bf ordinary double point}  is an isolated singularity $x_0$ of an $n$-dimensional variety $X$ that is locally analytically isomorphic to the vertex of the cone over a smooth $(n-1)$-dimensional quadric.
\begin{remark}\label{ODP}
  Let $X$ be a quasi-projective $4$-fold with an ordinary double point at $x_0$ and $\Sing(X)=\{x_0\}$. Let $f\colon X'\to X$ be the blow-up of $X$ at $x_0$. Then $X'$ is smooth, $\Exc(f)$ is a smooth quadric, and $f$ is a divisorial, $K$-negative elementary contraction of type $(3,0)^Q$; in particular $x_0$ is a locally factorial, terminal singularity.

  Let $A\subset X$ be a smooth irreducible surface containing $x_0$, and let $g\colon X''\to X$ be the blow-up of $A$. Then $X''$ is smooth and $g$  is a divisorial, $K$-negative elementary contraction of type $(3,2)$; we have $g^{-1}(x_0)\cong\pr^2$ while the other non-trivial fibers of $g$ are $\pr^1$'s.
\end{remark}

By a {\bf family of curves} in a projective variety $X$ we mean an irreducible closed subset $V\subset\Chow(X)$ such that for general $v\in V$, the corresponding cycle $C_v$ of $X$ is an integral curve; then for every $v\in V$ the cycle $C_v$ is a connected curve in $X$.
The numerical equivalence class $[C_v]\in\N(X)$  does not depend on $v\in V$, and we denote it by $[V]$.
The anticanonical degree of the family is $-K_X\cdot [V]$. We say that $V$ is a family of rational curves if for general $v\in V$ the curve $C_v$ is rational. We denote by $\Lo(V)$ the union of the curves $C_v$ for $v\in V$, and we say that $V$ is covering if $\Lo(V)=X$.

\bigskip

We work over the field of complex numbers, but we will occasionally also consider varieties defined over non-closed fields of characteristic zero, as follows. Let $X$ be a complex projective variety and $f\colon X\to Y$ a contraction of fiber type. The {\bf generic fiber} $X_\eta$ of $f$ is the fiber over the generic point $\eta$ of $Y$, and it is a projective variety over the field $K=\C(\eta)=\C(Y)$ of rational functions on $Y$. We will be interested in  the case where $X_{\eta}$ is a smooth del Pezzo surface;
we refer the reader to \cite{hassett} for properties of del Pezzo surfaces over non-closed fields.

\bigskip

\noindent {\bf Preliminaries on  Fano $4$-folds.}
We finally recall some preliminary results on contractions and rational contractions of $4$-folds and Fano $4$-folds.
\begin{lemma}[\cite{AW}, Th.~on p.~256; \cite{AW2}, Prop.~2.1]\label{singtarget}
  Let $X$ be a smooth projective $4$-fold and $f\colon X\to Y$ a $K$-negative elementary contraction of type $(3,2)$. Then $Y$ has at most isolated ordinary double points, in particular $Y$ has locally factorial, terminal singularities, and $S:= f(\Exc(f))$ has isolated singularities. If $y_0\in S$ is a singular point for $Y$ or for $S$, then $\dim f^{-1}(y_0)=2$.
  Moreover $f$ is the blow-up of $Y$ along $S$.\footnote{Namely the blow-up of $Y$ along the ideal sheaf $\ma{I}_{S/Y}$; $S$ is a reduced closed subscheme of $Y$ that may be singular.}
\end{lemma}
\begin{lemma}\label{blowuporder}
  Let $Y$ be a smooth, irreducible, quasi-projective $4$-fold, and $S_1,S_2\subset Y$ two smooth, irreducible surfaces intersecting transversally in one point $y_0\in Y$. Then there is a commutative diagram:
  $$\xymatrix{X\ar[r]^{\tau_2}\ar[d]_{\tau_1}\ar[dr]^f&{X_1}\ar[d]^{\sigma_1}\\
    {X_2}\ar[r]_{\sigma_2}&Y
  }$$
  where $\sigma_i\colon X_i\to Y$ is the blow-up of $S_i$, and $\tau_i\colon X\to X_j$ is the blow-up of the transform $\w{S}_i\subset X_j$ of $S_i$, where $\{i,j\}=\{1,2\}$. The composition $f\colon X\to Y$ is a $K$-negative birational contraction with $f(\Exc(f))=S_1\cup S_2$, $f^{-1}(y)\cong\pr^1$ for every $y\in (S_1\cup S_2)\smallsetminus\{ y_0\}$, and $f^{-1}(y_0)\cong\pr^1\times\pr^1$.
\end{lemma}
\begin{proof}
  Let us consider the blow-up $\sigma_1\colon X_1\to Y$ of $S_1$. The surface $S_2$ is blown-up at $y_0$, and its transform $\w{S}_2\subset X_1$ intersects the exceptional divisor $\Exc(\sigma_1)$ along a smooth irreducible curve $\Gamma$, which is the exceptional curve of the blow-up $\sigma_{1|\w{S}_2}\colon \w{S}_2\to S_2$.
  We also have $\ma{N}_{\w{S}_2/X_1}=(\sigma_{1|\w{S}_2})^*\ma{N}_{S_2/Y}$, therefore
  $(\ma{N}_{\w{S}_2/X_1})_{|\Gamma}\cong\ol_{\pr^1}^{\oplus 2}$.

  Now let $\tau_2\colon X\to X_1$ be the blow-up of $\w{S}_2$, and $f:=\tau_2\circ\sigma_1\colon X\to Y$. We have  $f^{-1}(y)\cong\pr^1$ and $-K_X\cdot f^{-1}(y)=1$ for every $y\in (S_1\cup S_2)\smallsetminus\{y_0\}$.

  Moreover $f^{-1}(y_0)=(\tau_2)^{-1}(\Gamma)=\pr_{\Gamma}(\ma{N}_{\w{S}_2/X_1}^{\vee})_{|\Gamma}\cong\pr^1\times\pr^1$, and $\Exc(\tau_2)_{|f^{-1}(y_0)}\cong\ol_{\pr^1\times\pr^1}(-1,0)$. If $C\subset f^{-1}(y_0)$ is a curve $\{\pt\}\times\pr^1$ with $\tau_2(C)=\Gamma$, then $\Exc(\tau_2)\cdot C=0$ and hence $-K_X\cdot C=-K_Y\cdot \Gamma=1$, so that $-K_{X|f^{-1}(y_0)}\cong\ol_{\pr^1\times\pr^1}(1,1)$. In particular we see that $f$ is $K$-negative, and $\NE(f)=\NE(\tau_2)+\R_{\geq 0}[C]$. It is not difficult to see that $C$ is numerically equivalent to the transform of a general fiber of $\sigma_{1|\Exc(\sigma_1)}$, and that the contraction of the extremal ray $\R_{\geq 0}[C]$ is  $\tau_1\colon X\to X_2$ as in the statement.
\end{proof}
 \begin{lemma}\label{blowupformula}
    Let $X$ be a projective $4$-fold with locally factorial and terminal singularities, and $f\colon X\to Y$ a $K$-negative elementary contraction of type $(3,2)$; set $E:=\Exc(f)$ and $S:=f(E)\subset Y$. Then
     $\chi(X,-K_X)=\chi(Y,-K_Y)-\chi(S,-K_{Y|S})$.
   \end{lemma}
   \begin{proof}
     We note that $Y$ still has locally factorial and terminal singularities.
     
     We have $R^if_*\ol_X(-K_X)=0$ for every $i>0$ \cite[Cor.~2.68]{KollarMori}, thus $\chi(X,-K_X)=\chi(Y,f_*\ol_X(-K_X))$.
     Consider the exact sequence on $X$:
 $$0\la \ol_X(-K_X)\la\ol_X(-K_X+E)\la\ol_X(-K_X+E)\otimes\ol_E \la 0.$$
 We have $-K_X+E=f^*(-K_Y)$, thus using  $f_*\ol_X=\ol_Y$, $R^1f_*\ol_X(-K_X)=0$, and the projection formula, via $f_*$ we get
 $$ 0\la f_*\ol_X(-K_X)\la \ol_Y(-K_Y)\la \ol_Y(-K_Y)\otimes\ol_S\la 0,$$ therefore $\chi(Y,f_*\ol_X(-K_X))=\chi(Y,-K_Y)-\chi(S,-K_{Y|S})$.
\end{proof}
\begin{thm}[\cite{kawasian}, Th.~5.2; \cite{floris}; Th.~1.2]\label{h0}
    Let $X$ be a smooth Fano $4$-fold. Then 
    $h^0(X,-K_X)\geq 2$.
    \end{thm}

 Let $X$ be a normal projective $4$-fold. An {\bf exceptional plane} is a surface $L\subset X_{\reg}$ such that $L\cong\pr^2$ and $\ma{N}_{L/X}\cong\ol_{\pr^2}(-1)^{\oplus 2}$; we denote by $C_L\subset L$ a curve corresponding to a line in $\pr^2$. An {\bf exceptional line} is a smooth rational curve $\ell\subset X_{\reg}$ such that $\ma{N}_{\ell/X}\cong\ol_{\pr^1}(-1)^{\oplus 3}$. Note that $K_X\cdot C_L=-1$ while $K_X\cdot\ell=1$.
\begin{lemma}[\cite{kawsmall}, Th.~1.1]\label{kawamata}
Let $X$ be a smooth projective $4$-fold and $f\colon X\to Y$ a small elementary $K$-negative contraction. Then $\Exc(f)$ is a finite disjoint union of exceptional planes.
  \end{lemma}
  \begin{lemma}[\cite{eff}, Rem.~3.6]\label{SQMFano}
  Let  $X$ be a smooth Fano $4$-fold and $\ph\colon X\dasharrow X'$ a SQM.
  We have the following:
  \begin{enumerate}[$(a)$]
 \item $X'$ is smooth, $X\smallsetminus\dom(\ph)=L_1\cup\cdots\cup L_r$ where $L_i$ are pairwise disjoint exceptional planes,
and
$X'\smallsetminus\dom(\ph^{-1})=\ell_1\cup\cdots\cup\ell_r$
where $\ell_i$ are pairwise disjoint exceptional lines.
\item
  Let  $C\subset X'$ be an irreducible curve, different from $\ell_1,\dotsc,\ell_r$, and intersecting $\ell_1\cup\cdots\cup\ell_r$ in $s\geq 0$ points; then $-K_{X'}\cdot C\geq 1+s$.
  \item If $-K_{X'}\cdot C= 1$, then  $C\cap(\ell_1\cup\cdots\cup\ell_r)=\emptyset$.
\end{enumerate}  
\end{lemma}
\begin{lemma}[\cite{eff}, Rem.~3.7]\label{Kneg}
  Let  $X$ be a smooth Fano $4$-fold and $f\colon X\dasharrow Y$ a rational contraction. Then there exists a resolution  $f'\colon X'\to Y$
of $f$ such that $f'$ is $K$-negative.
\end{lemma}
\begin{lemma}\label{coop}
Let $X$ be a smooth Fano $4$-fold and $g\colon X\to Z$ a contraction of fiber type. Then $\rho_Z\leq \delta_X+2$.
\end{lemma}
\begin{proof}
The proof of \cite[Lemma 2.6]{32} gives the statement. 
\end{proof} 
\section{Preliminaries on rational contractions of fiber type}\label{ratfibertype}
\noindent In this section we give many preliminary results on regular and rational contractions of fiber type of Mori dream spaces, that will be used throughout the paper.
\begin{lemma}[\cite{eff}, Lemma 2.21]\label{general}
      Let $X$ be a projective, normal, and $\Q$-factorial Mori dream space, $f\colon X\to Y$ a contraction of fiber type, and set  $F_y:=f^{-1}(y)$.

      Then there is a non-empty open subset $Y_0\subset Y$ such that the linear subspace $\N(F_y,X)\subset\N(X)$ is constant for $y\in Y_0$.
Moreover, for $y\in Y_0$,
       $\N(F_y,X)^{\perp}\cap\Eff(X)$
 is a face of $\Eff(X)$, has dimension  $\dim(\N(F_y,X)^{\perp})=\rho_X-\dim\N(F_y,X)$, and is the smallest face of $\Eff(X)$ containing $f^*(\Eff(Y))$.
\end{lemma}
\begin{remdefi}[$\boldsymbol{\tau_f,d_f}$]\label{dimfiber}
 Let $X$ be a projective, normal, and $\Q$-factorial Mori dream space, and $f\colon X\dasharrow Y$ a rational contraction of fiber type. We denote by 
 $\tau_f$ the minimal face of $\Eff(X)$ containing $f^*(\Eff(Y))$,\footnote{Equivalently, $\tau_f$ is the minimal face of $\Eff(X)$ containing $f^*(\Nef(Y))$.} and we set:
 $$d_f:=\rho_X-\dim\tau_f.$$
 We will use the following properties:
 \begin{enumerate}[$\bullet$]
   \item
  for every resolution
  $f'\colon X'\to Y$ of $f$, if $F\subset X'$ is a general fiber of $f'$, we have $\dim\N(F,X')=d_f$ (see Lemma \ref{general});
\item $d_f\leq\rho_X-\rho_Y$, in particular $d_f=1$ when $f$ is elementary;
\item $d_f=1$ when $\dim X-\dim Y=1$;
\item $\tau_f\cap\Mov(X)$ is a face of $\Mov(X)$ (because $\Mov(X)\subset\Eff(X)$) and contains $f^*(\Nef(Y))$, therefore $\dim(\tau_f\cap\Mov(X))\geq\rho_Y$, and $\tau_f$ is a movable face of $\Eff(X)$ if $Y\neq \pt$;
\item whenever $f$ is regular, or more generally regular and proper on some non-empty open subset of $X$, if $F\subset X$ is a general fiber, then $\tau_f=\N(F,X)^{\perp}\cap\Eff(X)$.
  Indeed if $\ph\colon X\dasharrow X'$ is a SQM such that $f\circ\ph^{-1}$ is regular, then $F$ must be contained in the open subset where $\ph$ is an isomorphism, hence $\N(F,X)\cong\N(\ph(F),X')$ under the natural isomorphism $\N(X)\cong\N(X')$
  (see \cite[Lemma 2.17]{fibrations}).
 \end{enumerate}
\end{remdefi}
        \begin{lemma}\label{supp}
      Let $X$ be a projective, normal, and $\Q$-factorial Mori dream space, $f\colon X\to Y$ a contraction of fiber type, and $F\subset X$ a general fiber.
Let $W\subset\Nu(X)$ be the linear subspace of classes of $\R$-divisors whose support does not dominate $Y$.  
      Then $W=\N(F,X)^{\perp}$.
\end{lemma}
\begin{proof}
Let $Y_0\subset Y$ be an open subset as in Lemma \ref{general}, so that we can assume that $\N(F,X)=\N(F_y,X)$ for every $y\in Y_0$.
  
   Let $D$ be a divisor whose support does not dominate $Y$. Then there is $y\in Y_0$ such that $(\Supp D)\cap F_y=\emptyset$, hence $D\cdot C=0$ for every curve $C\subset F_y$, and $[D]\in \N(F_y,X)^{\perp}$. This shows that $W\subseteq \N(F,X)^{\perp}$.

  By Lemma \ref{general} $\N(F,X)^{\perp}\cap \Eff(X)$ is a face of $\Eff(X)$ of dimension equal to $\dim(\N(F,X)^{\perp})$, therefore  $\N(F,X)^{\perp}$ is generated by the one-dimensional faces of $\Eff(X)$ contained in it.
  On the other hand every one-dimensional face of $\Eff(X)$ is generated by the class of some prime divisor, and we conclude that $\N(F,X)^{\perp}$ can be generated by classes of prime divisors.

  Now let $D\subset X$ be a prime divisor such that $[D]\in\N(F,X)^{\perp}$. Then $D$ cannot dominate $Y$, otherwise we could choose $y\in Y_0$ such that $F_y\not\subset D$ and $F_y\cap D\neq\emptyset$. Thus we could find a curve $C\subset F_y$ such that $D\cdot C>0$, a contradiction. Therefore $[D]\in W$, and we conclude that $\N(F,X)^{\perp}\subseteq W$ and finally $\N(F,X)^{\perp}= W$.
\end{proof}
\begin{remark}\label{tauf}
  In the setting of Def.-Rem.~\ref{dimfiber}, it follows from Lemmas \ref{general} and \ref{supp} that
$$\tau_f=\langle [D]\,|\,\text{$D$ is effective and $\Supp D$  does not dominate $Y$}\rangle.$$
\end{remark}
\begin{lemma}\label{genrho}
Let $X$ be a projective, normal, and $\Q$-factorial Mori dream space, $f\colon X\to Y$ a contraction of fiber type, and $F\subset X$ a general fiber. Let also $\eta$ be the generic point of $Y$ and $X_{\eta}$ the generic fiber of $f$. Then $\rho_{X_\eta}=\dim\N(F,X)=d_f$.
\end{lemma}
\begin{proof}
Recall that since $X$ is a Mori dream space, 
$\Pic(X)$ is finitely generated, and $\Nu(X)\cong \Pic(X)\otimes\R$.

  Consider the restriction $r\colon \Pic(X)\to\Pic(X_{\eta})$ and the induced homomorphism $r_{\R}\colon \Pic(X)\otimes\R\to\Pic(X_{\eta})\otimes\R$. 
  Since $X$ is $\Q$-factorial, $r_{\R}$ is surjective. Moreover $\ker r_{\R}$ coincides with the  linear subspace $W\subset\Nu(X)$ of classes of $\R$-divisors whose support does not dominate $Y$.  Therefore by Lemma \ref{supp} we have $\dim\ker r_{\R}=\rho_X-\dim\N(F,X)$, which yields the statement.
\end{proof}
\begin{lemma}\label{spread}
  Let $X$ be a normal quasi-projective variety
  and
  $f\colon X\to Y$ a contraction of fiber type. Let  $\eta$ be the generic point of $Y$, $X_{\eta}$ the generic fiber of $f$, $Z_{\eta}$ a normal projective variety over $K=\C(\eta)$, and $g_\eta\colon X_{\eta}\to Z_{\eta}$ a projective morphism  over $K$ such that $(g_\eta)_*\ol_{X_{\eta}}=\ol_{Z_\eta}$. 

  Then there exists   an open subset $Y_0$ of $\,Y$ such that $f_{|f^{-1}(Y_0)}\colon f^{-1}(Y_0)\to Y_0$ factors as the composition of two contractions  $f^{-1}(Y_0)\to Z_0\to Y_0$ that extend $X_\eta\to Z_\eta\to\eta$. 
\end{lemma}
\begin{proof}
  We  apply the results on spreading out schemes and morphisms. By \cite[Th.~3.2.1(i) and (ii)]{poonen}, there exist an open subset $Y_0$ of $Y$ and a quasi-projective variety $Z_0$, with a  surjective projective morphism $h\colon Z_0\to Y_0$, that extend $Z_\eta\to\eta$.

  Consider the 
  normalization $\nu\colon Z_0^{\nu}\to Z_0$; restricting to the generic fibers 
we get $\nu_{\eta}\colon (Z_0^{\nu})_{\eta}\to Z_\eta$ which is the normalization of $Z_\eta$. Since $Z_{\eta}$ is normal, $\nu$ is an isomorphism on the generic fibers, and up to composing with $\nu$ we can assume that $Z_0$ is normal. 

   Note that  $(g_\eta)_*\ol_{X_{\eta}}=\ol_{Z_\eta}$ implies that every fiber of $g_{\eta}$ is geometrically connected \cite[Ch.~5, Th.~3.15 and Cor.~3.17]{liu}.
   Then by \cite[Th.~3.2.1(iii) and (iv)]{poonen},  up to shrinking $Y_0$, we can assume that there exists a projective morphism $g\colon f^{-1}(Y_0)\to Z_0$, with connected fibers, extending  $g_\eta$, so that $f_{|f^{-1}(Y_0)}=h\circ g$.
\end{proof}  
\begin{lemma}\label{elem}
Let $X$ be a projective, normal, and $\Q$-factorial Mori dream space. Let $X_0\subset X$ be an open subset and $f_0\colon X_0\to Y_0$ a contraction. Then there exists a rational contraction $f\colon X\dasharrow Y$ such that $Y$ contains $Y_0$ as an open subset and $f_{|X_0}=f_0$.
\end{lemma}
\begin{proof}
  Let $Y_0\subset \pr^N$ be an embedding as locally closed subset. Then $f_0$ gives a rational map $F_0\colon X\dasharrow\pr^N$ with $B:=X\smallsetminus\dom(F_0)$ closed subset of codimension $\geq 2$.

  Let $H\subset\pr^N$ be a general hyperplane and consider the divisor $F_0^*(H)$ on $X\smallsetminus B$. 
  Then $F_0^*(H)$ extends uniquely to a Weil divisor on $X$, that we still denote by $F_0^*(H)$. Since $X$ is $\Q$-factorial, there exists $m\in\Z_{>0}$ such that $D:=mF_0^*(H)$ is Cartier; then $D$ is also movable, because its base locus is contained in $B$. Moreover, since $X$ is a Mori dream space, up to taking a possibly larger $m$ we can assume that the complete linear system $|D|$ defines a rational map $F\colon X\dasharrow\pr^M=\pr(H^0(X,D))$ such that, if $Y:=\overline{F(\dom(F))}\subset\pr^M$,  the induced map $f\colon X\dasharrow Y$ is a rational contraction. Then $Y$ contains $Y_0$ as an open subset and $f_{|X_0}=f_0$.
\end{proof}
 We will work with  two particular types of rational contractions of fiber type, namely ``quasi-elementary'' and ``special'' contractions; let us recall the definition and a few properties, and refer the reader to \cite[\S 2.2]{eff} and \cite[\S 2]{fibrations} respectively for more details.
\begin{definition}\label{defqes}
  Let $f\colon X\to Y$ be a contraction of fiber type.
  
  We say that $f$ is {\bf quasi-elementary} if $\dim\N(F,X)=\rho_X-\rho_Y$ for every fiber $F\subset X$ of $f$, namely if
  $d_f=\rho_X-\rho_Y$.

We say that $f$ is {\bf special}  if $Y$ is $\Q$-factorial and $\codim f(D)\leq 1$ for every prime divisor $D\subset X$.
\end{definition}
An elementary contraction of fiber type is always quasi-elementary.
A quasi-elementary contraction is always special, by the following.
\begin{lemma}[\cite{fibrations}, Prop.~2.2]\label{equivalent}
 Let $X$ be a projective, normal, and $\Q$-factorial Mori dream space and $f\colon X\to Y$ a contraction of fiber type. Then $f$ is quasi-elementary if and only if $Y$ is $\Q$-factorial and $f^*(B)$ is an irreducible (possibly non-reduced) divisor for every prime divisor $B\subset Y$. 
\end{lemma}

If $Y$ is a curve, then every contraction of fiber type $f\colon X\to Y$ is special. If $Y$ is a surface, then $f$ is special if and only if it is {\bf equidimensional} \cite[Lemma 2.7]{fibrations}.

\medskip

Consider now a {\bf rational} contraction of fiber type $f\colon X\dasharrow Y$. We say that $f$ is quasi-elementary, respectively special, if a resolution
of $f$ is quasi-elementary, respectively special; this does not depend on the choice of the resolution.
\begin{lemma}\label{eubea}
  Let $X$ be a projective, normal, and $\Q$-factorial Mori dream space and $f\colon X\dasharrow Y$ a rational contraction. The following are equivalent:
  \begin{enumerate}[$(i)$]
  \item  $f$ is of fiber type and special;
 \item $f$ is of fiber type and   $\dim(\tau_f\cap\Mov(X))=\rho_Y$ (see Def.-Rem.~\ref{dimfiber});
  \item there exists a proper face $\eta_f$ of $\Mov(X)$ such that 
    $f^*(\Nef(Y))\subset\eta_f$, $\dim\eta_f=\rho_Y$, and $\eta_f=\tau\cap \Mov(X)$ for some face $\tau$ of $\Eff(X)$.
  \end{enumerate}
  
  If these conditions hold and moreover $f$ is regular with general fiber $F\subset X$, then $\eta_f=(\N(F,X))^{\perp}\cap\Mov(X)$.
\end{lemma}
\begin{proof}
The equivalence of $(i)$ and $(ii)$  is \cite[Lemma 5.3]{small}.
  
  Assume  $(ii)$. Note that $\eta_f:=\tau_f\cap\Mov(X)$ is a face of $\Mov(X)$ because $\Mov(X)\subset\Eff(X)$, it is proper because $\tau_f\subsetneq\Eff(X)$, it contains $f^*(\Nef(Y))$, and $\dim\eta_f=\rho_Y$, hence we get $(iii)$.

  Conversely, assume $(iii)$. Since $\eta_f\subsetneq\Mov(X)$, we have $\tau\subsetneq\Eff(X)$, hence $f^*(\Nef(Y))\subset\partial\Eff(X)$, and $f$ is of fiber type.
  By definition $\tau_f$ is the minimal face of $\Eff(X)$ containing $f^*(\Nef(Y))$, therefore $\tau_f\subset\tau$, and $\tau_f\cap \Mov(X)=\tau\cap \Mov(X)=\eta_f$ has dimension $\rho_Y$, and we get $(ii)$.

  Finally, when $f$ is regular, $\eta_f=\tau_f\cap \Mov(X)=(\N(F,X))^{\perp}\cap\Mov(X)$ by Lemma \ref{general}.
\end{proof}
\begin{remark}\label{zumba}
  Let $X$ be a projective, normal, and $\Q$-factorial Mori dream space, and  $\tau$  a proper, movable face of $\Eff(X)$. Then there exists a special rational contraction of fiber type $f\colon X\dasharrow Y$ with $\eta_f=\tau\cap\Mov(X)$ and $\tau_f\subset\tau$, hence $\rho_Y=\dim(\tau\cap\Mov(X))$ and $d_f=\rho_X-\dim\tau_f\geq\rho_X-\dim\tau$ (notation as in Lemma \ref{eubea} and Def.-Rem.~\ref{dimfiber}).

  Indeed it is enough to choose a cone $\eta_0\in\MCD(X)$ such that $\eta_0\subset
  \tau\cap\Mov(X)$ and $\dim\eta_0=\dim(\tau\cap\Mov(X))$; then the rational contraction $f\colon X\dasharrow Y$ such that $f^*(\Nef(Y))=\eta_0$  
 has the desired properties, by Lemma \ref{eubea}.
\end{remark}  
\begin{proposition}[\cite{fibrations}, Prop.~2.13]\label{calanques}
 Let $X$ be a projective, normal, and $\Q$-factorial Mori dream space and $f\colon X\dasharrow Y$ a rational contraction of fiber type. 
Then $f$ can be factored as $X\stackrel{g}{\dasharrow} Z\stackrel{h}{\to}Y$ where  $h$ is birational and $g$ is a special rational contraction of fiber type.
\end{proposition}  
\begin{lemma}\label{properfiber}
  Let $X$ be a projective, normal, and $\Q$-factorial Mori dream space and $f\colon X\dasharrow Y$ a rational contraction of fiber type  that
 is regular and proper on some non-empty open subset of $X$.

  Then $f$ can be factored as $X\stackrel{g}{\dasharrow} Z\stackrel{h}{\to}Y$ where $h$ is birational and $g$ is a special rational contraction of fiber type 
  that is regular and proper on some non-empty open subset of $X$. In particular $f$ and $g$ have isomorphic general fibers.
\end{lemma}  
\begin{proof}
This follows from the same proof as \cite[Prop.~2.13]{fibrations}.
\end{proof}  
\begin{lemma}\label{induced}
  Let $X$ be a projective, normal, and $\Q$-factorial Mori dream space and $f\colon X\to Y$ a contraction of fiber type. Suppose that there exist a non-empty open subset $Y_0\subset Y$ and a contraction of fiber type $g_0\colon f^{-1}(Y_0)\to Z_0$ such that $f_{|f^{-1}(Y_0)}$ factors through $g_0$:
$$\xymatrix{{f^{-1}(Y_0)}\ar[r]_{\quad g_0}\ar@/^1pc/[rr]^{f}&{Z_0}\ar[r]&{Y_0}
    }$$
 Then $f$ can be factored as $X\stackrel{g}{\dasharrow} Z\stackrel{h}{\to}Y$, where $h$ is a contraction, and $g$ is 
 a special rational contraction  that coincides with $g_0$ on
$g_0^{-1}(U)$ for
 some non-empty open subset $U$ of $Z_0$.
\end{lemma}
\begin{proof}
  By Lemma \ref{elem} there exists a rational contraction $g'\colon X\dasharrow Z'$ that restricts to $g_0$ on $f^{-1}(Y_0)$, in particular $g'$ is of fiber type and  is regular and proper on $f^{-1}(Y_0)$.

  Now we apply Lemma \ref{properfiber} to $g'$, and deduce that
  $g'$ factors as  $X\stackrel{g''}{\dasharrow}Z''\stackrel{\alpha}{\to} Z'$ where
 $\alpha$ is birational and $g''$ is a special rational contraction of fiber type that is regular and  proper on some non-empty open subset of $X$. In particular $g''$ has the same general fiber as $g_0$.
 
Let us consider $F_f,F_g\subset X$ general fibers of $f$ and $g_0$ respectively; we have $F_g\subset F_f$, so $\N(F_g,X)\subset\N(F_f,X)$, and hence
$\N(F_f,X)^{\perp}\subset\N(F_g,X)^{\perp}\subset\Nu(X)$.

 Let us set:
 $$\eta_{g''}:=\N(F_g,X)^{\perp}\cap\Mov(X)=\tau_{g''}\cap\Mov(X)$$
(see Def.-Rem.~\ref{dimfiber}), so that 
 $\eta_{g''}$ is a face of $\Mov(X)$, and it has dimension
 $\rho_{Z''}$ because $g''$ is special, by Lemma \ref{eubea}.

We note that $\eta_{g''}$ is a union of cones of $\MCD(X)$ of dimension $\rho_{Z''}$, one being $(g'')^*(\Nef(Z''))$.

The cone $f^*(\Nef(Y))$ belongs to the fan $\MCD(X)$ and is contained both in $\Mov(X)$ and in $\N(F_f,X)^{\perp}$, therefore 
in $\eta_{g''}$. Hence there exists a cone $\eta\in\MCD(X)$ of dimension $\rho_{Z''}$ such that $f^*(\Nef(Y))\subseteq\eta\subseteq \eta_{g''}$.

Let $g\colon X\dasharrow Z$ be the rational contraction of fiber type such that $\eta=g^*(\Nef(Z))$. Then $\rho_Z=\rho_{Z''}$ and
there is a SQM $\psi\colon Z''\dasharrow Z$ such that $g=\psi\circ g''$
 (see \cite[Th.~1.2]{okawa_MCD}),
in particular $g$ and $g''$ have the same general fiber $F_g$.
 Then $g$ is special again by Lemma \ref{eubea}, and it
 coincides with $g_0$  on
$g_0^{-1}(U)$ for
 some non-empty open subset $U$ of $Z_0$.

Finally $f^*(\Nef(Y))\subset g^*(\Nef(Z))$, therefore there exists a contraction $h\colon Z\to Y$ such that $f=h\circ g$ (see Rem.~\ref{factor}).
\end{proof}  
\begin{lemma}\label{h}
  Let $X$ be a projective, normal, and $\Q$-factorial Mori dream space, and $f\colon X\dasharrow Y$ a rational contraction of fiber type that factors as $X\stackrel{g}{\dasharrow}Z\stackrel{h}{\dasharrow}Y$ where  $g$ and $h$ are rational contractions, $h$ is of fiber type, and $Z$ is $\Q$-factorial.

  If $f$ is special (respectively, quasi-elementary), then $h$ is special (respectively, quasi-elementary).
\end{lemma}
\begin{proof}
Up to composing with SQM's, we can assume that $f$, $g$, and $h$ are regular.
  
Suppose that $f$ is special; in particular $Y$ is $\Q$-factorial. Let $D\subset Z$ be a prime divisor such that $h(D)\subsetneq Y$.
Then $g^{-1}(D)$ has pure codimension one in $X$ and $g(g^{-1}(D))=D$, so there is an irreducible component 
 $D_X$ of $g^{-1}(D)$ such that $g(D_X)=D$. Since $f$ is special, $h(D)=f(D_X)$ must be a prime divisor in $Y$, so that $h$ is special.

Suppose now that $f$ is quasi-elementary. We use Lemma \ref{equivalent}:
first of all $Y$ is $\Q$-factorial. Let $B\subset Y$ be a prime divisor. Then $f^*(B)=g^*(h^*(B))$ is irreducible (possibly non-reduced), therefore $h^*(B)$ must have a unique irreducible component, and $h$ is quasi-elementary too.
\end{proof}
\begin{lemma}\label{montenero}
  Let $X$ be a projective, normal, and $\Q$-factorial Mori dream space, and $f\colon X\dasharrow Y$ a  rational contraction of fiber type. We have the following:
  \begin{enumerate}[$(a)$]
    \item
  $f$ factors as $X \stackrel{g}{\dasharrow} Z\stackrel{h}{\to} Y$
  where $h$ is a contraction and $g$ is an elementary rational contraction, either divisorial or of fiber type;
\item if $f$ is special, then $h$ is of fiber type or an isomorphism;
\item  
 when $g$ is divisorial:  
\begin{enumerate}[$\bullet$]
  \item if $f$ is special, then $\Exc(g)$ dominates $Y$ or a prime divisor in $Y$;
\item
if $f$ is quasi-elementary, then $\Exc(g)$ dominates $Y$.
\end{enumerate}
\end{enumerate}
\end{lemma}
\begin{proof}
Consider the cone $f^*(\Nef(Y))\in\MCD(X)$.
Since $f$ is of fiber type, $f^*(\Nef(Y))$ is contained in the boundary of $\Eff(X)$, hence also in the boundary of $\Mov(X)$.

Let us 
  consider a facet $\eta$ of $\Mov(X)$ containing $f^*(\Nef(Y))$,
and
$\eta_0\in\MCD(X)$ of dimension $\rho_X-1$ such that $f^*(\Nef(Y))\subset\eta_0\subset\eta$.
Then $\eta_0=g^*(\Nef(Z))$ for an elementary rational contraction $g\colon X\dasharrow Z$ and we have a factorization:
$$\xymatrix{X\ar@{-->}[r]_{g}\ar@{-->}@/^1pc/[rr]^{f}&{Z}\ar[r]_{h}&Y
    }$$
    where $h$ is a contraction (see Rem.~\ref{factor}).
By construction $\eta_0$ is contained in the boundary of $\Mov(X)$, therefore $g$ cannot be small. This gives $(a)$, and $(b)$ follows from \cite[Prop.~2.10]{fibrations}.

For $(c)$, suppose that $g$ is divisorial. The first statement follows from the definition of special contraction. For the second statement,
up to composing with a SQM we can assume that $g$ is regular; set $E:=\Exc(g)\subset X$, and let $F\subset  X$ be a general fiber of $f$.
If $f(E)\subsetneq Y$, then $F\cap E=\emptyset$, thus $\N(F,X)\subset E^{\perp}$ (see Rem.~\ref{cassis}). We also have $E\cdot \NE(g)<0$, thus $\NE(g) \not\subset E^{\perp}$, in particular $\NE(g)\not\subset\N(F,X)$. On the other hand $\NE(g)\subset\ker f_*$, and we conclude that $\N(F,X)\subsetneq\ker f_*$ and $f$ cannot be quasi-elementary.
\end{proof}
\begin{remark}\label{choices}
In the setting of Lemma \ref{montenero}, it follows from the proof
  that the choices for $g$ correspond to the facets of $\Mov(X)$ that contain $f^*(\Nef(Y))$. More precisely, given a facet $\eta$ of $\Mov(X)$ containing 
$f^*(\Nef(Y))$, we costruct $g\colon X\dasharrow Z$ as in the statement, with moreover $\eta=\Mov(X)\cap g^*(\Nu(Z))$. Then $g$ is of fiber type if and only if $\eta$ is contained in the boundary of $\Eff(X)$, otherwise $g$ is divisorial.
\end{remark}
\begin{lemma}\label{quasiel2}
  Let $X$ be a projective, normal, and $\Q$-factorial Mori dream space of dimension $n$, and $f\colon X\dasharrow Y$ a quasi-elementary rational contraction
  with $\dim Y=n-2$ and $\rho_X-\rho_Y=2$.
  Then one of the following holds:
  \begin{enumerate}[$(i)$]
    \item
      $f$ factors as $X\stackrel{g}{\dasharrow}Z\stackrel{h}{\to} Y$ where $\dim Z=n-1$ and $g$ and $h$ are elementary of fiber type;
    \item
      $f$ admits two factorizations 
 $X\stackrel{\alpha_i}{\dasharrow}W_i\stackrel{h_i}{\to}Y$ for $i=1,2$,
 where $\alpha_i$ is elementary and divisorial with exceptional locus dominating $Y$, $h_i$ is elementary of fiber type, and $\alpha_1^*(\Nu(W_1))\neq\alpha_2^*(\Nu(W_2))$.
 \end{enumerate}
\end{lemma}
\begin{proof}
 We note that since $f$ is quasi-elementary, the cone  $f^*(\Nef(Y))$ has dimension $\rho_X-2$ and is contained in a face of $\Mov(X)$ of dimension $\rho_X-2$ (see  Lemma \ref{eubea}). Therefore
$f^*(\Nef(Y))$  is contained 
in exactly two facets $\tau_1$ and $\tau_2$ of 
$\Mov(X)$.

If one of these facets is contained in $\partial\Eff(X)$, then by applying Lemma \ref{montenero} and Rem.~\ref{choices} we factor $f$ as in $(i)$.
Otherwise, if
 no $\tau_i$ is contained in 
$\partial\Eff(X)$, for $i=1,2$ again by applying Lemma \ref{montenero} and Rem.~\ref{choices}
 we get a factorization of $f$ as
 $$X\stackrel{\alpha_i}{\dasharrow}W_i\stackrel{h_i}{\to}Y,$$
where $\alpha_{i}$ is a divisorial elementary rational contraction 
with exceptional locus dominating $Y$, and $h_{i}$ is elementary.
Moreover  $\alpha_1^*(\Nu(W_1))\neq\alpha_2^*(\Nu(W_2))$ because
 $\tau_1\neq\tau_2$.
\end{proof}
We conclude this section by recalling two results on special or quasi-elementary rational contractions of Fano $4$-folds, that are needed in the sequel. 
\begin{lemma}[\cite{fibrations}, Lemma 4.3]\label{specialsurf}
 Let  $X$ be a smooth Fano $4$-fold and $f\colon X\dasharrow S$ a special rational contraction onto a surface. Then $S$ is smooth.
\end{lemma}
\begin{lemma}[\cite{small}, Lemma 5.9]\label{delpezzo}
 Let  $X$ be a smooth Fano $4$-fold with $\rho_X\geq 7$ and $f\colon X\dasharrow S$ a quasi-elementary rational contraction onto a surface. Then $S$ is a smooth del Pezzo surface.
\end{lemma} 
\section{Factoring rational contractions of fiber type of Fano $4$-folds through $3$-folds}\label{secbrasile}
\noindent
In this section we consider rational contractions of fiber type $f\colon X\dasharrow Y$ of Fano $4$-folds.
When $\dim Y=3$, the results in \cite{3folds} give a good understanding of $X$ and $f$, see Th.~\ref{CS}.
We consider here the case
$\dim Y\in\{1,2\}$. We show that when $d_f\geq 5$, then $f$ always factors through a $3$-fold (Th.~\ref{brasileintro} from the Introduction, see Th.~\ref{brasile} below).
 This has in turn important applications on the geometry of Fano $4$-folds with $\rho_X\geq 7$, given in Th.~\ref{scala} and Cor.~\ref{summary}.
  \begin{thm}\label{brasile}
  Let $X$ be a smooth Fano $4$-fold and $f\colon X\dasharrow Y$ a non-trivial rational contraction of fiber type.
\begin{enumerate}[$(a)$]
\item
  If $d_f\geq 5$, then
 $f$ can be factored as $X\stackrel{g}{\dasharrow} Z\stackrel{h}{\to} Y$ where $\dim Z=3$, $g$ is a special rational contraction, and $h$ is a contraction of fiber type.
\item If $d_f=4$ and $\dim Y=2$, then either the statement in $(a)$ holds, or 
  $f$ is regular with general fiber a del Pezzo surface of degree one.
  \end{enumerate}
\end{thm}
As explained in the Introduction, for the proof of Th.~\ref{brasile} we will use different strategies depending on the dimension of the base $Y$.
We recall that $d_f=1$ when $\dim Y=3$, therefore $d_f\geq 4$ implies that either $\dim Y=2$ or $Y\cong\pr^1$.
For the case $\dim Y=2$, we will consider the generic fiber of a $K$-negative resolution $f'$ of $f$, and rely on Lemma \ref{cabofrio}; for the case $Y\cong\pr^1$, we will use Prop.~\ref{P1}, that exploits the monodromy action on the nef cone of a general fiber of $f'$.
  \begin{lemma}\label{cabofrio}
    Let $K$ be a field of characteristic zero and let $S$ be a smooth, projective del Pezzo surface over $K$.
      If $\rho_S\geq 4$, then one of the following holds:
      \begin{enumerate}[$(i)$]
        \item
          there exists a surjective morphism $S\to C$ onto a curve;
        \item $\rho_S=4$ and $K_S^2=1$.
          \end{enumerate}
      \end{lemma}
      \begin{proof}
        We assume that $S$ has no surjective morphism to a curve, and show $(ii)$.
        
        There exists a birational map $S\to S_0$ where $S_0$ is a minimal del Pezzo surface over $K$. By our assumptions $S_0$ does not have a conic bundle structure, hence $\rho_{S_0}=1$ (see  \cite[Th.~3.9]{hassett}). Among the (finitely many) possible  birational maps $S\to S_0$ with $\rho_{S_0}=1$, let us choose one with $K_{S_0}^2$ maximal.  Note that $K_{S_0}^2\neq 7$ (see [\emph{ibid.}, Ex.~3.1.2]).

        The map $S\to S_0$ is the blow-up of $r$ distinct (closed) points, of degrees $d_1,\dotsc,d_r$;\footnote{The degree of a closed point $p$ is dimension of its residue field $K(p)$ over the base field $K$.} we have
        \stepcounter{thm}
        \begin{equation}\label{fabri}
        4\leq \rho_S=1+r\quad\text{ and }\quad 
      r\leq \sum_{i=1}^rd_i  =K_{S_0}^2-K_S^2\leq K_{S_0}^2-1,\end{equation}
in particular $r\geq 3$ and $K_{S_0}^2\geq 4$.
        
        Let $i\in\{1,\dotsc,r\}$ and      let us consider $S_1:=\Bl_{p_i}S_0$. Then $S_1$ is a del Pezzo surface with $\rho_{S_1}=2$, and $S_1$ does not have a morphism onto a curve, therefore  there is another birational map $S_1\to S_0'$ which blows-up  a point of degree $d_i'$. We have $K_{S_1}^2=K_{S_0}^2-d_i=K_{S_0'}^2-d_i'$, and $K_{S_0}^2\geq K_{S_0'}^2$
    by the maximality of $K_{S_0}^2$, hence $d_i\geq d_i'$.

  The birational map $$S_0\longleftarrow S_1\la S_0'$$ is called an elementary link of type II, see \cite[(2.2.2)]{isk96}, and these links are classified in [\emph{ibid.}, Th.~2.6(ii)].
       It follows from this classification and from $d_i\geq d_i'$
       that $d_i\geq 2$ if      $K_{S_0}^2\in\{4,6,9\}$, and   $d_i\geq 3$ if      $K_{S_0}^2\in\{5,8\}$.  By \eqref{fabri} this implies that  $K_{S_0}^2=9$, therefore  $(S_0)_{\bar K}\cong\pr^2_{\bar K}$, where $\bar K$ is the algebraic closure of $K$ and $(S_0)_{\bar K}=S_0\times_{\Spec K}\Spec {\bar K}$. Let $\pi\colon\pr^2_{\bar K}\to S_0$ be the projection. 

       We show that there is at most one index $i\in\{1,\dotsc,r\}$ such that $d_i=2$.       
       Assume that this is not the case, for instance that $d_1=d_2=2$. This means that, for $i=1,2$,
       $\pi^{-1}(p_i)=\{p_{i1},p_{i2}\}\subset \pr^2_{\bar K}$ is an orbit for the action of the Galois group $\text{Gal}(\bar K/K)$. Then the line $\ell_i:=\overline{p_{i1}p_{i2}}\subset \pr^2_{\bar K}$ is fixed by the action of the Galois group, hence $\ell_i=\pi^*(\ell_{i,K})$ where $\ell_{i,K}\subset S_0$ is a curve, and $H^0(S_0,\ell_{1,K}+\ell_{2,K})\cong H^0(\pr^2_{\bar K},\ol_{\pr^2_{\bar K}}(2))$. 
   In the linear system $|\ell_{1,K}+\ell_{2,K}|$, the pencil of curves     
     containing $p_1$ and $p_2$ defines a map $\Bl_{p_1,p_2}S_0\to\pr^1_K$, against our assumptions.

     Therefore  there is at most one  $i$ such that $d_i=2$, and $d_j\geq 3$ for every $j\neq i$.       
Again by \eqref{fabri} we get $8\leq 2+3(r-1)\leq\sum_id_i=9-K_S^2\leq 8$, and we conclude that $r=3$, $\rho_S=4$, $(d_1,d_2,d_3)=(2,3,3)$ up to order, and $K_S^2=1$.
        \end{proof}
        \begin{proof}[Proof of Th.~\ref{brasile} in the case $\dim Y=2$]
  Let $\ph\colon X\dasharrow X'$ be a SQM such that $f':=f\circ\ph^{-1}\colon X'\to Y$ is a $K$-negative resolution of $f$, and $F\subset X'$ a general fiber of $f'$.
    Let also $\eta$ be the generic point of $Y$ and $X_\eta$ the generic fiber of $f'$; then $X_\eta$ is a smooth del Pezzo surface over $K=\C(\eta)$, with $K_{X_{\eta}}^2=K_F^2$ (by generic flatness) and $\rho_{X_{\eta}}=
  \dim\N(F,X')=d_f$ (by Lemma \ref{genrho}).

  If $d_f=\rho_{X_{\eta}}\geq 4$, then
  we can apply Lemma \ref{cabofrio}. If we are in case $(i)$,
 there is a surjective morphism $g_\eta\colon X_\eta\to C_{\eta}$ onto a curve, defined over $K$. By normalizing and taking the Stein factorization, we can assume that $C_{\eta}$ is smooth and $(g_\eta)_*\ol_{X_\eta}=\ol_{C_\eta}$.

  Then by Lemma \ref{spread} there are a non-empty open subset $Y_0\subset Y$,
  a quasi-projective $3$-fold $Z_0$ with a contraction $Z_0\to Y_0$ extending $C_\eta\to\eta$, and a contraction $g_0\colon (f')^{-1}(Y_0)\to Z_0$, such that $f'_{|(f')^{-1}(Y_0)}$ factors through $g_0$.
  Finally using Lemma \ref{induced} we get $(a)$.

  Suppose now that we are in case $(ii)$ of Lemma \ref{cabofrio}, namely $d_f=\rho_{X_{\eta}}= 4$ and $K_{ X_{\eta}}^2=1$. Then  $K_F^2=1$ and $\ph$ is an isomorphism by \cite[Lemma 5.10]{small}, so $f$ is regular and we get $(b)$.
\end{proof} 
\begin{proposition}\label{P1}
  Let $X$ be a smooth Fano $4$-fold, $X\dasharrow X'$ a SQM, and $f\colon X'\to\pr^1$ a $K$-negative contraction. Let $\iota\colon F\hookrightarrow X'$ be a general fiber of $f$.

  If $\rho_F\geq 6$, or if $\rho_F=4,5$ and $\iota_*\colon\N(F)\to\N(X')$ is injective, 
  then $f$ can be factored as $X\stackrel{g}{\dasharrow} Z\stackrel{h}{\to} \pr^1$ where $\dim Z=3$, $g$ is a special rational contraction, and $h$ is a contraction.
 \end{proposition}  
 \begin{proof}
   Let $Y_0\subseteq\pr^1$ be an open subset over which $f$ is smooth, and set $X_0:=f^{-1}(Y_0)$ and  $F:=f^{-1}(y)$
for $y\in Y_0$.  Note that $F$ is a smooth Fano $3$-fold. We consider the monodromy action of $\pi_1(Y_0,y)$ on $H^2(F,\R)=\Nu(F)$;  the image of the restriction $\Nu(X_0)\to\Nu(F)$ is the invariant subspace $\Nu(F)^{\pi_1(Y_0,y)}$, see
\cite[Th.~16.24]{voisin}.

\medskip

   Assume that
   $\rho_F\geq 6$. Then $F\cong\pr^1\times S$ where $S$ is a del Pezzo surface \cite[\S 12.6]{fanoEMS}; consider the projections $\pi_1\colon F\to\pr^1$ and $\pi_2\colon 
   F\to S$. 

   Set $\tau:=\pi_1^*\Nef(\pr^1)$ and $\sigma:=\pi_2^*\Nef(S)$. Then
$\sigma$ is a facet of $\Nef(F)$,
   $\tau$ is a one-dimensional face of $\Nef(F)$, and every other one-dimensional face of $\Nef(F)$  is contained in $\sigma$, because it corresponds to a contraction of $F$ that factors through $\pi_2$.

Let  $g\in\pi_1(Y_0,y)$.
By \cite{wisndef,wisnrigidity}, we have $g(\Nef(F))=\Nef(F)$, so that $g$ permutes the facets of $\Nef(F)$; see \cite[\S 2]{CFST} for an analysis of the monodromy action on the nef cone of $F$. In particular
$g(\sigma)=\sigma$ by \cite[Th.~2.7]{CFST}, because $\pi_2$ is the unique elementary contraction of fiber type of $F$, the other elementary contractions of $F$ being birational. This also implies that $g(\tau)=\tau$; in particular $g$ preserves the linear subspaces $\pi_{1\,}^*\Nu(\pr^1)$ and $\pi_{2\,}^*\Nu(S)$ of $\Nu(F)$.

We have $K_F=\pi_1^*K_{\pr^1}+\pi_2^*K_S$, and $g(K_F)=K_F$ because $K_F=K_{X_0|F}$. Therefore $g$ must fix also $\pi_2^*K_S$, so that $\pi_2^*(-K_S)\in\Nu(F)^{\pi_1(Y_0,y)}$. We conclude that there exists $H\in\Pic(X_0)$ such that $H_{|F}=\pi_2^*(-K_S)$. It follows from \cite[Prop.~1.3]{wisndef}
that $H$ is relatively nef over $Y_0$, hence
it induces a contraction 
$g_0\colon X_0\to Z_0$ such that  $f_{|X_0}$ factors through $g_0$ and $g_{0|F}=\pi_2$, so that the general fiber of $g_0$ is $\pr^1$. Finally using Lemma \ref{induced} we get the statement.

\medskip

Suppose now that
$\rho_F=4,5$ and $\iota_*\colon\N(F)\to\N(X')$ is injective. Then the restriction $\Nu(X')\to\Nu(F)$ is surjective, hence  $\Nu(X_0)\to\Nu(F)$ is surjective too,
and the monodromy action is trivial. Thus again by \cite[Prop.~1.3]{wisndef}
 we see that every contraction of $F$ extends to a relative contraction of $X_0$ over $Y_0$.

 Since $F$ is a Fano $3$-fold with $\rho_F\geq 4$, by \cite[Theorem on p.~141]{fanoEMS} it has a conic bundle onto a surface. Then $f_{|X_0}$ factors through  a contraction $g_0\colon X_0\to Z_0$ with general fiber $\pr^1$, and we conclude as above by Lemma \ref{induced}.
\end{proof}
\begin{proof}[Proof of Th.~\ref{brasile} in the case $Y\cong\pr^1$]
Let $f'\colon X'\to\pr^1$ be a $K$-negative resolution of $f$, and $\iota\colon F\hookrightarrow X'$ a general fiber.
 If $d_f=\dim\N(F,X')\geq 5$, then either $\rho_F\geq 6$, or $\rho_F=5$ and $\iota_*\colon\N(F)\to\N(F,X')$ is an isomorphism. Thus the statement follows from Prop.~\ref{P1}.
\end{proof}
The following implies Th.~\ref{scalaintro} from the Introduction.
\begin{thm}\label{scala}
  Let $X$ be a smooth Fano $4$-fold with $\rho_X\geq 7$, not isomorphic to a product of surfaces, and $f\colon X\dasharrow Y$ a non-trivial rational contraction of fiber type.

  Then $d_f\leq 4$. If moreover $d_f= 4$ and $\dim Y=2$, then  $Y\cong\pr^2$ and $f\colon X\to\pr^2$ is regular and equidimensional, with general fiber a del Pezzo surface of degree one.
\end{thm}
The bound $d_f\leq 4$ is sharp: see
Ex.~\ref{es3} for an example where $d_f=4$, with $Y=\pr^1$ and $\rho_X=7$.
 Moreover the condition $\rho_X\geq 7$ is necessary, as the following examples show.
\begin{example}\label{es1}
  Let $X$ be a Fano $4$-fold with $\delta_X=3$ and $\rho_X=6$ which is not a product of surfaces. Such $4$-folds have been classified in \cite[\S 7 and Corrigendum]{delta3}, there are 9 families. Each such $X$ has a quasi-elementary contraction $f\colon X\to S$ where $S\cong \pr^1\times\pr^1$ or $S\cong\mathbb{F}_1$, thus $d_f=\rho_X-\rho_S=4$.
\end{example}
\begin{example}\label{es2}
  Let $X$ be  as in Ex.~\ref{es1}.
  By composing $f\colon X\to S$ with a $\pr^1$-bundle $S\to\pr^1$, we get a quasi-elementary contraction $f'\colon X\to\pr^1$ with $d_{f'}=\rho_X-1=5$.
\end{example}  
\begin{proof}[Proof of Th.~\ref{scala}]
  We note first of all that if $\dim Y=3$, then the general fiber of $f$ is $\pr^1$ and $d_f=1$, therefore we can assume that $\dim Y\in\{1,2\}$.
  \begin{prg}\label{firstcase}
  Suppose that  $f$ factors as   $X\stackrel{g}{\dasharrow}Z\stackrel{h}{\to}Y$ where $\dim Z=3$, $g$ is a special rational contraction, and $h$ is a contraction of fiber type.
 Then $d_f\leq 3$ if $\dim Y=2$ and  $d_f\leq 4$ if $Y\cong\pr^1$.
\end{prg}
\begin{proof}
We apply the results in \cite{3folds} to $g\colon X\dasharrow Z$. First of all, 
we have $\delta_X\leq 1$ by Th.~\ref{starting}, hence $\rho_X-\rho_Z\leq 2$ by [\emph{ibid.}, Lemma 3.12]. Moreover $Z$ has the following property:
if $Z\dasharrow Z'$ is a SQM, and $Z'\to W$ is a sequence of elementary divisorial contractions, then every contraction of the sequence is of type $(2,0)$. This is shown in [\emph{ibid.}, Lemma 5.9, 5.13, Lemma 5.15] when $\rho_X-\rho_Z=2$, and in [\emph{ibid.}, proof of Th.~6.3] and \cite[\S 4.2]{eff}
when $g$ is elementary.

  Let $g'\colon X'\to Z$ be a resolution of $g$, and note that $h\circ g'\colon X'\to Y$ is a resolution of $f$. Let  $F_h\subset Z$ be a general fiber of $h\colon Z\to Y$, so that $F:=(g')^{-1}(F_h)\subset X'$ is a general fiber of  $h\circ g'$. By Def.-Rem.~\ref{dimfiber} and Lemma \ref{easy} we get $d_f=\dim\N(F,X')\leq \rho_X-\rho_Z+\dim\N(F_h,Z)\leq 2+\dim\N(F_h,Z)$.

  If $\dim Y=2$, then $F_h\cong\pr^1$, hence $\dim\N(F_h,Z)=1$ and $d_f\leq 3$.

  \medskip
  
Assume that $Y\cong\pr^1$. We show that $\dim\N(F_h,Z)\leq 2$, which implies that $d_f\leq 4$.
  
  By Lemma \ref{montenero} we can factor $h$ as $Z\stackrel{\alpha}{\dasharrow} W\stackrel{\beta}{\to}\pr^1$ where $\alpha$ is an elementary rational contraction, either divisorial or of fiber type, and $\beta$ is a  contraction.

  If $\alpha$ is divisorial, we apply again Lemma \ref{montenero} to $\beta$, and we proceed in this way until we find a factorization of $h$ as: $$\xymatrix{{Z}\ar@/^1pc/[rr]^h\ar@{-->}[r]_{\alpha'}&{W'}\ar[r]_{\beta'}
    &{\pr^1}}\quad\text{or}\quad
\xymatrix{{Z}\ar@/^1pc/[rrr]^h\ar@{-->}[r]_{\alpha'}&{W'}\ar@{-->}[r]_{\beta''}&S\ar[r]_{\gamma}
    &{\pr^1}}
  $$
  where $\alpha'$ is a (possibly empty) sequence of elementary divisorial rational contractions, $\beta'$ and $\beta''$ are elementary of fiber type,  $\dim S=2$, and $\beta'$ and $\gamma$ are regular.

  Up to composing with SQM's, we can assume that $\beta''$  and $\alpha'$ are regular (note that $\dim\N(F_h,Z)$ does not vary, see Def.-Rem.~\ref{dimfiber}). Then $\alpha'(F_h)\subset W'$ is a general fiber of $\beta'$ or of $\gamma\circ \beta''$, and $\dim\N(\alpha'(F_h),W')\leq 2$ because both contractions $\beta'$ and $\beta''$ are elementary, and the general fiber of $\gamma$ is $\pr^1$. Moreover, since $\dim \alpha'(\Exc(\alpha'))=0$ by the properties of $Z$, we have $\alpha'(F_h)\cap \alpha'(\Exc(\alpha'))=\emptyset$, hence $\dim\N(F_h, Z)=\dim\N(\alpha'(F_h),W')\leq 2$ (see \cite[Lemma 2.17]{fibrations}).
\end{proof}
\begin{prg}
  We consider now the case where $f$ cannot be factored as in \ref{firstcase}; then we get $d_f\leq 4$ by
  Th.~\ref{brasile}$(a)$.

  Suppose moreover  that
$d_f= 4$ and $\dim Y=2$. Then Th.~\ref{brasile}$(b)$ implies that
$f\colon X\to Y$ is regular, with general fiber a del Pezzo surface of degree one.

We show that $f$ is special and $Y\cong\pr^2$.
By Prop.~\ref{calanques} we can factor $f$ as $X\stackrel{f'}{\dasharrow} Y'\stackrel{\ph}{\to}Y$ where $f'$ is a special rational contraction of fiber type, and $\ph$ is birational; then $Y'$ is a smooth rational surface (Lemma \ref{specialsurf}). If $Y'\cong\pr^2$, then $\ph$ is an isomorphism, $f$ is special, and we get the statement.

Assume by contradiction that $Y'\not\cong\pr^2$. Then $Y'$ is obtained as a blow-up of some Hirzebruch surface, therefore it has a contraction $\psi\colon Y'\to\pr^1$. Consider $\psi\circ f'\colon X\dasharrow\pr^1$. We show that $d_{\psi\circ f'}>4$, contradicting the first part of the proof.

  Let $\zeta\colon X\dasharrow X'$ be a SQM such that
  $f'':=f'\circ\zeta^{-1}\colon X'\to Y'$ is regular.
$$\xymatrix{ X\ar@{-->}[r]_{f'}\ar@/^1pc/[rr]^f\ar@{-->}[d]_{\zeta}&
  {Y'}\ar[r]_{\ph}\ar[dr]_{\psi}& Y\\
  {X'}\ar[ur]_{f''}&&{\pr^1}
  }$$
 Since $\ph$ is birational, $f''$ and $\ph\circ f''$ have the same general fiber, and
 $d_{f''}=d_{\ph\circ f''}=d_f=4$; moreover  $d_{\psi\circ f'}=d_{\psi\circ f''}$ (see Def.-Rem.~\ref{dimfiber}). 
 If $F_2\subset X'$ is a general fiber of $\psi\circ f''$, and $F_1\subset F_2\subset X'$ a general fiber of $f''$, then $d_{\psi\circ f''}=\dim\N(F_2,X')>\dim\N(F_1,X')= 4$. This concludes the proof.  \qedhere
\end{prg}
\end{proof}
The following implies Cor.~\ref{summaryintro} from the Introduction.
\begin{corollary}\label{summary}
 Let $X$ be a smooth Fano $4$-fold with $\rho_X\geq 7$, not isomorphic to a product of surfaces.
  If $\tau$ is a movable face of $\Eff(X)$, then 
  $\dim\tau\geq\rho_X-4\geq 3$.

  In particular, every face of dimension $1$ or $2$ of $\Eff(X)$ is fixed.
\end{corollary}
Again, the condition $\rho_X\geq 7$ is necessary, as the following example shows.
\begin{example}\label{perone}
Let $X$ be a Fano $4$-fold as in Examples \ref{es1} and \ref{es2}; we have $\rho_X=6$ and $X$ is not a product of surfaces. The cone $\Eff(X)$ has a one-dimensional movable face, given by $(f')^*\Nef(\pr^1)$ where $f'\colon X\to\pr^1$ is a quasi-elementary contraction as in Ex.~\ref{es2}.
\end{example}
\begin{proof}[Proof of Cor.~\ref{summary}]
By Rem.~\ref{zumba} there is  a non-trivial rational contraction of fiber type 
$f\colon X\dasharrow Y$
such that $\dim\tau\geq \rho_X-d_f$. We have 
$d_f\leq 4$ by Th.~\ref{scala},  hence $\dim\tau\geq\rho_X-4$.
\end{proof}
\section{Preliminaries on fixed prime divisors of Fano $4$-folds}\label{prelfixed}
\noindent In this section we recall the classification of fixed prime divisors in Fano $4$-folds $X$ with 
 $\rho_X\geq 7$,
     or $\rho_X=6$ and $\delta_X\leq 2$, and give many related properties that are used in the sequel.
     \begin{thmdef}[{\bf the type of a fixed prime divisor}]\label{fixed}  {\em (See \cite[Th.~5.1,
Cor.~5.2, Lemma 5.25]{blowup}.)}
  \  Let $X$ be a smooth Fano $4$-fold with  $\rho_X\geq 7$,
     or $\rho_X=6$ and $\delta_X\leq 2$,
    and $D$ a fixed prime divisor in $X$.
      \begin{enumerate}[$(a)$]
      \item There exists a unique diagram:
        $$X\stackrel{\xi}{\dasharrow}\w{X}\stackrel{\sigma}{\la}Y$$
        where $\xi$ is a SQM, $\sigma$ is a divisorial  elementary contraction with exceptional divisor the transform $\w{D}$ of $D$, and $Y$ is Fano (possibly singular);
      \item $\sigma$ is of type $(3,0)^{\sm}$, $(3,0)^Q$, $(3,1)^{\sm}$, or $(3,2)$, and we define $D$ to be of type  $(3,0)^{\sm}$, $(3,0)^Q$, $(3,1)^{\sm}$, or $(3,2)$, accordingly;
        \item if $D$ is of type $(3,2)$, then $X=\w{X}$. In the other cases $\xi$ factors as a sequence of at least $\rho_X-4$ $D$-negative and $K$-negative flips.
        \item We define $C_D\subset D\subset X$ to be the transform of a general irreducible curve $C_{\w{D}}\subset \w{D}\subset \w{X}$ contracted by $\sigma$, of minimal anticanonical degree. Then $C_D\cong\pr^1$, $D\cdot C_D=-1$, and $C_D\subset\dom(\xi)$.
        \item Given a SQM $\ph\colon X\dasharrow X'$ and a divisorial elementary  contraction $\sigma'\colon X'\to Y'$ with $\Exc(\sigma')$ the transform of $D$, there is a commutative diagram:          $$\xymatrix{X\ar@/^1pc/@{-->}[rr]^{\ph}\ar@{-->}[r]_{\xi}&{\w{X}}\ar@{-->}[r]_{\psi_{\s X}}\ar[d]_{\sigma}&{X'}\ar[d]^{\sigma'}\\
            & Y\ar@{-->}[r]^{\psi_{\s Y}}&{Y'}
          }$$
          where $\psi_{\s X}$ and $\psi_{\s Y}$ are SQM's, $\w{D}\subset\dom(\psi_{\s X})$,
          and $\sigma(\w{D})\subset\dom(\psi_{\s Y})$.
      \end{enumerate}  
    \end{thmdef}
    Among the known families of Fano $4$-folds with $\rho_X\geq 7$, we have examples of all four types of fixed prime divisors, see \S\ref{fanomodel} and \S\ref{newex}.
    \begin{example}
Let $X=S_1\times S_2$ with $S_1$ and $S_2$ del Pezzo surfaces, and $\rho_X\geq 7$. Then every fixed prime divisor $D\subset X$ is of type $(3,2)$, and $D=C\times S_2$ or $S_1\times C$ where $C$ is a $(-1)$-curve.
    \end{example}  
\begin{lemma}[\cite{blowup}, Rem.~2.17(2) and \cite{eff}, Cor.~3.14]\label{dim32}
  Let $X$ be a smooth Fano $4$-fold with $\rho_X\geq 7$,  or $\rho_X=6$ and $\delta_X\leq 2$. Let $E\subset X$ be a fixed prime divisor of type $(3,2)$, $X\dasharrow X'$ a SQM, and $E'\subset X'$ the transform of $E$. Then
   $E$ does not contain exceptional planes and
  $\dim\N(E,X)=\dim\N(E',X')$.
\end{lemma}

    Let $X$ be a smooth Fano $4$-fold, $\ph\colon X\dasharrow X'$ a SQM, and $E\subset X'$ a fixed prime divisor. We define {\bf the type of $E$} to be the type
    of its transform $E_{\s X}\subset X$, and we define $C_E\subset E\subset X'$ to be the transform of $C_{E_{\s X}}\subset X$.
    \begin{remark}\label{greg}
      In the above setting, we have $C_{E_X}\subset\dom(\ph)$ and $C_E\subset\dom(\ph^{-1})$. 

  Indeed $X\smallsetminus\dom(\ph)$ is a finite union of exceptional planes (Lemma \ref{SQMFano}). If $E$ is of type $(3,2)$, then $E_X$ does contain exceptional planes by Lemma \ref{dim32}, hence $X\smallsetminus\dom(\ph)$ intersects $E_X$ at most in dimension one, and does not touch $C_{E_X}$.

  Suppose that $E$ is not of type $(3,2)$, and let $L\subset X\smallsetminus\dom(\ph)$ be an exceptional plane. If $L\subset E_X$, then $C_{E_X}\cap L=\emptyset$. If instead $L\not\subset E_X$ but $L\cap E_X\neq\emptyset$, let us consider the SQM $\xi\colon X\dasharrow\w{X}$ as in Th.-Def.~\ref{fixed}, and let $\w{E}\subset\w{X}$ be the transform of $E_X$. Then $L\not\subset X\smallsetminus\dom(\xi)$ and if $\w{L}\subset\w{X}$ is its transform, we have $\dim(\w{L}\cap\w{E})=1$, and the general curve $C_{\w{E}}$ is disjoint from $\w{L}$, thus again $C_{E_X}\cap L=\emptyset$.
\end{remark}
     \begin{remark}[the cone $\Mov(X)^{\vee}$]\label{movdual}
   Let $X$ be a smooth Fano $4$-fold with $\rho_X\geq 7$, or $\rho_X=6$ and $\delta_X\leq 2$.
 Consider the cone $\Mov(X)^{\vee}\subset\N(X)$, dual of the cone of movable divisors, and note that $\mov(X)\subset\Mov(X)^{\vee}\subset\NE(X)$ because dually $\Nef(X)\subset\Mov(X)\subset\Eff(X)$. By 
  \cite[Lemma 5.29]{blowup} we have:
  $$\Mov(X)^{\vee}=\langle [C_D]\rangle_{D\text{ fixed}}+\mov(X),$$
  and every $\langle[C_D]\rangle$ is a one-dimensional face of  $\Mov(X)^{\vee}$.

  This means that $\Mov(X)^{\vee}$ has two types of one-dimensional faces, and dually that $\Mov(X)$ has two types of facets. The one-dimensional face $\langle[C_D]\rangle$ 
  corresponds to the facet $\Mov(X)\cap [C_D]^{\perp}$ of $\Mov(X)$; the cones of $\MCD(X)$ of dimension $\rho_X-1$ contained in this facet correspond precisely to divisorial elementary rational contractions $X\dasharrow Y$ with exceptional divisor $D$.

  The second type of one-dimensional faces $\alpha$ of $\Mov(X)^{\vee}$ are those contained in $\mov(X)$, so that $\alpha$ is a common one-dimensional face of the two cones. Then the corresponding facet $\alpha^{\perp}\cap\Mov(X)$ of $\Mov(X)$ is contained in a (movable) facet of $\Eff(X)$, and  the cones of $\MCD(X)$ of dimension $\rho_X-1$ contained in this facet correspond  to  elementary rational contractions of fiber type.
\end{remark}
\begin{lemma}\label{iso32}
  Let $X$ be a smooth Fano $4$-fold with $\rho_X\geq 7$, or $\rho_X=6$ and $\delta_X\leq 2$. Let
 $\ph\colon X\dasharrow X'$ be a SQM and
  $E\subset X'$ a fixed prime divisor of type $(3,2)$. Then $[C_E]$ generates an extremal ray of $\NE(X')$ if and only if $E$ does not meet any exceptional line.
\end{lemma}
\begin{proof}
  If $[C_E]$ generates an extremal ray of $\NE(X')$, it must be divisorial
by Rem.~\ref{greg}, and
  the statement follows from  Th.-Def.~\ref{fixed} and Lemma \ref{SQMFano}.

Conversely, suppose that $E\subset X'$ does not meet  exceptional lines. Then   $E\subset\dom (\ph^{-1})$ (Lemma \ref{SQMFano}$(a)$) and  $\ph^{-1}(E)\subset X$ cannot contain exceptional planes (Lemma \ref{dim32}), therefore neither can $E$.

Since $E\cdot C_E<0$, $\NE(X')$ must have an $E$-negative extremal ray $R$; we have $\Lo(R)\subset E$, hence $R$ is birational.

  We claim that $R$ cannot be small. Indeed if $R_0$ is a small extremal ray of  $\NE(X')$, and $K_{X'}\cdot R_0\geq 0$, then $\Lo(R_0)$ is a finite union of exceptional lines (see Lemma \ref{SQMFano}), hence $\Lo(R_0)\cap E=\emptyset$. If instead $-K_{X'}\cdot R_0>0$, then $\Lo(R_0)$ is a finite union of exceptional planes (see Lemma \ref{kawamata}), therefore $\Lo(R_0)\not\subset E$.

  We conclude that $R$ is divisorial, and by Th.-Def.~\ref{fixed}
  we must have $[C_E]\in R$.
\end{proof}
\begin{lemma}\label{extremal}
Let $X$ be a smooth Fano $4$-fold with $\rho_X\geq 7$,  or $\rho_X=6$ and $\delta_X\leq 2$. Let $\ph\colon X\dasharrow X'$ be a SQM, $E\subset X'$ a fixed prime divisor of type $(3,2)$, and $f\colon X'\to Y$ a $K$-negative contraction such that $f(C_E)=\{\pt\}$. Then $[C_E]$ generates an extremal ray of $\NE(X')$.
\end{lemma}
\begin{proof}
  Since $[C_E]\in\NE(f)$ and $E\cdot C_E<0$, there exists an extremal ray $R$ of $\NE(f)$ such that $E\cdot R<0$. Moreover $R$ is $K$-negative, because $f$ is. If $R$ is small, then $\Lo(R)$ is a finite union of exceptional planes (see Lemma \ref{kawamata}). In particular $\Lo(R)$ must be contained in $\dom(\ph^{-1})$ (see Lemma \ref{SQMFano}), and the transform of $E$ in $X$ contains an exceptional plane, a contradiction (see Lemma \ref{dim32}).
    Hence $R$ is divisorial, and by Th.-Def.~\ref{fixed}
  we must have $[C_E]\in R$.
\end{proof} 
\begin{lemma}\label{target}
  Let $X$ be a smooth Fano $4$-fold with $\rho_X\geq 7$,  or $\rho_X=6$ and $\delta_X\leq 2$. Let $\ph\colon X\dasharrow X'$ be a SQM and $f\colon X'\to Y$ an elementary contraction of type $(3,2)$; set $A:=f(\Exc(f))\subset Y$.
 
 Then
  $Y$ can contain finitely many pairwise disjoint exceptional lines, all disjoint from $A$. If $\Gamma\subset Y$ is an irreducible curve that is not an exceptional line, then $-K_Y\cdot \Gamma\geq 1$. Suppose moreover that
  $-K_Y\cdot \Gamma=1$. Then $\Gamma$ cannot meet any exceptional line, and if $\Gamma$ meets $A$, then $\Gamma\subset A$.

Furthermore,  for every SQM $\psi\colon Y\dasharrow Y'$, we have $A\subset\dom(\psi)$.
      \end{lemma}
      \begin{proof}
        We note that $f$ is $K$-negative (see Lemma \ref{SQMFano}) and $Y$
is locally factorial with
    at most ordinary double points, contained in $A$ (Lemma \ref{singtarget}).
        By Th.-Def.~\ref{fixed}$(e)$ we have a commutative diagram:
        $$\xymatrix{ X\ar[d]_{\sigma}\ar@{-->}[r]^{\ph}&{X'}\ar[d]^f\\
          {\w{Y}}\ar@{-->}[r]^{\ph_Y}&Y
          }$$
          where $\ph_Y$ is a SQM, $\sigma$ is a divisorial elementary contraction with exceptional divisor the transform of $\Exc(f)$, $\w{Y}$ is Fano, $\Exc(\sigma)\subset\dom(\ph)$, and $\sigma(\Exc(\sigma))\subset\dom(\ph_Y)$.
Then the same proof as \cite[proof of Lemma 5.48]{3folds}
gives the statement.  Note that if   $\Gamma\subset Y$ is an irreducible curve with $-K_Y\cdot\Gamma=1$, $\Gamma\cap A\neq\emptyset$, and $\Gamma\not\subset A$, then the transform $\Gamma'\subset X'$ has $-K_{X'}\cdot\Gamma'\leq 0$, thus $\Gamma'$ is an exceptional line (Lemma \ref{SQMFano}) and meets $\Exc(f)$, contradicting   Lemma \ref{iso32}.  \end{proof}
\begin{lemma}[\cite{small}, Lemmas 4.12 and 4.13]\label{EF}
Let $X$ be a smooth Fano $4$-fold with $\rho_X\geq 7$ and $D,E\subset X$ two
distinct fixed prime divisors.
\begin{enumerate}[$(a)$]
  \item
If $D\cdot C_E=0$, then $D$ and $E$ are adjacent.
\item
  If $E$ is of type $(3,2)$ and $D$ and $E$ are adjacent, then $D\cdot C_E=0$.
\end{enumerate}\end{lemma}
\begin{lemma}\label{linate}
Let $X$ be a smooth Fano $4$-fold with $\rho_X\geq 7$, or $\rho_X=6$ and $\delta_X\leq 2$. Let $D_1,\dotsc,D_m\subset X$ be
fixed prime divisors such that $D_i\cdot C_{D_j}=0$ for $i\neq j$. Then $\langle [D_1],\dotsc,[D_m]\rangle$ is a fixed face of $\Eff(X)$ of dimension $m$, and there exists a birational (rational) contraction $f\colon X\dasharrow Y$ with exceptional divisors $D_1,\dotsc,D_m$, where $Y$ is $\Q$-factorial with $\rho_Y=\rho_X-m$.
\end{lemma}
\begin{proof}
  It is easy to see that $[D_1],\dotsc,[D_m]\in\Nu(X)$ are linearly independent. Moreover if $D$ is an effective  divisor with $[D]\in\langle [D_1],\dotsc,[D_m]\rangle$, $D\neq 0$, then $D\cdot C_{D_i}<0$ for some $i\in\{1,\dotsc,m\}$, therefore $D$ is not movable, and $\langle [D_1],\dotsc,[D_m]\rangle\cap\Mov(X)=\{0\}$.

  Each $[D_i]$ generates a one-dimensional face of $\Eff(X)$.
  To show that $\langle [D_1],\dotsc,[D_m]\rangle$ is a face, we proceed by induction on $m$, by applying Lemma \ref{cones} with the linear maps on $\Nu(X)$ given by intersection with $C_{D_i}$.

  Finally, for the existence of the map $f$, see \cite[Lemma 4.2 and its proof]{small}.
\end{proof}  
\begin{lemma}\label{30sm}
  Let $X$ be a smooth Fano $4$-fold with $\rho_X\geq 7$ and $D$ and $E$ two adjacent fixed prime divisors, 
 of type $(3,0)^{\sm}$ and $(3,2)$ respectively. Then $D\cap E=\emptyset$.
\end{lemma}
\begin{proof}
  This follows from Lemma \ref{dim32} and \cite[Lemma 4.9]{fibrations}.
\end{proof}
\begin{lemma}[\cite{small}, proof of Lemma 7.2]\label{30Q}
Let $X$ be a smooth Fano $4$-fold with $\rho_X\geq 7$. Let $D$ be a fixed prime divisor of type $(3,0)^{Q}$, and $E_1$, $E_2$ fixed prime divisors of type $(3,2)$, both adjacent to $D$, such that $D\cap E_i\neq\emptyset$ for $i=1,2$. Then $E_1\cdot C_{E_2}=E_2\cdot C_{E_1}=1$.
\end{lemma}
\begin{lemma}\label{disjoint}
  Let $X$ be a smooth Fano $4$-fold with $\rho_X\geq 7$, $\psi\colon X\dasharrow X'$ a SQM,
  and $D,E\subset X'$ two adjacent fixed prime divisors, of type $(3,1)^{\sm}$ and $(3,2)$ respectively, with $E\cdot C_D>0$.
Then $E\cdot C_D=1$ and $E \cap L=\emptyset$ for every  exceptional plane $L\subset D$.
\end{lemma}
\begin{proof}
  Let $D_X,E_X\subset X$ be the transforms of $D,E$ respectively.
  It is shown in \cite[Lemma 4.23]{small} that  $E_X\cdot C_{D_X}=1$ and
  that $E_X$ is disjoint from every exceptional plane contained in $D_X$.
  This implies the statement,
  indeed $E\cdot C_D=E_X\cdot C_{D_X}$ (see Rem.~\ref{greg}), and if  $L\subset D$ is an exceptional plane,
  then $L\subset\dom(\psi^{-1})$ (see Lemma \ref{SQMFano}), and its transform $L_X\subset X$ is an exceptional contained in $D_X$, therefore $L_X\cap E_X=\emptyset$, and hence $L\cap E=\emptyset$.
\end{proof}
\begin{lemma}\label{square}
  Let $X$ be a smooth Fano $4$-fold with $\rho_X\geq 7$,
  $X\dasharrow\wi{X}$ a SQM, and
  $\sigma\colon \wi{X}\to Y$ and  $\alpha\colon \wi{X}\to Z$ divisorial elementary contractions of type $(3,2)$ and $(3,1)^{\sm}$ respectively. Set $E:=\Exc(\sigma)$ and $D:=\Exc(\alpha)$, and assume that $D\cdot C_E=0$ and $E\cdot C_D>0$.

 Then 
 $\langle [C_D],[C_E]\rangle$ is a face of $\NE(\wi{X})$ and we have a commutative diagram:
 \stepcounter{thm}
 \begin{equation}\label{sq} \xymatrix{
{\wi{X}}\ar[d]_{\alpha}\ar[r]^{\sigma}&Y\ar[d]^{\alpha'}\\
Z\ar[r]^{\sigma'}&W
}\end{equation}
where $\sigma'$ is an elementary contraction of type $(3,2)$ with exceptional divisor  $\alpha(E)$, and $\alpha'$ is the blow-up of a smooth point with exceptional divisor $\sigma(D)$ (see Fig.~\ref{figura_square}). Moreover $D\cong\pr_{\pr^1}(\ol^{\oplus 2}\oplus\ol(1))$ and $E\cdot C_D=1$.
\end{lemma}
\begin{figure}\caption{The varieties in Lemma \ref{square}.}\label{figura_square}

  \bigskip

  \medskip
  
  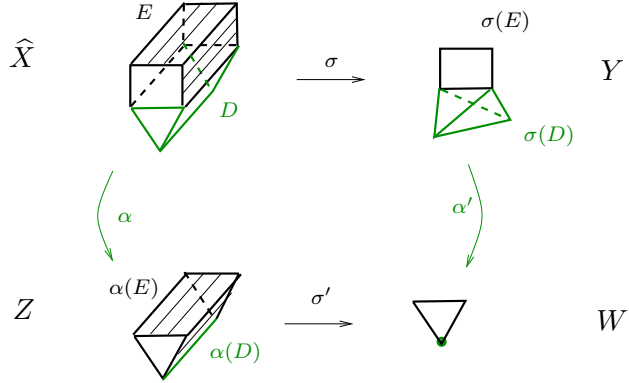
\end{figure}
\begin{proof}
  Since $D\cdot C_E=0$, $D$ and $E$ are adjacent by Lemma \ref{EF}, and
  $E\cdot C_D=1$ by Lemma \ref{disjoint}.

  We note that $D\cdot R\geq 0$ for every extremal ray $R$ of $\NE(\wi{X})$ different from $\R_{\geq 0}[C_D]$; this follows from Th.-Def.~\ref{fixed},
    Lemma \ref{SQMFano}, and \cite[Rem.~5.6]{blowup}.  

 Let $H\in\Pic(\wi{X})$ be the pullback under $\sigma$ of an ample divisor on $Y$, and set $m:=H\cdot C_D>0$. Then $(H+mD)\cdot C_E=
 (H+mD)\cdot C_D=0$, and if $R$ is an extremal ray of $\NE(\wi{X})$
different from $\R_{\geq 0}[C_D]$ and $\R_{\geq 0}[C_E]$,
then $H\cdot R>0$, $D\cdot R\geq 0$, and hence $(H+mD)\cdot R>0$. This shows that
$H+mD$ is nef and $(H+mD)^{\perp}\cap\NE(\wi{X})=\langle [C_D],[C_E]\rangle$ is a face of $\NE(\wi{X})$, hence we have a commutative diagram as \eqref{sq} where $\sigma'$ and $\alpha'$ are divisorial elementary contractions with exceptional divisors the images of $E$ and $D$ respectively. Then it is not difficult to see as in \cite[proof of Lemma 4.23]{small} that $\sigma'$ is of type $(3,2)$, $\alpha$ is the blow-up of a one-dimensional fiber of $\sigma'$,
$D\cong\pr_{\pr^1}(\ol^{\oplus 2}\oplus\ol(1))$, $\sigma(D)\cong\pr^3$, and $\alpha'$ is of type $(3,0)^{\sm}$. See also \cite[Lemma 2.18]{3folds}.
\end{proof}
\begin{lemma}\label{unique}
  Let $X$ be a smooth Fano $4$-fold with $\rho_X\geq 7$ and $D$ a fixed prime divisor of type $(3,1)^{\sm}$. Then there is at most one fixed prime divisor $E$ of type $(3,2)$ which is adjacent to $D$ and such that $E\cdot C_D>0$.
\end{lemma}
\begin{proof}
This follows from \cite[Lemma 4.23 and its proof]{small}; see also Lemma \ref{square}. Let $X\stackrel{\xi}{\dasharrow}\w{X}\to Y$ be the contraction of $D$ as in Th.-Def.~\ref{fixed}$(a)$, and $\w{D},\w{E}\subset\w{X}$ the transforms of $D,E$ respectively. Then $\w{D}$ is isomorphic to the blow-up of $\pr^3$ along a line, $E\subset\dom(\xi)$ so that $\w{E}\cong E$,  $\w{E}
\cap\w{D}$ is the exceptional divisor of $\w{D}\to\pr^3$, and the fibers of the blow-up  $\w{D}\to\pr^3$ are numerically equivalent to $\xi(C_E)$. Therefore the class $[C_E]$ is uniquely determined by $D$.
\end{proof}
\begin{lemma}\label{31sm}
  Let $X$ be a smooth Fano $4$-fold with $\rho_X\geq 7$. Let $D$ be a fixed prime divisor of type $(3,1)^{\sm}$, and $E_1$, $E_2$ distinct fixed prime divisors of type $(3,2)$, both adjacent to $D$, such that $D\cap E_i\neq\emptyset$, for $i=1,2$. Then one of the following holds:
  \begin{enumerate}[$(i)$]
     \item  
  $E_1\cdot C_{E_2}=E_2\cdot C_{E_1}=0$;
  \item  up to exchanging $E_1$ and $E_2$, we have
   $E_1\cdot C_D=0$ and  $E_2\cdot C_D=E_2\cdot C_{E_1}=1$.
\end{enumerate}  
\end{lemma}
\begin{proof}
We assume that $(i)$ does not hold, and show $(ii)$.
If $E_1\cdot C_{E_2}=0$, then $E_1$ and $E_2$ are adjacent by Lemma \ref{EF}, hence
$E_2\cdot C_{E_1}=0$ by the same lemma; similarly if $E_2\cdot C_{E_1}=0$. Therefore we can assume that 
$E_1\cdot C_{E_2}>0$ and $E_2\cdot C_{E_1}>0$.

By Lemma \ref{unique} there exists $i\in\{1,2\}$ such that $E_{i}\cdot C_D=0$; up to exchanging $E_1$ and $E_2$, we can suppose that  $E_{1}\cdot C_D=0$. Note that $D\cdot C_{E_1}=0$ because $D$ and $E_1$ are adjacent (Lemma \ref{EF}).

Then by \cite[Lemma 6.9 and Prop.~6.1]{small} there exists an exceptional plane $L\subset D$ such that $C_D\equiv C_L+C_{E_1}$. By Lemma \ref{dim32} we have $L\not\subset E_2$, thus $E_2\cdot C_L\geq 0$ and 
 $E_2\cdot C_D\geq E_2\cdot C_{E_1}>0$. 
  By  Lemma \ref{disjoint}
  we have  $E_2\cdot C_D=1$, which implies that $E_2\cdot C_{E_1}=1$.
\end{proof}
\begin{lemma}\label{deltaY}
Let $X$ be a smooth Fano $4$-fold with $\rho_X\geq 7$, $D$ a fixed prime divisor of type $(3,0)^{\sm}$ or $(3,1)^{\sm}$, and $X\stackrel{\xi}{\dasharrow} \w{X}\stackrel{\sigma}{\to}Y$ the contraction of $D$ as in Th.-Def.~\ref{fixed}$(a)$. Then $\delta_Y\leq 2$.
\end{lemma}
\begin{proof}
  We note first of all that $X$ is not isomorphic to a product of surfaces because it has a  fixed prime divisor of type $(3,0)^{\sm}$ or $(3,1)^{\sm}$,
therefore
$\delta_X\leq 1$ by Th.~\ref{starting}. Moreover $Y$ is a smooth Fano $4$-fold.

 If by contradiction $\delta_Y\geq 3$, then $\NE(Y)$ has
 an extremal ray $R$ of type $(3,2)$ with $\codim\N(E,Y)=\delta_Y$ where $E=\Lo(R)$ (see \cite[Rem.~3.2.6]{codim} for the case $\delta_Y\geq 4$, and \cite[Th.~1.4 and proof of Lemmas 3.2 and 4.5]{delta3} for the case $\delta_Y=3$).

Consider the transforms $E_X\subset X$ and $E_{\w{X}}\subset\w{X}$.
Let $\Gamma\subset E$ be a general fiber of the contraction of $R$, and let
$\Gamma_X\subset X$ and $\Gamma_{\w{X}}\subset\w{X}$ be its transforms. By generality we have $\Gamma\cap \sigma(\Exc(\sigma))=\emptyset$, hence   $-K_{\w{X}}\cdot\Gamma_{\w{X}}=-K_Y\cdot\Gamma=1$. By Lemma \ref{SQMFano} this implies that $\Gamma_{\w{X}}\subset\dom(\xi^{-1})$, therefore $-K_X\cdot \Gamma_X=1$ and $E_X\cdot\Gamma_X=E\cdot\Gamma=-1$. Thus $E_X\subset X$ is a non-nef prime divisor covered by rational curves of anticanonical degree one, so that  $E_X$ is a fixed prime divisor of type $(3,2)$ by \cite[Lemma 2.18]{blowup}.
Then using Lemmas \ref{easy}  and \ref{dim32} we get
$$\dim\N(E,Y)\geq\dim\N(E_{\w{X}},\w{X})-1=\dim\N(E_X,X)-1\geq\rho_X-2=\rho_Y-1,$$ a contradiction.
\end{proof}   
  \begin{lemma}\label{lifting}
  Let $X$ be a Fano $4$-fold with $\rho_X\geq 7$ and $\delta_X\leq 1$, $X\dasharrow X'$ a SQM, and $f\colon X'\to S$ a  contraction.
Assume that  $\rho_S\geq 3$, and  
  let $\Gamma\subset S_{\reg}$ be a $(-1)$-curve. Then there is an exceptional line $\ell\subset X'$, whose class is extremal in $\NE(X')$, such that $f(\ell)=\Gamma$.
  \end{lemma}
\begin{proof}
  Since $\R_{\geq 0}[\Gamma]$ is an extremal ray of $\NE(S)$, there is an extremal ray $R$ of $\NE(X')$ such that $f_*(R)=\R_{\geq 0}[\Gamma]$, so that the contraction of $R$ must be birational, with fibers of dimension at most one (see \cite[\S 2.5]{fanos}).  If $K_{X'}\cdot R\geq 0$, then $R$ contains the class of  an exceptional line $\ell\subset X'$ (see Lemma \ref{SQMFano}) and we have $f(\ell)=\Gamma$.

 We suppose that $K_{X'}\cdot R< 0$ and show that this gives a contradiction.
 Then $R$ cannot be small (Lemma \ref{kawamata}), thus it must be of type $(3,2)$  with locus $D:=f^{-1}(\Gamma)$, and $D$ is a fixed prime divisor of type $(3,2)$ (see Th.-Def.~\ref{fixed}). By Lemma \ref{easy} we have $\dim\N(D,X')\leq 1+\rho_X-\rho_S\leq\rho_X-2$.
 On the other hand, if $D_X\subset X$ is the transform of $D$,  Lemma \ref{dim32} gives $\dim\N(D_X,X)=\dim\N(D,X')$,  contradicting $\delta_X\leq 1$.
\end{proof}
\begin{remark}\label{uffa}
  Let $X$ be a smooth Fano $4$-fold with $\rho_X\geq 7$ and
  $E,D_1,\dotsc,D_m$ fixed prime divisors such that $E$ is of type $(3,2)$, each $D_i$ is of type $(3,1)^{\sm}$, and $E\cdot C_{D_i}=D_i\cdot C_{D_j}=0$ for every $i,j=1,\dotsc,m$, $i\neq j$.

  By Lemma \ref{EF} this implies that $D_i\cdot C_{E}=0$  for every $i=1,\dotsc,m$, and Lemma \ref{linate} gives that
  $\langle [E],[D_1],\dotsc,[D_m]\rangle$ is a fixed face of $\Eff(X)$, and $E,D_1,\dotsc,D_m$ can be contracted together by a birational (rational) contraction with $\Q$-factorial target. However the geometry of this contraction depends on the order in which we contract the divisors; let us describe this. See also \cite[Lemma 2.19]{3folds} for the case $m=1$, and \cite[Lemma 6.9, Prop.~6.1, Prop.~6.4]{small}.

Let us first describe how the divisors $E,D_1,\dotsc,D_m$ intersect in $X$.
For every $i,j\in\{1,\dotsc,m\}$, $i\neq j$, $D_i\cap D_j$ is either empty or a disjoint union of exceptional planes \cite[Lemma 4.13(b)]{small}. Let us also note that $X$ is not a product of surfaces because it has fixed prime divisors of type $(3,1)^{\sm}$, thus $\delta_X\leq 1$ by Th.~\ref{starting}, and $E$ can be disjoint from at most one $D_i$ (otherwise, if $E$ is disjoint from $D_i$ and $D_j$ with $i\neq j$, then $\N(E,X)\subset D_i^{\perp}\cap D_j^{\perp}$ by Rem.~\ref{cassis}, a contradiction).
Moreover
by \cite[Lemma 2.17$(ii)$]{3folds}, for every $i=1,\dotsc,m$ such that $D_i\cap E\neq\emptyset$, there exists an exceptional plane $L_i\subset D_i$ such that $D_i\cdot C_{L_i}=-1$, $E\cdot C_{L_i}=1$, $C_{D_i}\equiv C_E+C_{L_i}$, and $E\cap L=\emptyset$ for every exceptional plane $L\subset D_i$ such that $C_L\not\equiv C_{L_i}$; we also have that $E\cap D_i$ is a disjoint union of  surfaces isomorphic to $\mathbb{F}_1$. For $j\neq i$ we get $D_j\cdot C_{L_i}=0$ and hence $D_j\cap L_i=\emptyset$, because $D_j\cdot C_{L_0}<0$ for every exceptional plane $L_0\subset D_j$.

We can first perform a sequence of $(D_1+\cdots+D_m)$-negative and $E$-trivial flips, and get a SQM $\xi_1\colon X\dasharrow\w{X}$. Then $E\subset\dom(\xi_1)$, so that $[C_E]$ is still extremal in $\NE(\w{X})$ (Lemma \ref{iso32}), and we have a diagram:
$$\xymatrix{X\ar@{-->}[r]^{\xi_1}\ar[d]_{\alpha}&{\w{X}}\ar[dr]^f\ar[d]^{\tilde{\alpha}}&\\
Y\ar@{-->}[r]^{\xi_Y}&{\w{Y}}\ar[r]^{\sigma_Y}&W
 }$$
 where $\alpha$ and $\tilde\alpha$ are the contractions of $\R_{\geq 0}[C_E]$, of type $(3,2)$, and locally isomorphic, $Y$ is Fano with at most isolated, locally factorial, terminal singularities, and $\xi_Y$ is a SQM (Th.-Def.~\ref{fixed}, Lemma \ref{singtarget}).

\begin{figure}\caption{The varieties in Rem.~\ref{uffa}, with $m=1$.}\label{figura2}

\bigskip
  
  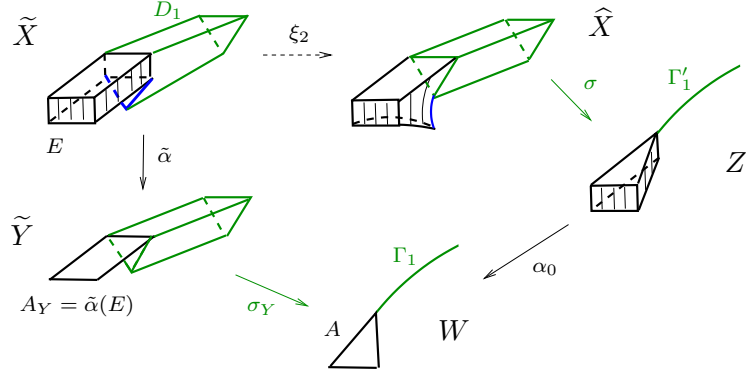
\end{figure}
 
 In $\w{Y}$ the divisors $D_1,\dotsc,D_m$ are pairwise disjoint, are contained in $\w{Y}_{\reg}$, and are the exceptional divisors of $\sigma_Y\colon\w{Y}\to W$ blow-up of $m$ pairwise disjoint smooth irreducible curves $\Gamma_1,\dotsc,\Gamma_m\subset W_{\reg}$;  $W$ has the same singularities as $Y$ and $\w{Y}$ (see Fig.~\ref{figura2}).

 Set $f:=\sigma_Y\circ\tilde\alpha\colon\w{X}\to W$. Then $f$ is
a resolution of the birational map $X\dasharrow W$, with
  exceptional divisors $E,D_1,\dotsc,D_m$. In $\w{X}$ the relative cone of $f$ is $\NE(f)=\langle [C_E],[C_{L_1}],\dotsc,[C_{L_m}]\rangle$, and $C_{D_i}\equiv C_E+C_{L_i}$ for every $i$,\footnote{Or  $\NE(f)=\langle [C_E],[C_{L_1}],\dotsc,[C_{L_{m-1}}],[C_{D_m}]\rangle$ if $E\cap D_m=\emptyset$ in $X$.} while in $\w{Y}$ we have $\NE(\sigma_Y)=
 \langle [C_{D_1}],\dotsc,[C_{D_m}]\rangle$.

 Moreover $W$ is Fano. To see this, consider $f^*(-K_W)=-K_{\w{X}}+E+2(D_1+\cdots+D_m)$, and let $R$ be an extremal ray of 
$\NE(\w{X})$. If $K_{\w{X}}\cdot R\geq 0$, then there is an exceptional line
$\ell\subset\w{X}$ with $[\ell]\in R$ (see Lemma \ref{SQMFano}), and we have 
$(D_1+\cdots+D_m)\cdot\ell>0$, $-K_{\w{X}}\cdot\ell=-1$, $E\cdot\ell=0$, thus $f^*(-K_W)\cdot\ell>0$. 
Suppose instead that $R$ is $K$-negative. 
The extremal rays of $\NE(\w{X})$ that are negative for some exceptional divisor of $f$ are precisely those in $\NE(f)$, therefore if $R\not\subset\NE(f)$ we have $-K_{\w{X}}\cdot R>0$, $(E+2(D_1+\cdots+D_m))\cdot R\geq 0$, and
again $f^*(-K_W)\cdot R>0$. We conclude that $f^*(-K_W)$ is nef and $f^*(-K_W)^{\perp}\cap\NE(\w{X})=\NE(f)$, which implies that $-K_W$ is ample.
 
 On the other hand from $\w{X}$ we can also consider the remaining  $(D_1+\cdots+D_m)$-negative flips, that will be given by the extremal rays generated by $[C_{L_1}],\dotsc,[C_{L_m}]$ (which are $E$-positive), and get a SQM $\xi_2\colon \w{X}\dasharrow\wi{X}$. Now we get a birational map $\sigma\colon \wi{X}\to Z$ which is the blow-up of $m$ pairwise disjoint smooth irreducible curves $\Gamma'_1,\dotsc,\Gamma'_m\subset Z$, and $Z$ is smooth and Fano (this can be seen with a similar argument as for $W$). In $Z$ we have $E\cdot \Gamma'_i>0$ for every $i=1,\dotsc,m$. Moreover $[C_E]$ is extremal in $\NE(Z)$, and there is an associated elementary contraction of type $(3,2)$ $\alpha_0\colon Z\to W$.
 $$\xymatrix{X\ar@{-->}[r]^{\xi_1}\ar[d]_{\alpha}&{\w{X}}\ar[dr]^f\ar[d]^{\tilde{\alpha}}\ar@{-->}[r]^{\xi_2}&{\wi{X}}\ar[d]^g\ar[r]^{\sigma}&Z\ar[dl]^{\alpha_0}\\
Y\ar@{-->}[r]^{\xi_Y}&{\w{Y}}\ar[r]^{\sigma_Y}&W&
 }$$
 
 Set $g:=\alpha_0\circ\sigma\colon \wi{X}\to W$. Then $g$ is another resolution of the birational map $X\dasharrow W$, and now we have $\NE(g)=\langle [C_{D_1}],\dotsc,[C_{D_m}],[\ell_1],\dotsc,[\ell_m]\rangle$ where $\ell_i\subset\wi{X}$ is the exceptional line corresponding to the exceptional plane $L_i\subset\w{X}$.\footnote{Or  $\NE(g)=\langle [C_{D_1}],\dotsc,[C_{D_m}],[\ell_1],\dotsc,[\ell_{m-1}]\rangle$ if $E\cap D_m=\emptyset$ in $X$.} Moreover $C_E\equiv C_{D_i}+\ell_i$ for every $i$.
\end{remark}
\begin{lemma}\label{carolina}
  In the setting of Rem.~\ref{uffa}, let $i\in\{1,\dotsc,m\}$ be such that $E\cap D_i\neq\emptyset$,
  $A\subset W$ be the image of $E$, and $\Gamma_i\subset W$ the image of $D_i$, so that $A$ is an irreducible surface and $\Gamma_i$ a smooth irreducible curve contained in $W_{\reg}$. Then $A$ and $\Gamma_i$ intersect transversally at finitely many points, contained in $A_{\reg}$.
\end{lemma}
\begin{proof}
Let $S$ be a connected component of $E\cap D_i$ in $X$. Then $S\cong\mathbb{F}_1$ and if $\ell\subset\mathbb{F}_1$ is the fiber of the $\pr^1$-bundle and $e\subset\mathbb{F}_1$ is the $(-1)$-curve, we have $\ell\equiv C_E$, $e=S\cap L_i$, and $e\equiv C_{L_i}$ (notation as in Rem.~\ref{uffa}).

The surface $S$ is contained in $\dom(\xi_1)$, therefore it is mapped isomorphically to $\w{X}$, where we still denote it by $S$. Set $A_Y:=\tilde\alpha(E)\subset\w{Y}$ and $\Lambda:=\tilde\alpha(S)\subset A_Y$;
then $\Lambda$ is a connected component of $A_Y\cap \tilde\alpha(D_i)$.

We note that $S=\tilde\alpha^{-1}(\Lambda)$ and $\tilde\alpha_{|S}$ is a $\pr^1$-bundle, so that $\tilde\alpha$ has one-dimensional fibers over $\Lambda$, hence both $A_Y$ and $\w{Y}$  are smooth at every point of $\Lambda$
(Lemma \ref{singtarget}).

Moreover $\Lambda\cong\pr^1$ and $\Lambda\equiv C_{D_i}$. If $\tilde\alpha(D_i)_{|A_Y}=m\Lambda+B$ where $B$ is a divisor in $A_Y$ supported on the other connected components of $A_Y\cap \tilde\alpha(D_i)$ (note that $A_Y\cap \tilde\alpha(D_i)\subset (A_Y)_{\reg}$), we get $-1=\tilde\alpha(D_i)\cdot C_{D_i}=\tilde\alpha(D_i)\cdot \Lambda=\tilde\alpha(D_i)_{|A_Y}\cdot_{A_Y} \Lambda=m\Lambda^2$, so that $m=1$ and $\Lambda$ is a $(-1)$-curve in $A_Y$.

This shows that, in $W$, $w_0:=\sigma_Y(\Lambda)=f(S)$ is a point, and $A$ is smooth at $w_0$. Moreover $\sigma_{Y|A_Y}$ is the blow-up of the reduced point $w_0$, therefore the curve $\Gamma_i$ and $A$ intersect transversally at $w_0$.
\end{proof}  
\begin{lemma}\label{classification}
  Let $X$ be a smooth Fano $4$-fold with $\rho_X\geq 7$, not isomorphic to a product of surfaces, $X\dasharrow \w{X}$ a SQM, and $f\colon \w{X}\to Y$ a special contraction with $\dim Y=3$ and $\rho_X-\rho_Y=2$.
  Then  $\NE(f)=\langle [C_{E_1}],[C_{E_2}]\rangle$ where  $E_1,E_2\subset \w{X}$  are  fixed prime divisors of type $(3,2)$ such that
  $\N(E_i,\w{X})=\N(\w{X})$ for $i=1,2$.
\end{lemma}
\begin{proof}
  The Fano $4$-folds as in the statement are classified in \cite[Th.~5.1]{3folds}.
  The description of $\NE(f)$ is in [\emph{ibid.}, Lemma 5.6], and we keep the same notation. In particular
  let $\tilde\alpha\colon\w{X}\to \w{W}$ the elementary contraction of type $(3,2)$ with exceptional divisor $E_1$, and $S:=\tilde\alpha(E_1)\subset \w{W}$. Then 
  $\N(S,\w{W})=\N(\w{W})$ by [\emph{ibid.}, 5.67]; on the other hand
  $\ker\tilde\alpha_*\subset \N(E_1,\w{X})$ and $\tilde\alpha_*(\N(E_1,\w{X}))=
  \N(S,\w{W})$, and we conclude that $\N(E_1,\w{X})=\N(\w{X})$.
  
  Let us consider now $E_2$. If we are in case $(b)$ of [\emph{ibid.}, Lemma 5.39], then the situation for $E_1$ and $E_2$ is symmetric [\emph{ibid.}, Lemma 5.46, 5.60], and as before we conclude that $\N(E_2,\w{X})=\N(\w{X})$. If we are in case $(a)$ and the surface $S$ is singular, then $\tilde\alpha\colon\w{X}\to\w{W}$ has a $2$-dimensional fiber $F_0$ (Lemma \ref{singtarget}). We have $E_2\cdot\NE(\tilde\alpha)>0$, hence $\dim(F_0\cap E_2)\geq 1$ and $\ker\tilde\alpha_*\subset\N(E_2,\w{X})$. On the other hand $T:=\tilde\alpha(E_2)\supset S$, so that $\tilde\alpha_*(\N(E_2,\w{X}))=\N(T,\w{W})\supset\N(S,\w{W})=\N(\w{W})$, and we conclude again that $\N(E_2,\w{X})=\N(\w{X})$.

  We are left with the case where we are in case $(a)$ and $S$ is smooth.
  Then by [\emph{ibid.}, Th.~5.1, Lemma 5.68, 5.70] there is a SQM $\w{W}\dasharrow\wi{W}:=\Bl_{q_0,\dotsc,q_r}\pr^4$ (see \S\ref{fanomodel}), with $r=\rho_X-3$, and $S\subset \w{W}$ is the transform of a cubic scroll $A\subset\pr^4$ containing the points $q_0,\dotsc,q_r$. Moreover $Y\cong\Bl_{p_1,\dotsc,p_r}\pr^3$ and there is a $\pr^1$-bundle $\pi\colon \w{W}\to Y$, induced by the projection $\pr^4\dasharrow\pr^3$ from $q_0$, such that $f=\pi\circ\tilde\alpha$.

  The points $q_0,\dotsc,q_r\in\pr^4$ are in general linear position [\emph{ibid.}, 5.63], therefore up to exchanging some points we can assume that $q_0$ does not belong to the $(-1)$-curve in $A\cong\mathbb{F}_1$. Then the projection of $A$ from $q_0$ is a smooth quadric surface $B_0\subset\pr^3$, and $B:=\pi(S)=\Bl_{p_1,\dotsc,p_r}B_0$ is a smooth surface with $\rho_B=2+r=\rho_X-1$. Finally $E_2\subset\w{X}$ is the transform of $T=\pi^{-1}(B)\subset\w{W}$, which is a $\pr^1$-bundle over $B$, thus $T$ is smooth with $\rho_T=\rho_X=\rho_{\w{W}}+1$. We have $E_2\cong T$ because
  both $T$ and $S$ are smooth,  $S\subset T$, and $\tilde\alpha$ is the blow-up along $S$.

  Let us consider the restriction $r\colon \Nu(\w{W})\to\Nu(T)$.
  Since 
  $\N(T,\w{W})=\N(S,\w{W})=\N(\w{W})$, dually $r$  is injective.
  
  We claim that the class  $[S]\in\Nu(T)$ is not in the image $r(\Nu(\w{W}))$. Let us first show that this implies the statement.  Consider the commutative diagram: $$\qquad\qquad\qquad\xymatrix{{\Nu(\w{W})}\ar[r]^{r}\ar[d]_{\tilde\alpha^*}&{\Nu(T)}
    \ar[d]^{\ph:=(\tilde\alpha^*)_{|\Nu(T)}=(\tilde\alpha_{|E_2})^*}\\
{\Nu(\w{X})}\ar[r]^{r'}&{\Nu(E_2)}
}$$
where the horizontal maps are the restrictions. Since $r$ is injective and $\ph$ is an isomorphism, $\ph( r(\Nu(\w{W})))$ is a hyperplane contained in $\im r'$. Moreover $\ph([S])$ is outside this hyperplane and  $\ph([S])=[E_{1|E_2}]\in\im r'$, and we conclude that $r'$ is surjective, hence an isomorphism as $\rho_{E_2}=\rho_X$, and finally that $\N(E_2,\w{X})=\N(\w{X})$.

\medskip

  To prove our claim, note that the canonical class $K_A$ is not the restriction of a class in $\Nu(\pr^4)$.
  Let $\sigma\colon \wi{W}\to \pr^4$ be the blow-up of $q_0,\dotsc,q_r$, and $S'\subset\wi{W}$  the transform of $A$. Then $\sigma_{|S'}$ is the blow-up of $A$ along $q_0,\dotsc,q_r$, and $K_{S'}=(\sigma_{|S'})^*(K_A)+\sum_{i=0}^r e_i$, where $e_i\subset S'$ are the $(-1)$-curves, that are restrictions to $S'$ of the exceptional divisors of $\sigma$. Hence we see that $K_{S'}$ is not the restriction of a class in $\Nu(\wi{W})$. Moreover $S\subset\w{W}$ is contained in the open subset where the SQM $\w{W}\dasharrow \wi{W}$ is an isomorphism \cite[Lemma 5.48]{3folds}, therefore  $K_{S}$ is not the restriction of a class in $\Nu(\w{W})$.
We have $S\subset T$ and $K_S=(K_T+S)_{|S}$, hence $K_T+S\not\in r(\Nu(\w{W}))$. On the other hand $K_T=(K_{\w{W}}+T)_{|T}$, and we conclude that $[S]\not\in r(\Nu(\w{W}))$.
\end{proof}
\section{Quasi-elementary rational contractions onto surfaces}\label{seczucca}
\noindent In this section we prove the following result (Th.~\ref{zuccaintro} from the Introduction), that will be needed in the next section; it improves Theorems \ref{brasile} and \ref{scala} in the case of a quasi-elementary rational contraction onto a surface.
\begin{thm}\label{zucca}
  Let $X$ be a smooth Fano $4$-fold with $\rho_X\geq 7$, not isomorphic to a product of surfaces, and  $f\colon X\dasharrow S$ a quasi-elementary rational contraction onto a surface.  
  Then $\rho_X-\rho_S\leq 3$, and if $\rho_X-\rho_S>1$, then $f$ factors as $X\stackrel{g}{\dasharrow}Y\stackrel{h}{\dasharrow} S$ where $\dim Y=3$,
  $g$ and $h$ are rational contractions, $g$ is special, and $h$
is elementary. 
\end{thm}
The condition $\rho_X\geq 7$ is necessary, see Ex.~\ref{es1}; see moreover \S\ref{exX} for an example of a Fano $4$-fold $X$ with $\rho_X=3$ and a quasi-elementary contraction $X\to\pr^2$, with general fiber $\Bl_{2\mskip1mu\pts}\pr^2$, that does not factor through a $3$-fold. We also note that the bound $\rho_X-\rho_S\leq 3$ is sharp, as the following examples show.
\begin{example}\label{es3}
  Let $X$ be one of the two families of Fano $4$-folds with $\rho_X=7$ described in \S\ref{newex}. Then there are a
   quasi-elementary rational contraction  $f\colon X\dasharrow S$
  with $\dim S=2$ and 
    $d_f=\rho_X-\rho_S=3$, and
      a rational contraction $f'\colon X\dasharrow\pr^1$ with $d_{f'}=4$.
\end{example}  
\begin{example}
  Let $X$ be the Fano model of the blow-up of $\pr^4$ at $r$ general points, with $r\in\{6,7,8\}$, so that $\rho_X\in\{7,8,9\}$ (see \S\ref{fanomodel}). Then there is a
  a quasi-elementary rational contraction $f\colon X\dasharrow S$
  with $\dim S=2$ and   $d_f=\rho_X-\rho_S=2$. 
\end{example}
Theorems \ref{zucca} and \ref{CS} allow to classify complete the case where $\rho_X-\rho_S= 3$, as follows.
\begin{corollary}
 Let $X$ be a smooth Fano $4$-fold with $\rho_X\geq 7$, not isomorphic to a product of surfaces, and  $f\colon X\dasharrow S$ a quasi-elementary rational contraction onto a surface with  
  $\rho_X-\rho_S\geq 3$. Then $\rho_X-\rho_S= 3$ and $X$ is as in Th.~\ref{CS}$(ii)$.
\end{corollary}
\begin{prg}[\em Overview of the proof of Th.~\ref{zucca}]\label{overviewzucca}
We first show a preliminary property in Lemma \ref{quadric}, then we prove
Th.~\ref{zucca}.
The proof is quite long and will take the whole section.

The bound $\rho_X-\rho_S\leq 3$ follows easily from Th.~\ref{scala}; the main work is to prove that $f$ factors through a $3$-fold, when it is not elementary. So suppose that $\rho_X-\rho_S\in\{2,3\}$ and let $\tilde{f}\colon \w{X}\to S$ be a $K$-negative resolution of $f$. We analyse the possible factorizations of $f$ in elementary contractions, using Lemmas \ref{montenero} and \ref{quasiel2}, and show that either the statement holds, or $\tilde{f}$ has two distinct factorizations $\w{X}\stackrel{\sigma_i}{\to}Y_i\stackrel{g_i}{\to} S$ where $\sigma_i$ is elementary of type $(3,2)$ and $g_i$ is quasi-elementary and $K$-negative, for $i=1,2$. Set $E_i:=\Exc(\sigma_i)\subset\w{X}$.

The surface $S$ is smooth del Pezzo with $\rho_S\geq 4$; for a $(-1)$-curve $\Gamma\subset S$,  $D:=\tilde{f}^*\Gamma$ is a fixed prime divisor in $\w{X}$. Using results from Section \ref{prelfixed} we show that $D$ is of type $(3,1)^{\sm}$, and that up to exchanging $E_1$ and $E_2$ we have $E_1\cdot C_D=0$, $E_2\cdot C_D=E_2\cdot C_{E_1}=1$, and $E_1\cdot C_{E_2}>1$. Then by applying Lemma \ref{quadric} we deduce that the general fiber of $g_2$ is $\pr^1\times\pr^1$.
Thanks to  this, if $\rho_X-\rho_S=3$, we factor again $g_2$ using Lemma \ref{quasiel2}, and we get the statement. 

We finally prove that in this setting we cannot have $\rho_X-\rho_S=2$,
proceeding by contradiction. We show that $Y_1$ is smooth and $g_1\colon Y_1\to S$ is a $\pr^2$-bundle, while
the general fiber of $\tilde{f}$ is $\Bl_{2\mskip1mu\pts}\pr^2$.
Then we consider a birational map $\alpha\colon S\to\pr^2$, and show that  $\alpha\circ g_1\colon Y_1\to \pr^2$ also factors as $Y_1\stackrel{\text{SQM}}{\dasharrow} \w{Y}_1 \stackrel{\beta}{\to} \pr^2\times\pr^2
\stackrel{h}{\to}\pr^2$, where $h$ is the first projection and $\beta$ is the blow-up of $\rho_X-3$ lines contained in distinct fibers of $h$ (compare with \S\ref{Y}).

Let  $B_2\subset \pr^2\times\pr^2$ be the transform of $E_2\subset\w{X}$; then $B_2$ is a prime divisor of bidegree $(2,1)$, containing the lines blown-up by $\beta$. We also consider $A\subset B_2$ the surface  image of $E_1\subset\w{X}$.
A careful analysis of the possible configurations for $B_2$ and $A$ shows that this case gives a contradiction.
\end{prg}
  \begin{lemma}\label{quadric}
    Let $X$ be a smooth Fano $4$-fold with $\rho_X\geq 7$ and
 $X\dasharrow\w{X}$ a SQM. 
    Let  $f\colon\w{X} \to
    S$ be a $K$-negative contraction onto a surface that  factors as $g\circ\sigma$, where $\sigma\colon\w{X}\to Y$ is a divisorial elementary contraction  of type $(3,2)$.
    $$\xymatrix{{\w{X}}\ar@/^1pc/[rr]^f\ar[r]_{\sigma}&{Y}\ar[r]_g&S
    }$$
    Suppose also that there is a $(-1)$-curve $\Gamma\subset S_{\reg}$ such that $D:=f^*(\Gamma)$ is a fixed prime divisor of type $(3,1)^{\sm}$ with $E\cdot C_D>0$, where $E:=\Exc(\sigma)\subset\w{X}$.

    Then the general fiber of $g$ is $\pr^1\times\pr^1$.
\end{lemma}
\begin{proof}
Since $f(C_E)=\{\pt\}$, we have $D\cdot C_E=0$, therefore $D$ and $E$ are adjacent by Lemma \ref{EF}.
  
Let $\tau\colon S\to S_1$ be the contraction of the $(-1)$-curve $\Gamma$, and set $p:=\tau(\Gamma)\in S_1$; the composition $\tau\circ f\colon\w{X}\to S_1$ contracts ${D}$ to the point $p$.   We run an MMP for ${D}$, relative to $\tau\circ f$: this means that we consider $D$-negative small extremal rays in $\NE(\tau\circ f)$ and their flips, until we find a $D$-negative divisorial extremal ray in the relative cone. In this way 
we get a commutative diagram:
$$\xymatrix{{\w{X}}\ar[d]_{{f}}\ar@{-->}[r]^{\ph}&{\wi{X}}\ar[r]^{\alpha}&{Z}\ar[dl]^h\\
S\ar[r]^{\tau}&{S_1}&  }$$
where $\ph$ is a sequence of $D$-negative flips, $\alpha$ is a divisorial elementary contraction with exceptional divisor the transform $\wi{D}\subset\wi{X}$ of $D$, and $h$ is a contraction of fiber type. Note that both $\wi{D}$ and the indeterminacy locus of $\ph^{-1}$ are  contained in the fiber $(h\circ\alpha)^{-1}(p)$.

By Th.-Def.~\ref{fixed}, $Z$ is smooth and is a SQM of a smooth Fano $4$-fold, and  
 $\alpha$ is the blow-up of a smooth curve $C\subset Z$.

 Since $[C_E]$ is extremal in $\NE(\w{X})$, by Lemma \ref{iso32} $E$ is  disjoint from all exceptional lines in $\w{X}$.
 Let us factor $\ph$ as a sequence of $D$-negative flips:
$$\w{X}=\w{X}_1\stackrel{\sigma_1}{\dasharrow} \w{X}_2\stackrel{\sigma_2}{\dasharrow}\cdots\stackrel{\sigma_{r-1}}{\dasharrow}
\w{X}_r=\wi{X}.$$
Let $\wi{E}\subset\wi{X}$ and $E_i\subset\w{X}_i$ be the transforms of $E$.
We show by induction on $i=1,\dotsc,r$ that $E_i$ is disjoint from all exceptional lines in $\w{X}_{i}$. For $i=1$ we have already remarked this.

Suppose that the claim holds in $\w{X}_i$, and consider the flip $\sigma_i\colon \w{X}_i\dasharrow\w{X}_{i+1}$. 
If $\sigma_i$ is not $K$-negative, then its indeterminacy locus is a disjoint union of exceptional lines (see Lemma \ref{SQMFano}), and $X_{i+1}$ will have less exceptional lines than $X_i$, still disjoint from $E_{i+1}$. If instead $\sigma_i$ is $K$-negative, then its indeterminacy locus is a finite disjoint union of exceptional planes $L_j\subset X_i$ (see Lemma \ref{kawamata}) such that $D\cdot C_{L_j}<0$, because $\sigma_i$ is $D$-negative. Then $E_i\cap L_j=\emptyset$ by Lemma \ref{disjoint}, thus in $X_{i+1}$ the divisor $E_{i+1}$ stays disjoint from the new exceptional lines in the indeterminacy locus of $\sigma_i^{-1}$.

We conclude that $\wi{E}\subset\wi{X}$ is disjoint from all exceptional lines in $\wi{X}$, and
by  Lemma \ref{iso32}
the class $[C_{\wi{E}}]$ generates an extremal ray of $\NE(\wi{X})$, which is contracted by the map $h\circ\alpha\colon\wi{X}\to S_1$. Let $\hat\sigma\colon\wi{X}\to\wi{Y}$ be the associated contraction, of type $(3,2)$ with exceptional divisor $\wi{E}$ (see Th.-Def.~\ref{fixed}).

 By Lemma \ref{square} we have a commutative diagram:
$$\xymatrix{{\w{X}}\ar@/_1pc/[dd]_{f} \ar@{-->}[r]^{\ph}\ar[d]^{\sigma}&{\wi{X}}\ar[d]^{\hat\sigma}\ar[r]^{\alpha}&{Z}\ar[d]_{\sigma'}\ar@/^1pc/[dd]^h&\\
  Y \ar@{-->}[r]^{\ph_Y} \ar[d]^g&{\wi{Y}}\ar[r]^{\alpha'}&W\ar[d]_{g_1}\\
S\ar[rr]_{\tau}&&{S_1}}$$
where
$\sigma'$ is a divisorial elementary contraction of type $(3,2)$ with $\Exc(\sigma')=\alpha(\wi{E})$,  
$\alpha'$ is the blow-up of a smooth point $q\in W$ with
$\Exc(\alpha')=\hat\sigma(\wi{D})=:D'\subset\wi{Y}$, and $g_1$ is a contraction of fiber type.

Moreover $\ph_Y$ is a SQM such that $\hat\sigma(\wi{E})\subset\dom(\ph_Y^{-1})$ (by Th.-Def.~\ref{fixed}), therefore
$\wi{Y}\smallsetminus\dom(\ph_Y^{-1})$ is contained in the fiber $(g_1\circ\alpha')^{-1}(p)$, and the same for $D'$. We conclude that
$g\colon Y\to S$, $g_1\circ \alpha'\colon\wi{Y}\to S_1$, and $g_1\colon W\to S_1$, all have isomorphic general fibers.

We show that $F\cong\pr^1\times\pr^1$ where $F\subset W$ is the general fiber of $g_1$; this gives the statement.

By Lemma \ref{singtarget} both $\wi{Y}$ and $W$ have at most isolated, locally factorial, and terminal singularities; in particular $F$ is smooth.

We show that $g_1$ is $K$-negative. Let $\ell\subset W$ be an irreducible curve with $-K_W\cdot\ell\leq 0$, and $\ell_{\wi{Y}}\subset\wi{Y}$ its transform.
Note that $\ell_{\wi{Y}}\not\subset D'$ because $\alpha'(D')=\{q\}$, thus $D'\cdot\ell_{\wi{Y}}\geq 0$ and
$0\geq -K_W\cdot\ell=(-K_{\wi{Y}}+3D')\cdot\ell_{\wi{Y}}\geq -K_{\wi{Y}}\cdot\ell_{\wi{Y}}$.
By Lemma \ref{target} $\ell_{\wi{Y}}$ is an exceptional line, disjoint from $\hat{\sigma}(\wi{E})$, and $K_{\wi{Y}}\cdot\ell_{\wi{Y}}=1$ also gives $D'\cdot\ell_{\wi{Y}}= 0$,
thus $\ell_{\wi{Y}}\cap D'=\emptyset$.

 Then
 the transform $\ell_{\wi{X}}\subset\wi{X}$ is again an exceptional line,
disjoint from both 
 $\wi{E}$ and $\wi{D}$.
It is not difficult to see that $\ell_{\wi{X}}$ must be contained in $\dom(\ph^{-1})$, so that its transform $\ell_{\w{X}}\subset\w{X}$ 
is an exceptional line, disjoint from $D$. Since $f$ is $K$-negative, $f(\ell_{\w{X}})\subset S$ is a curve. Moreover it cannot be the $(-1)$-curve $\Gamma$, otherwise $\ell_{\w{X}}$ would be contained in $D$. We conclude that $\tau(f(\ell_{\w{X}}))=g_1(\ell)\subset S_1$ is a curve, and $g_1$ is $K$-negative, therefore its general fiber
$F$ is a smooth del Pezzo surface.

Suppose by contradiction that $F\not\cong\pr^1\times\pr^1$. Then $F$ is covered by a family of curves of anticanonical degree $3$, thus $W$ is covered by a family of curves of anticanonical degree $3$, contracted by $g_1$. If $C\subset W$ is a curve of the family containing the point $q$ blown-up by $\alpha'\colon\wi{Y}\to W$, then every irreducible component $C_0$ of $C$ satisfies $-K_W\cdot C_0>0$, because $g_1$ is $K$-negative.
Let $C_0$ be an irreducible component of $C$ containing $q$, and
$\wi{C}_0\subset\wi{Y}$ its transform.
If $-K_W\cdot C_0=1$ we get $-K_{\wi{Y}}\cdot\wi{C}_0\leq -2$, and if $-K_W\cdot C_0=3$ we get  $-K_{\wi{Y}}\cdot\wi{C}_0=0$ or $-K_{\wi{Y}}\cdot\wi{C}_0\leq -3$,
in any case 
contradicting Lemma \ref{target}. Therefore it must be  $-K_W\cdot C_0=2$, $-K_{\wi{Y}}\cdot\wi{C}_0\leq -1$, and
 $\wi{C}_0$ is an exceptional line by Lemma \ref{target}. However we must also have $C=C_0\cup C_1$ with $C_1$ irreducible, $-K_W\cdot C_1=1$, and $q\not\in C_1$. Then the transform of $C_1$ in $\wi{Y}$ is an irreducible curve of anticanonical degree one intersecting an exceptional line, contradicting 
Lemma \ref{target}.
Therefore $F\cong\pr^1\times\pr^1$.
\end{proof}
\begin{proof}[Proof of Th.~\ref{zucca}]
 Let us assume that $f$ is not elementary, so that $\rho_X-\rho_S>1$.
  
Since $f$ is quasi-elementary, we have $d_f=\rho_X-\rho_S$, and $S$ is a smooth del Pezzo surface (Lemma \ref{delpezzo}). By Th.~\ref{scala} we have $\rho_X-\rho_S\leq 4$, and if 
$\rho_X-\rho_S=4$, then it should be $S\cong\pr^2$, which is impossible because $\rho_X\geq 7$. Therefore we have  $2\leq \rho_X-\rho_S\leq 3$, in particular $\rho_S\geq \rho_X-3\geq 4$. 

We are left to show that $f$ factors as $X\stackrel{g}{\dasharrow}Y\stackrel{h}{\dasharrow} S$ where $\dim Y=3$, $g$ is special, and $h$ is elementary.
Let us consider a $K$-negative resolution $\tilde{f}\colon\w{X}\to S$ of $f$.
\begin{prg}\label{2}
Either we get the statement, or $\w{X}$ contains two fixed prime divisors $E_1,E_2$ of type $(3,2)$ such that $C_{E_1}$ and $C_{E_2}$ are contracted by $\tilde{f}$ and $(E_1\cdot C_{E_2})(E_2\cdot C_{E_1})>1$.
\end{prg}
\begin{proof} 
  We consider first the case $\rho_X-\rho_S=2$, and apply Lemma \ref{quasiel2} to $f$. If we get $(i)$, we have the statement. If we get $(ii)$,
  then $f$ admits two factorizations $X\stackrel{\alpha_i}{\dasharrow} W_i\stackrel{h_i}{\to}S$ where  $\alpha_i$ is a divisorial elementary rational contraction with exceptional locus dominating $S$ for $i=1,2$, so that $\alpha_i$ is of type $(3,2)$. Moreover $\alpha_1^*\Nu(W_1)\neq
  \alpha_2^*\Nu(W_2)$, hence $\Exc(\alpha_1)\neq\Exc(\alpha_2)$ by Th.-Def.~\ref{fixed}.

 Let $E_i\subset \w{X}$ be the transform of $\Exc(\alpha_{i})$; then $E_i$ is
 a fixed prime divisor of type $(3,2)$, $E_1\neq E_2$, and $C_{E_i}$ is contracted by $\tilde{f}$. Moreover $\R_{\geq 0}[C_{E_i}]$ is an extremal ray of $\NE(\tilde{f})$
 by Lemma \ref{extremal}, and finally 
 $\NE(\tilde{f})=\langle [C_{E_1}],[C_{E_2}]\rangle$. Since $f$ is quasi-elementary, we have $\dim(\NE(\tilde{f})\cap\mov(X))=2$ (see \cite[Prop.~2.22]{eff}), hence
 $(E_1\cdot C_{E_2})(E_2\cdot C_{E_1})>1$ by \cite[Lemma 4.6(a)]{fibrations}.

\smallskip
  
  Suppose now that $\rho_X-\rho_S=3$; we proceed similarly. By Lemma \ref{montenero} $f$ factors as
  $$X\stackrel{g}{\dasharrow} X_2\stackrel{h}{\to} S$$ where $g$ is elementary, either divisorial or of fiber type, and $h$ is of fiber type (because $f$ is quasi-elementary). Moreover $h$ is quasi-elementary by Lemma \ref{h} and 
$\rho_{X_2}-\rho_S=2$, therefore $\dim X_2\neq 3$, and we conclude that $\dim X_2=4$ and 
  $g$ is birational divisorial with $\Exc(g)$ dominating  $S$. 

  Now we apply Lemma \ref{quasiel2} to $h$. If we get $(i)$, then $h$ factors as
  $X_2\stackrel{g_2}{\dasharrow} Y\stackrel{h_2}{\to} S$ where $\dim Y=3$ and both $g_2$ and $h_2$ are elementary of fiber type. Since $\Exc(g)$ dominates $S$, its image in $Y$ is a divisor, hence the composition $g_2\circ g\colon X\dasharrow Y$ is special, and we get  the statement.

  If instead we get $(ii)$, for $i=1,2$ we have a factorization of $h$ as
  $$X_2\stackrel{\alpha_i}{\dasharrow} W_i\stackrel{h_i}{\to} S$$
   where  $\alpha_i$ is a divisorial elementary rational contraction with exceptional locus dominating $S$ for $i=1,2$, so that $\alpha_i$ is of type $(3,2)$.
  Let $E_i\subset \w{X}$ be the transform of $\Exc(\alpha_{i})\subset X_2$. Similarly as in the previous case $\rho_X-\rho_S=2$ we see that $E_i$ is a fixed prime divisor of type $(3,2)$, $C_{E_i}$ is contracted by $\tilde{f}$, and $E_1\neq E_2$. Let moreover  $G\subset \w{X}$ be the transform of $\Exc(g)$, which is again a fixed prime divisor of type $(3,2)$.

  We note that $G$ and $E_i$ are adjacent, because they are both contracted by the composite birational map $\w{X}\stackrel{\text{SQM}}{\dasharrow} X\stackrel{g}{\dasharrow}X_2\stackrel{\alpha_i}{\dasharrow}W_{i}$, where $W_i$ is $\Q$-factorial and $\rho_X-\rho_{W_i}=2$ (see e.g.\ \cite[proof of Lemma 4.2]{small}). 
  By Lemma \ref{EF} we conclude that $G\cdot C_{E_i}=E_i\cdot C_{G}=0$ for $i=1,2$.

 Let us show that
$({E}_{1}\cdot C_{{E}_{2}})({E}_{2}\cdot C_{{E}_{1}})> 1$.
Let $\beta\colon\wi{X}\to W_{1}$ be a resolution of the above map $\w{X}\dasharrow W_{1}$; with a slight abuse of notation, we still denote by $G$ and $E_i$ their transforms in $\wi{X}$.
Then $G,E_1$ are the exceptional divisors for $\beta$, and
$\beta(E_2)\subset W_{1}$ is a divisor that dominates $S$ under $h_1\colon W_1\to S$, therefore
 $\beta(E_2)\cdot \NE(h_{1})>0$. Moreover
$\beta^*(\beta(E_2))=E_2+(E_2\cdot C_{E_1})E_1+(E_2\cdot C_{G})G=E_2+(E_2\cdot C_{E_1})E_1$, and  $[\beta_*(C_{E_2})]\in\NE(h_{1})$.
Hence:
$$0<\beta(E_2)\cdot \beta_*(C_{E_2})=\bigl(E_2+(E_2\cdot C_{E_1})E_1\bigr)
\cdot C_{E_2}=(E_1\cdot C_{E_2})(E_2\cdot C_{E_1})-1,$$
and we conclude that  $(E_{1}\cdot C_{E_{2}})(E_{2}\cdot C_{E_{1}})> 1$.
\end{proof}
\begin{prg}\label{torteria}
  Let $i\in\{1,2\}$. Since $\tilde{f}$ is $K$-negative and  contracts $C_{E_i}$, by Lemma \ref{extremal} 
  $\R_{\geq 0}[C_{E_i}]$ is an extremal ray of $\NE(\tilde{f})$,
and we have a
factorization:
$$\xymatrix{ &{\w{X}}\ar[dl]_{\sigma_1}\ar[dr]^{\sigma_2}\ar[d]^{\tilde{f}} &\\
  {Y_1}\ar[r]_{g_1}&S&{Y_2}\ar[l]^{g_2}}$$
  where $\sigma_i$ is the contraction
  of $\R_{\geq 0}[C_{E_i}]$, $Y_i$ has at most isolated terminal and locally factorial singularities (Lemma \ref{singtarget}), $\sigma_i(E_i)$ dominates $S$, and $g_i$ is quasi-elementary (Lemma \ref{h}).

  We also note that $g_i$ is $K$-negative: suppose otherwise; then by Lemma \ref{target} $g_i$ must contract some exceptional line of  $Y_i$, disjoint from $\sigma_i(E_i)$, and this contradicts the $K$-negativity of $\tilde{f}$.

   Let $F_i\subset Y_i$ be a general fiber of $g_i$, so that $F_i$ is a smooth del Pezzo surface.
\end{prg}
\begin{prg}\label{lucia}
  Let   $\Gamma\subset S$ be a $(-1)$-curve (recall that $\rho_S\geq 4$), and set $D:=\tilde{f}^{-1}(\Gamma)$; then $D$ is a prime divisor because $\tilde{f}$ is quasi-elementary (Lemma \ref{equivalent}), and it is fixed because $\Gamma$ is. Moreover $\tilde{f}^{*}(\Gamma)=mD$ for some $m\in\Z_{>0}$, thus $D\cdot \gamma=0$ for every $\gamma\in\NE(\tilde{f})$, and $D$ is the pullback of some divisor in $S$; this implies that $m=1$ and  $D=\tilde{f}^{*}(\Gamma)$.

  We show that $D$ is of type $(3,1)^{\sm}$.
We have $\tilde{f}(D)=\Gamma$, therefore $\dim\N(D,\w{X})\leq 1+\rho_X-\rho_S\leq 4$
(Lemma \ref{easy}).
On the other hand we have
 $\delta_X\leq 1$ by Th.~\ref{starting}, and we deduce from Lemma \ref{dim32} that  $D$ cannot be of type $(3,2)$.

Moreover for $i=1,2$ we have $D\cdot C_{E_i}=0$, hence $D$ is adjacent to $E_i$ by Lemma \ref{EF}; in $\w{X}$ we have ${D}\cap {E}_i\neq\emptyset$ because ${E}_i$ dominates $S$, thus $\dim({D}\cap{E}_i)=2$ and since the indeterminacy locus of the map $\w{X}\dasharrow X$ is one-dimensional (Lemma \ref{SQMFano}$(a)$), we have $\w{D}\cap \w{E}_i\neq\emptyset$ in $X$, where $\w{D},\w{E}_i\subset X$ are the transforms of $D,E_i$ respectively.

Then $D$ cannot be of type $(3,0)^{\sm}$ by Lemma \ref{30sm}. Moreover $D$ cannot be of type $(3,0)^Q$, otherwise by Lemma \ref{30Q} we get a contradiction with  $(E_1\cdot C_{E_2})(E_2\cdot C_{E_1})>1$. We conclude that $D$ is of type $(3,1)^{\sm}$.
 
We apply Lemma \ref{31sm}. Since $(E_1\cdot C_{E_2})(E_2\cdot C_{E_1})>1$, up to exchanging $E_1$ and $E_2$ we have:
\stepcounter{thm}
\begin{equation}\label{freddo}
E_{1}\cdot C_D=0,\quad E_{2}\cdot C_D=E_{2}\cdot C_{E_{1}}=1,\quad \text{and}\quad E_{1}\cdot C_{E_{2}}\geq 2.\end{equation}
 Note that $E_1\cdot C_{E_2}$ and $E_2\cdot C_{E_1}$ do not depend on $D$,
therefore \eqref{freddo} must hold for 
the pullback $D$ of \emph{every} $(-1)$-curve $\Gamma\subset S$.

  Since $E_2\cdot C_D>0$, we can apply
  Lemma \ref{quadric} to $\tilde{f}=g_2\circ\sigma_2$ and $D$, and deduce that  $F_{2}\cong\pr^1\times\pr^1$, where $F_2\subset Y_2$ is a general fiber of $g_2$.
\end{prg}
\begin{prg} Suppose that $\rho_X-\rho_S=3$, so that
  $\rho_{Y_2}-\rho_S=2$, and apply Lemma \ref{quasiel2} to $g_2\colon Y_2\to S$.
  We show that $(ii)$ cannot happen. Otherwise, $g_2$ factors as  $Y_2\stackrel{\alpha}{\dasharrow} W\stackrel{h}{\to} S$ where $\alpha$ is elementary and divisorial with exceptional locus dominating $S$. Consider
a resolution $\alpha'\colon Y_2'\to W$ of $\alpha$. Then there is a sequence of flips $\psi\colon Y_2\dasharrow Y_2'$, relative to $g_2$, such that 
$\alpha=\alpha'\circ\psi$.
$$\xymatrix{{Y_2}\ar[d]_{g_2}\ar@{-->}[r]^{\psi}\ar@{-->}[rd]^{\alpha}&{Y_2'}\ar[d]^{\alpha'}\\
  S&W\ar[l]_h
  }$$
By Lemma \ref{target} each flip in $\psi$ has locus contained in the smooth locus of the $4$-fold, and is either $K$-negative, or $K$-positive. Therefore the locus of each flip is a finite disjoint union of exceptional planes or lines, and such locus must be contracted to finitely many points by $g_2$. We conclude that $h\circ\alpha'$ sends the indeterminacy locus $Y_2'\smallsetminus\dom(\psi^{-1})$ to a finite set of points, therefore $\psi$ is an isomorphism on the general fiber $F_2$, and 
$h\circ\alpha'$ has general fiber $F_2'\cong \pr^1\times\pr^1$. However this is impossible, because $\alpha'$ must restrict to a non-trivial birational contraction on $F_2'$. 

Therefore we must get $(i)$ from Lemma \ref{quasiel2},
  and $g_2$ factors as
  $Y_2\stackrel{\beta}{\dasharrow} Z\stackrel{k}{\to} S$
  with $\dim Z=3$ and both maps are elementary  of fiber type. 
Since $E_2$ dominates $S$, its image in $Z$ is a surface, $\beta\circ\sigma_2\colon \w{X}\dasharrow Z$ is special,
   and we get the statement.
  \end{prg}
\begin{prg}
  We are left with the case where $\rho_X-\rho_S=2$ and $g_1$ and
  $g_2$ are elementary. We show that this case does not happen, as it leads to a contradiction.

  Note that $\NE(\tilde{f})=\langle [C_{E_1}],[C_{E_2}]\rangle$ (see \ref{torteria}).

 For $i=1,2$ let $r_i$ be the degree of $(g_i)_{| \sigma_i(E_i)}\colon \sigma_i(E_i)\to S$. Then $\sigma_i$ blows-up $F_i$ in $r_i$ points, so if $F\subset\w{X}$ is a general fiber of $\tilde{f}$, then $F$ is a del Pezzo surface with
  \stepcounter{thm}
  \begin{equation}\label{rhoF}
    \rho_F=\rho_{F_i}+r_i.
    \end{equation}

  We show that:
  \stepcounter{thm}
  \begin{equation}\label{rel}
    r_1E_{2}\cdot C_{E_{1}}=r_2E_{1}\cdot C_{E_{2}}.
    \end{equation}
Indeed write $(E_{i})_{|F}=\Gamma_{i1}+\cdots+\Gamma_{ir_i}$; then $\Gamma_{ij}\equiv C_{E_i}$ for every $j$, and
\begin{gather*}
  r_1E_{2}\cdot C_{E_{1}}=E_2\cdot (\Gamma_{11}+\cdots+\Gamma_{1r_1})=(E_{2})_{|F}\cdot_F(\Gamma_{11}+\cdots+\Gamma_{1r_1})\\=(\Gamma_{21}+\cdots+\Gamma_{2r_2})\cdot_F
  (E_1)_{|F}=r_2E_{1}\cdot C_{E_{2}}.
  \end{gather*}
  
  Now \eqref{rhoF}, \eqref{rel}, \eqref{freddo}, and $\rho_{F_2}=2$, give:
$$2+r_{2}=\rho_F=\rho_{F_{1}}+r_{1}=\rho_{F_{1}}+r_{2}E_{1}\cdot C_{E_{2}}\geq 1+2r_2,$$
 and we get  $r_{2}=1$, $F_{1}\cong\pr^2$, $\rho_F=3$,  and
  $r_{1}=E_{1}\cdot C_{E_{2}}=2$.
\end{prg}
\begin{prg}  
  We show that:
  \stepcounter{thm}
\begin{equation}
  \label{wX}
  -K_{\w{X}}+\tilde{f}^*(-K_S)=2E_{1}+3E_{2}.\end{equation}

  Indeed
  $(-K_{\w{X}}-2E_{1}-3E_{2})\cdot C_{E_{i}}=0$ for $i=1,2$, thus 
$-K_{\w{X}}-2E_{1}-3E_{2}=\tilde{f}^*M$ for some divisor $M$ on $S$.

By Lemma \ref{lifting}, for every $(-1)$-curve $\Gamma\subset S$ there exists an exceptional line $\ell\subset \w{X}$ such that $\tilde{f}(\ell)=\Gamma$. We have $E_i\cdot\ell=0$ for $i=1,2$ by Lemma \ref{iso32}, thus
$$-1=(-K_{\w{X}}-2E_{1}-3E_{2})\cdot\ell=\tilde{f}^*M\cdot\ell=M\cdot \tilde{f}_*(\ell),$$
and this implies that 
$M\cdot\Gamma=-1=K_S\cdot\Gamma$ for every $(-1)$-curve $\Gamma\subset S$. Since $S$ is a del Pezzo surface with $\rho_S=\rho_X-2\geq 5$, $\N(S)$ is generated by classes of $(-1)$-curves, and we conclude that $M=K_S$.
\end{prg}
\begin{prg}\label{usseaux}
 Set $B_2':=\sigma_1(E_2)\subset Y_1$. Applying the pushforward $(\sigma_1)_*$ to \eqref{wX} we get:
  \stepcounter{thm}
\begin{equation}\label{formula}
  -K_{Y_1}+g_1^*(-K_S)=3B_2'.
\end{equation}

Restricting to $F_1\cong\pr^2$ gives $\ol_{Y_1}(B_2')_{|F_1}\cong\ol_{\pr^2}(1)$,
in particular $B_2'$ is $g_1$-ample. If $A$ is an ample divisor on $S$, then for $r\gg 0$ the divisor $B_2'+rg_1^*(A)$ is ample on $Y_1$, and we still have $\ol_{Y_1}(B_2'+rg_1^*(A))_{|F_1}\cong\ol_{\pr^2}(1)$.
By \cite[Prop.~3.2.1]{beltrametti_sommese}, this implies that $g_1$ is a $\pr^2$-bundle and
 $Y_1$ is smooth, so that $Y_1$ is a SQM of a smooth Fano $4$-fold $Y'_1$ with $\rho_{Y_1'}\geq 6$ (see Th.-Def.~\ref{fixed}).
\end{prg}
\begin{prg}\label{points}
  Let us choose a birational morphism $\alpha\colon S\to\pr^2$, which blows-up the points $p_1,\dotsc,p_m\in\pr^2$ with  $m=\rho_S-1=\rho_X-3\geq 4$.
  For $i\in\{1,\dotsc,m\}$ let $\Gamma_i\subset S$ be the $(-1)$-curve over $p_i$, and
  set $D_i:=\tilde{f}^*\Gamma_i\subset\w{X}$. Then $D_i$ is a  fixed prime divisor of type $(3,1)^{\sm}$ with $D_i\cdot C_{E_1}=D_i\cdot C_{E_2}=E_1\cdot C_{D_i}=0$ and $E_2\cdot C_{D_i}=1$, for every $i=1,\dotsc,m$ (see \ref{lucia}). Moreover for $i\neq j$ we have $D_i\cap D_j=\emptyset$, therefore $D_i\cdot C_{D_j}=0$. By Rem.~\ref{uffa} this implies that $\langle [E_1],[D_1],\dotsc,[D_m]\rangle$ is a fixed face of $\Eff(\w{X})\cong\Eff(X)$, of dimension $\rho_X-2$.

 For $i\in\{1,\dotsc,m\}$ set $D_i':=\sigma_1(D_i)\subset Y_1$; the divisors $D_1',\dotsc,D_m'$ are contracted to points by $\alpha\circ g_1$. We run a MMP in $Y_1$ for $D_1'+\cdots+D_m'$, relative to $\alpha\circ g_1$, and get
 a commutative diagram: $$\xymatrix{{\w{X}}\ar[dr]_{\tilde{f}}\ar[r]^{\sigma_1}&{Y_{1}}\ar[d]^{g_1}\ar@{-->}[r]^{\xi}&{\w{Y}_1}\ar[r]^{\beta}&W\ar[dl]^h\\
&    S\ar[r]^{\alpha}&{\pr^2}&
  }$$
  where $\xi$ is a SQM, $\beta$ is birational with exceptional divisors the transforms $D_1'',\dotsc,D_m''\subset\w{Y}_1$ of $D_1',\dotsc,D_m'\subset Y_1$, and $h\colon W\to\pr^2$ is an elementary contraction with general fiber $F_W\cong\pr^2$. Moreover it follows from 
Rem.~\ref{uffa} that
  $\beta$ is the blow-up of $m$ pairwise disjoint smooth curves $C_1,\dotsc,C_m\subset W$, and $W$ is smooth and a SQM of a smooth Fano $4$-fold.

  Let $\w{B}_2'\subset\w{Y}_1$ and $B_2\subset W$ be the transforms of $E_{2}$, and set $H=h^*\ol_{\pr^2}(1)\in\Pic(W)$.
 We have $-K_S=\alpha^*(-K_{\pr^2})-\sum_{i=1}^m\Gamma_i$, and by \eqref{formula}
  $-K_{Y_1}+g_1^*\alpha^*(-K_{\pr^2})-\sum_{i=1}^mD_i'=3B_2'$, which in $\w{Y}_1$ gives
  $-K_{\w{Y}_1}+3\beta^*(H)-\sum_{i=1}^mD_i''=3\w{B}_2'$.
  Finally via the pushforward $\beta_*$ we get:
   \stepcounter{thm}
\begin{equation}
  \label{W}
-K_W=3 (B_2-H).
\end{equation}

This implies that
  $B_{2|F_W}\cong\ol_{\pr^2}(1)$, therefore using again  \cite[Prop.~3.2.1]{beltrametti_sommese} as in \ref{usseaux} we get
  $W=\pr_{\pr^2}(\ma{E})$, where $\ma{E}$ is a rank $3$ vector bundle on $\pr^2$.
  Recall that $W$ is a SQM of a smooth Fano $4$-fold. Moreover, since $-K_W$ is divisible by $3$ in $\Pic(W)$,
  $W$ cannot contain exceptional lines nor exceptional planes, so the second elementary contraction of $W$ cannot be small (see Lemma \ref{SQMFano} and Lemma \ref{kawamata}), and $W$ is Fano of index
    $3$. This also implies that $c_1(\ma{E})\equiv 0\mod 3$, and by \cite{szurekwisnnagoya}
   we conclude that $W\cong\pr^2\times\pr^2$ . We can assume that $h\colon \pr^2\times\pr^2\to\pr^2$ is the first projection.

Let us consider the curves $C_1,\dotsc,C_m\subset \pr^2\times\pr^2$ blown-up by $\beta$.    
 We  have
 that $h(C_i)=p_i\in\pr^2$ with $p_i\neq p_j$ for $i\neq j$, so that $C_i$ is a smooth curve in $h^{-1}(p_i)=\pr^2$. If $L\subset h^{-1}(p_i)$ is a general line, we have $-K_{\pr^2\times\pr^2}\cdot L=3$ and $L$ must meet $C_i$, hence $D_i''\cdot\w{L}>0$, where $\w{L}\subset\w{Y}_1$ is 
its transform. On the other hand $\w{L}$ cannot be an exceptional line, because $\w{Y}_1$ can contain at most finitely many of them (see Lemma \ref{SQMFano}), therefore  $-K_{\w{Y}_1}\cdot \w{L}>0$, which implies that $D_i''\cdot\w{L}=1$ 
and  $C_i\cdot_{\pr^2} L=1$, so that $C_i$ is itself a line in $h^{-1}(p_i)=\pr^2$.

We have therefore shown that
$Y_1$ is a SQM of a blow-up of $\pr^2\times\pr^2$ along $m$ lines contained in different fibers of the first projection (see Ex.~\ref{Y}); in particular $\w{Y}_1\smallsetminus\dom(\xi^{-1})$ is the disjoint union of $m$ exceptional planes, given by the transforms of $h^{-1}(p_1),\dotsc,h^{-1}(p_m)$, and $Y_1\smallsetminus\dom\xi$ is the disjoint union of $m$ exceptional lines.
 \end{prg}
 \begin{prg}\label{frecciar}
  Consider the surface $\sigma_1(E_1)\subset B_2'\subset Y_1$, and let $\w{A}\subset\w{B}_2'\subset\w{Y}_1$ and $A\subset B_2\subset \pr^2\times\pr^2$ be its transforms.
 
  Let $i\in\{1,\dotsc,m\}$. It follows from Lemma \ref{carolina} that, whenever $A\cap C_i\neq\emptyset$, $A$ and $C_i$ intersect transversally at finitely many points, where $A$ is smooth.
  
From \eqref{W} we see that $B_2$ has bidegree $(2,1)$ in $\pr^2\times\pr^2$,
hence $h_{|B_2}\colon B_2\to\pr^2$ is generically a $\pr^1$-bundle.
Moreover $C_i\subset B_2$, because $E_2\cdot C_{D_i}>0$ (see \ref{points}); on the other hand $B_2$ cannot contain the whole fiber $h^{-1}(p_i)$, otherwise $\w{B}_2'\subset\w{Y}_1$ would contain an exceptional plane in the indeterminacy locus of $\xi^{-1}$, $\w{B}_2\subset{Y}_1$ would intersect an exceptional line, and $E_2\subset\w{X}$ would intersect an exceptional line too, a contradiction (see Lemmas  \ref{target} and \ref{iso32}).  Therefore $C_i=(h_{|B_2})^{-1}(p_i)$, $h_{|B_2}$ is a $\pr^1$-bundle over $p_i$, and
 $C_i\subset (B_2)_{\reg}$.
  
 Finally we note that since $E_1\cdot C_{E_2}=2$, we have $A\cdot_{B_2}C_i=2$, and we conclude that $A$ and $C_i$ intersect at exactly two points.

We set $C_i=\{p_i\}\times L_i$ with $L_i\subset\pr^2$ a line.
  \end{prg}
  \begin{prg}\label{roma}
    Let $i,j,k\in\{1,\dotsc,m\}$ with $i<j<k$; we show that $L_i\cap L_j\cap L_k=\emptyset$ (in particular this also shows that $L_i\neq L_j$). By contradiction, suppose otherwise, and let $q\in L_i\cap L_j\cap L_k$.

  Let $h'\colon \pr^2\times\pr^2\to\pr^2$ be the second projection,  and 
  in $(h')^{-1}(q)=\pr^2$ consider a general conic $\Gamma$ through $(p_i,q)$, $(p_j,q)$, and $(p_k,q)$ (note that $p_i,p_j,p_k$ are not collinear in $\pr^2$, because $\Bl_{p_i,p_j,p_k}\pr^2$ is a del Pezzo surface, see \ref{points}). Then $-K_{\pr^2\times\pr^2}\cdot \Gamma=6$ and if $\w{\Gamma}\subset\w{Y}_1$ is its transform, we have $D_i''\cdot\w{\Gamma}>0$, $D_j''\cdot\w{\Gamma}>0$,  $D_k''\cdot\w{\Gamma}>0$, hence
  $-K_{\w{Y}_1}\cdot\w{\Gamma}\leq 0$, so that $\w{\Gamma}$ is an exceptional line by Lemma \ref{SQMFano}. However this is impossible, because $\w{Y}_1$ contains at most finitely many exceptional lines (again by Lemma \ref{SQMFano}), while $\Gamma$ moves in a positive dimensional family.
  \end{prg}
\begin{prg}\label{torq}
  Set $\Lambda:=h'(A)\subset\pr^2$, and let $i\in\{1,\dotsc,m\}$. We  have $h'(C_i\cap A)\subset L_i\cap\Lambda$, in particular $L_i\cap \Lambda\neq \emptyset$ because $C_i\cap A\neq\emptyset$. Moreover $\Lambda$ cannot be a point, otherwise it should lie in $L_1\cap L_2\cap L_3$ which is empty by \ref{roma} (recall that $m\geq 4$). 

  Let $q\in L_i\cap \Lambda$; we show that $(p_i,q)\in A$. Suppose otherwise, and let $L$ be a line in $(h')^{-1}(q)$ containing $(p_i,q)$ and meeting $A$.
  Then $-K_{\pr^2\times\pr^2}\cdot L=3$, and if $\w{L}\subset\w{Y}_1$ is its transform, we have $-K_{\w{Y}_1}\cdot\w{L}\leq 1$, $\w{L}\cap\w{A}\neq\emptyset$, and $\w{L}\not\subset\w{A}$,  contradicting Lemma \ref{target}.
  Therefore $(p_i,q)\in A$.

 Since we also have $(p_i,q)\in C_i$, this implies that $h'(C_i\cap A)=L_i\cap \Lambda$, hence $L_i$ intersects $\Lambda$ set-theoretically in two points (see \ref{frecciar}). 
  We conclude that $\Lambda$ is an irreducible curve of degree at least $2$.
\end{prg}
\begin{prg}\label{tauZ}
  Let $\tau\colon Z\to\pr^2\times\pr^2$ be the blow-up of $A$, with exceptional divisor
  $E_1'\subset Z$; we note that 
  $Z$ is a smooth Fano $4$-fold with $\rho_Z=3$ (see Rem.~\ref{uffa}), and there is a birational contraction $\ph\colon X\dasharrow Z$.  

Let $q\in\Lambda$ be a general point, $\Gamma\subset \pr^2\times\{q\}$ a general line, and $\w{\Gamma}\subset Z$ its transform. Note that $A$ has at most isolated singularities (Lemma \ref{singtarget}), therefore $(h'_{|A})^{-1}(q)\subset\pr^2\times\{q\}$ is a smooth curve; let $d$ be its degree.
Then $E_1'\cdot\w{\Gamma}=d$, hence $-K_Z\cdot\w{\Gamma}=3-d\geq 1$, and  we get $d\in\{1,2\}$.

  Let $G_Z\subset Z$ be the transform of $G:=\pr^2\times\Lambda\subset\pr^2\times\pr^2$. We have
  $G\cdot \Gamma=0$ and $\tau^*(G)=G_Z+mE_1'$ with $m\geq 1$, therefore $G_Z\cdot\w{\Gamma}=-md<0$.

  If $d=2$, we see that $G_Z$ is a fixed prime divisor covered by rational curves of anticanonical degree one. Consider the birational contraction $\ph\colon X\dasharrow Z$, and let $G_X\subset X$ be the transform of $G_Z$. Then the indeterminacy locus of $\ph^{-1}\colon Z\dasharrow X$ must be disjoint from the transform of $\pr^2\times\{q\}$ (see \cite[Prop.~2.10]{3folds}), so that $G_X$ 
  is again a fixed prime divisor covered by rational curves $C$ with $-K_X\cdot C=1$. By \cite[Lemma 2.18]{blowup} $G_X$ must be of type $(3,2)$, but this contradicts  $G_X\cdot C=-2m\leq -2$.
\end{prg}
\begin{prg}\label{lambda}
  Therefore $d=1$, and $A$ is $\pr^1$-bundle over $\Lambda$; since $A$ has isolated singularities, this implies that
  $\Lambda$ is smooth.
  
Recall that $B_2$ has bidegree $(2,1)$ in $\pr^2\times\pr^2$,
hence  $h'_{|B_2}\colon B_2\to\pr^2$ is generically a conic bundle.

We notice that $h'_{|B_2}$ does not have $2$-dimensional fibers, otherwise the fibers of $h_{|B_2}$, as curves in $\pr^2$, should all have a common point.  However the curves $C_i$ are fibers of $h_{|B_2}$ (see \ref{frecciar}), and 
by \ref{roma} there is no common point on the lines $L_i$'s. Therefore $h'_{|B_2}\colon B_2\to \pr^2$ is  a conic bundle; let $\Delta\subset\pr^2$ be its discriminant locus, parametrising singular conics.

Since the general fiber of $h'_{|A}$ is a line, we have $\Lambda\subset\Delta$, and $(h'_{|B_2})^*(\Lambda)=A+A'$ where $A'$ is a prime divisor that dominates $\Lambda$ (possibly $A'=A$), such that the general fiber of $h'_{|A'}$ is a line too.

This implies that $B_2$ must be singular, otherwise $h'_{|B_2}$ would be an elementary contraction of a smooth Fano $3$-fold, and $A\cdot C=0$ for the general fiber $C\subset B_2$, so that $A$ should be the pullback of a divisor in $\pr^2$, but it is not. Since $B_2$ is singular,  $h_{|B_2}\colon B_2 \to\pr^2$ must have a $2$-dimensional fiber $F=\{p_0\}\times\pr^2\subset B_2$, and all the fibers of the conic bundle $h'_{B_2}$  contain the point $p_0$.
\end{prg}  
 \begin{prg}
Recall that $\deg\Lambda\geq 2$ by \ref{torq}.
   Suppose first that $\Lambda$ is a conic, and recall that $C_1\subset (B_2)_{\reg}$ and $A\cdot_{B_2} C_1=2$ (see \ref{frecciar}). Then
 $2=\Lambda\cdot_{\pr^2} L_1=(h'_{|B_2})^*(\Lambda)\cdot_{B_2} C_1=(A+A')\cdot_{B_2} C_1=2+A'\cdot_{B_2} C_1$, therefore $A'\cdot_{B_2} C_1=0$. We also have $C_1\not\subset A'$, because $L_1\not\subset\Lambda$, and we conclude that $A'\cap C_1=\emptyset$, namely $A'\cap (h_{|B_2})^{-1}(p_1)=\emptyset$, and $p_1\not\in h(A')$.

 Then $h(A')\subsetneq\pr^2$, so that $A'=L_0\times\Lambda$ where $L_0$ is a fixed line. 
 Note that $B_2$ is the total space of a net of conics; let $\Gamma\subset\pr^2$ be the subvariety parametrising conics of the net that have $L_0$ as a component. It is not difficult to see that
 $\Gamma$ must be linear, and since it contains the conic $\Lambda$, we get $\Gamma=\pr^2$, but this is impossible because $B_2$ is irreducible.
Therefore we conclude that
  $\deg\Lambda>2$.
\end{prg}  
 \begin{prg}
If  $\Delta\subsetneq\pr^2$,  then $\Delta$ is a cubic curve and $\Lambda\subset\Delta$, therefore $\Lambda=\Delta$ and $\Delta$ is smooth (see \ref{lambda}).
A local computation shows that this implies that $B_2$ is smooth too (see e.g.\ \cite[Rem.~3.9 and its proof]{3folds}), but we have already excluded this (see \ref{lambda}).

Therefore $\Delta=\pr^2$ and every fiber of $h'_{|B_2}$ is singular. Then it is not difficult to see that either all conics of the net have a fixed component, or they are all singular at the same point $p_0$ (see for instance \cite[Table 1]{nets}), and
since $B_2$ is irreducible, we must be in the second case.
The fibers of $h'_{|A}$ are components of fibers of $h'_{|B_2}$, hence they all contain $p_0$, and $F\cap A$ is a section of $h'_{|A}$, where $F=\{p_0\}\times\pr^2$. Therefore $F\cap A\cong\Lambda$, and
 $F\cap A$ has degree $\geq 3$. Let $L \subset F$ be a general line, and $\w{L}\subset Z$ its transform (see \ref{tauZ}). Then $E_1'\cdot\w{L}\geq 3$ and $-K_{\pr^2\times\pr^2}\cdot L=3$, thus $-K_Z\cdot\w{L}\leq 0$, a contradiction.
  This concludes the proof of Th.~\ref{zucca}.
\qedhere
\end{prg}  
\end{proof}
\section{Fixed divisors of type $(3,2)$}\label{sec32}
\noindent In this section we show the following result, which implies Th.~\ref{sharp32intro} from the Introduction.
\begin{thm}\label{sharp32}
  Let $X$ be a smooth Fano $4$-fold with $\rho_X\geq 7$, not isomorphic to a product of surfaces, and $E\subset X$ a fixed prime divisor of type $(3,2)$ such that $\N(E,X)\subsetneq\N(X)$.

  Then $\rho_X\leq 9$,
  $X$ has an elementary rational contraction onto a  $3$-fold, and $E$ is of type $(3,2)^{\sm}$. More precisely there is a commutative diagram 
 \stepcounter{thm}
\begin{equation}\label{cavour} \xymatrix{X\ar[d]_{\sigma'}\ar@{-->}[r]^{\xi}&{\w{X}}\ar[r]^{\tilde{f}}\ar[d]_{\sigma}&Y\ar[d]^g\\
W'\ar@{-->}[r] & W\ar[r]^h&Z
}
\end{equation}
where:
\begin{enumerate}[$\bullet$]
\item $W'$ is a smooth Fano $4$-fold and $\sigma'$ is  of type $(3,2)^{\sm}$,
  with exceptional divisor $E$;
\item $\xi$ and $W'\dasharrow W$ are SQM's, and $E\subset\dom(\xi)$;
\item $Y$ is a smooth weak Fano\footnote{Namely $-K_Y$ is nef and big.} $3$-fold and $\tilde{f}$ is an elementary contraction and a conic bundle, finite on $\xi(E)$;
\item
  $Z$ is a smooth del Pezzo surface.
  \end{enumerate}
\end{thm}
The bound $\rho_X\leq 9$ improves the bound $\rho_X\leq 12$ shown in \cite[Prop.~5.32]{blowup};
we do not know if it is sharp.
Let us also note that, up to flops, there are only three possibilities for the $3$-fold $Y$, see \cite[Cor.~6.4]{3folds}.
\begin{prg}[\em Overview of the proof of Th.~\ref{sharp32}]\label{overviewsharp32}
We consider the cone $\Mov(X)^{\vee}\subset\N(X)$, dual of the cone of movable divisors, that has $\langle [C_E]\rangle$ as a face of dimension one (see Rem.~\ref{movdual}). The main step is to show that there exists a two-dimensional face $\langle [C_E]\rangle+\alpha$ of $\Mov(X)^{\vee}$ with $\alpha\subset\mov(X)$ (in particular this implies that $X$ has an elementary rational contraction of fiber type, see again Rem.~\ref{movdual}).
To prove this, we work by contradiction, and show that otherwise we can produce a special rational contraction of fiber type $\psi\colon X\dasharrow T$, with $\rho_X-\rho_T=2$, and contracting $C_E$. Then we reach a contradiction by applying Lemma \ref{classification} when $\dim T=3$, and Th.~\ref{zucca} when $\dim T=2$.

Once we have the one-dimensional face $\alpha$ of $\Mov(X)^{\vee}$ as above, this gives an elementary rational contraction of fiber type $X\dasharrow Y$. If $X\to W'$ is the elementary contraction of type $(3,2)$ with exceptional divisor $E$, since  $\langle [C_E]\rangle+\alpha$ is a face of $\Mov(X)^{\vee}$, we see that there are elementary rational contractions $Y\dasharrow Z$ and $W'\dasharrow Z$ onto the same target. Then by analysing the maps, and using that $\N(E,X)\subsetneq\N(X)$, we get the statement.
\end{prg}
\begin{proof}[Proof of Th.~\ref{sharp32}]
Since $X$ is not a product of surfaces and $\rho_X\geq 7$, we have $\delta_X\leq 1$ by Th.~\ref{starting}. Therefore
$\delta_X=1$ and $\N(E,X)$ is a hyperplane in $\N(X)$.

  Consider the cone $\Mov(X)^{\vee}\subset\N(X)$, dual of the cone of movable divisors;  we have 
 $\Mov(X)^{\vee}=\langle [C_D]\rangle_{D\text{ fixed}}+\mov(X)$,
 and $\langle[C_E]\rangle$ is a one-dimensional face of $\Mov(X)^{\vee}$, see Rem.~\ref{movdual}.
\begin{prg}\label{6.3}
 There exists a one-dimensional face $\alpha$ of $\Mov(X)^{\vee}$,  contained in $\mov(X)$, such that 
 $\langle[C_E]\rangle+\alpha$ is a face of $\Mov(X)^{\vee}$.
\end{prg}
\begin{proof}
  We proceed by contradiction, and assume that every $2$-dimensional face of 
  $\Mov(X)^{\vee}$ containing $\langle[C_E]\rangle$ has the form
  $\langle[C_E],[C_D]\rangle$ for some fixed prime divisor $D\subset X$. 

  Then
  as in \cite[proof of Prop.~5.32]{blowup} we show that there exists a
 $2$-dimensional face  $\langle[C_E],[C_{E'}]\rangle$ of 
  $\Mov(X)^{\vee}$ where $E'$ is a
  fixed prime divisor of type $(3,2)$ such that  $[C_{E'}]\not\in\N(E,X)$, $E\cap E'\neq\emptyset$, $E\cdot C_{E'}>0$, and $E'\cdot C_{E}>0$. In particular $\langle [C_{E}],[C_{E'}]\rangle\cap\mov(X)\neq\{0\}$ by \cite[Lemma 4.6(a)]{fibrations}.

The $2$-dimensional face  $\langle[C_E],[C_{E'}]\rangle$ of 
$\Mov(X)^{\vee}$ corresponds by duality to the codimension $2$ face
 $\eta:=\Mov(X)\cap C_{E}^{\perp}\cap C_{E'}^{\perp}$
of $\Mov(X)$. Note that $\eta\subset\partial \Eff(X)$, because $\langle [C_{E}],[C_{E'}]\rangle\cap\mov(X)\neq\{0\}$. Let $\eta_0\in\MCD(X)$ be a cone of dimension $\rho_X-2$ with $\eta_0\subset\eta$, and
 $\psi\colon X\dasharrow T$ the  rational contraction such that $\eta_0=\psi^*(\Nef(T))$; then  $\rho_T=\rho_X-2\geq 5$, $\psi$ is of fiber type, and $\dim T\in\{2,3\}$. 

\medskip
 
    We show that $\psi$ is special.
    
    Let $\tau_{\psi}$ be the minimal face of $\Eff(X)$ containing $\eta$ (see Def.-Rem.~\ref{dimfiber}). By Lemma \ref{eubea}, we need to show that 
    $\tau_{\psi}\cap\Mov(X))=\eta$.

    Note that $\tau_{\psi}\cap\Mov(X)$ is a proper face of $\Mov(X)$ and contains $\eta$, therefore if  $\tau_{\psi}\cap\Mov(X))\supsetneq\eta$, then  $\tau_{\psi}\cap\Mov(X)$ must be a facet of $\Mov(X)$ containing $\eta$, and contained in 
 $\partial \Eff(X)$.
  
 On the other hand $\eta$ is contained in exactly two facets of $\Mov(X)$, that are
  $\Mov(X)\cap C_{E}^{\perp}$ and $\Mov(X)\cap C_{E'}^{\perp}$, and both of them are not contained in $\partial \Eff(X)$.  Therefore $\tau_{\psi}\cap\Mov(X)$ cannot be a facet, it must be
 $\tau_{\psi}\cap\Mov(X)=\eta$, and $\psi$ is special.

 Let
$$\tilde{\psi}\colon\w{X}\la T$$
be a $K$-negative resolution of $\psi$, and let $\w{E},\w{E}'\subset\w{X}$ be the transforms of $E,E'$. Since $\eta_0\subset C_E^{\perp}\cap C_{E'}^{\perp}$, $\tilde{\psi}$ contracts both $C_{\w{E}}$ and $C_{\w{E}'}$, and by Lemma \ref{extremal} we have $\NE(\tilde{\psi})=\langle [C_{\w{E}}],[C_{\w{E}'}] \rangle$. In particular, since $\NE(\tilde{\psi})$ does not contain small extremal rays, there are no flips relative to $\tilde{\psi}$, and $\tilde{\psi}\colon\w{X}\to T$ is the unique resolution of $\psi$.

 We have $\dim\N(\w{E},\w{X})=\dim\N(E,X)=\rho_X-1$ by Lemma \ref{dim32}.
If $\dim T=3$, then $\N(\w{E},\w{X})\subsetneq\N(\w{X})$ contradicts Lemma \ref{classification}, therefore
$\dim T=2$.

Since $\psi$ is special, the images of $\w{E},\w{E}'$ in $T$ are either a curve, or $T$ itself. 
If  $\tilde{\psi}(\w{E})$ is a curve in $T$, then  $\dim\N(\w{E},\w{X})\leq 1+\rho_X-\rho_T=3$ (see Lemma \ref{easy}),  a contradiction, and
 similarly for $\w{E}'$. Therefore $\tilde{\psi}(\w{E})=\tilde{\psi}(\w{E}')=T$  and 
 $\tilde{\psi}$ is quasi-elementary, because the general fiber will contain curves $C_{\w{E}}$, $C_{\w{E}'}$.

 Then  Th.~\ref{zucca} implies that
 $\psi$ factors as the composition of two rational elementary contractions $X\stackrel{f}{\dasharrow}Y\stackrel{g}{\dasharrow} T$
 where $\dim Y=3$. Up to composing with a SQM of $Y$, we can assume that $g$ is regular. 
 By taking a resolution $\hat{f}\colon\wi{X}\to Y$ of $f$, we get also a resolution $g\circ\hat{f}\colon\wi{X}\to T$ of $\psi$, hence it must be $\wi{X}=\w{X}$ and $\tilde{\psi}=g\circ\hat{f}$. This implies that $\NE(\hat{f})$ is an extremal ray of $\NE(\tilde{\psi})$, which 
 gives a contradiction, because  $\NE(\tilde{\psi})=\langle [C_{\w{E}}],[C_{\w{E}'}] \rangle$.
\end{proof}
\begin{prg}
Let $\langle[C_E]\rangle+\alpha$ be a $2$-dimensional face  of $\Mov(X)^{\vee}$, with $\alpha$ one-dimensional face of $\mov(X)$ (see \ref{6.3}).
 Dually this means that both $C_E^{\perp}\cap\Mov(X)$ and $\alpha^{\perp}\cap\Mov(X)$ are facets of $\Mov(X)$, and they intersect in a face $\eta$ of $\Mov(X)$ of dimension $\rho_X-2$.

Let us choose two cones $\eta_0,\tau\in\MCD(X)$ as follows: $\eta_0\subset\eta$ of dimension $\rho_X-2$, and $\tau\subset \alpha^{\perp}\cap\Mov(X)$  of dimension $\rho_X-1$ and containing $\eta_0$.

Let $f\colon X\dasharrow Y$ be the elementary rational contraction with $\tau=f^*(\Nef(Y))$; $f$ is of fiber type, because $\alpha\subset \partial\mov(X)$, hence
 $\tau\subset\partial\Eff(X)$.
Since $\eta_0\subset\tau$, there is a contraction $g\colon Y\to Z$ such that $(g\circ f)^*(\Nef(Z))=\eta_0$ (see Rem.~\ref{factor}), and $g$ is elementary because $\dim\eta_0=\rho_X-2$.   Note that since $\rho_Z=\rho_X-2\geq 5$, we have $\dim Z\in \{2,3\}$.

Let $X\stackrel{\xi}{\dasharrow}\w{X}\stackrel{\tilde{f}}{\to} Y$ be a $K$-negative resolution of $f$
and 
$\ph\colon\wi{X}\to Z$ a $K$-negative resolution of $g\circ f$.
We have $\eta_0\subset C_E^{\perp}$, thus $\ph$ contracts $C_E$, moreover 
$\ph$ is $K$-negative: by Lemma \ref{extremal} it factors through an elementary contraction $\sigma\colon \wi{X}\to W$, of type $(3,2)$, with exceptional divisor the transform $\wi{E}\subset\wi{X}$ of $E$.
We have a diagram:
$$\xymatrix{{\w{X}}\ar[d]_{\tilde{f}}\ar@{-->}[rr]&&{\wi{X}}\ar[dl]_{\ph}\ar[d]^{\sigma}\\
  Y\ar[r]_g&Z&W\ar[l]^{h}
  }$$
We also note that $\dim\N(\wi{E},\wi{X})=\dim\N(E,X)=\rho_X-1$ by Lemma \ref{dim32}.
\end{prg}
\begin{prg}
 Since $\sigma$ is of type $(3,2)$ and $h$ is elementary of fiber type, both $W$ and $Z$ are $\Q$-factorial. Moreover $\tilde{f}$ is elementary of fiber type, $Y$ is $\Q$-factorial too, and $g$ cannot be small.

 We show that $g$ is of fiber type. By contradiction, if $g$ is divisorial, then there is a prime divisor $D\subset\w{X}$ such that $\tilde{f}(D)=\Exc(g)$ and hence $\codim g(\tilde{f}(D))>1$. If $\wi{D}\subset\wi{X}$ is the transform of $D$, then $h(\sigma(\wi{D}))=g(\tilde{f}(D))$, thus it must be $\codim \sigma(\wi{D})>1$ and $\wi{D}=\wi{E}$. On the other hand $\dim \ph(\wi{D})\leq 1$, therefore $\dim\N(\wi{D},\wi{X})\leq 3$ (see Lemma \ref{easy}), a contradiction.

 We conclude that $g$ is of fiber type, $\dim Y=3$, $\dim Z=2$, and the composition $g\circ \tilde{f}$ is quasi-elementary. This implies that $\sigma(\wi{E})$ dominates $Z$; moreover $Z$ is a smooth del Pezzo surface (Lemma \ref{delpezzo}).

 In particular $f\colon X\dasharrow Y$ is a rational contraction onto a $3$-fold, and we get $\rho_X\leq 9$ by Th.~\ref{CS}.
\end{prg}
\begin{prg}
Let us consider the $2$-dimensional cone $\NE(\ph)=\NE(\sigma)+R$, where $R$ is an extremal ray of $\NE(\wi{X})$.
We have $\N(\wi{E},\wi{X})\subsetneq \N(\wi{X})$ while $\ph_*(\N(\wi{E},\wi{X}))=\N(\ph(\wi{E}),Z)=\N(Z)$, therefore $\ker\ph_*\not\subset\N(\wi{E},\wi{X})$. On the other hand $\NE(\sigma)\subset \N(\wi{E},\wi{X})$, therefore we must have
$R\not\subset \N(\wi{E},\wi{X})$. We also have $\wi{E}\cdot\NE(\sigma)<0$, and there are curves $C\subset\wi{X}$ contracted by $\ph$ such that $\wi{E}\cdot C>0$, therefore $\wi{E}\cdot R>0$.

Let $F\subset\wi{X}$ be a non-trivial fiber of the contraction of $R$, so that $\N(F,\wi{X})=\R R$. Then $\wi{E}\cap F\neq\emptyset $ (because $\wi{E}\cdot R>0$) and $\dim(\wi{E}\cap F)=0$ (because $R\not\subset \N(\wi{E},\wi{X})$), hence $\dim F=1$.
Since $\ph$ is $K$-negative, this implies that $R$ cannot be small (Lemma \ref{kawamata}). Then there are no small rays in $\NE(\ph)$, therefore $\w{X}=\wi{X}$, $\ph=g\circ\tilde{f}$, $R=\NE(\tilde{f})$, and we get a diagram as \eqref{cavour}. 
Moreover
$\tilde{f}$ is finite on $\w{E}$ because $\NE(\tilde{f})\not\subset \N(\w{E},\w{X})$, hence
every fiber of $\tilde{f}$ has dimension one, so that $\tilde{f}$ is a conic bundle and $Y$ is smooth (see \cite[Th.~(1.2)]{wisn}), and $Y$ is weak Fano by \cite[Cor.~6.4]{3folds}.

We show that $\sigma$ is of type $(3,2)^{\sm}$. By Lemma \ref{singtarget}, we have to show that every fiber of $\sigma$ has dimension $\leq 1$. If $\sigma$ had some $2$-dimensional fiber $F_0$, then $\tilde{f}(F_0)\subset Y$ would be a $2$-dimensional fiber of $g$, which is impossible. This also implies that $W$ is smooth.

By Th.-Def.~\ref{fixed}, there is an elementary contraction $\sigma'\colon X\to W'$ of type $(3,2)$ with exceptional divisor $E$, and $W'$ is Fano; moreover $\sigma'$ and $\sigma$ are locally isomorphic, therefore $E\subset\dom\xi$ where $\xi\colon X\dasharrow\w{X}$ is the SQM, $\sigma'$ is of type $(3,2)^{\sm}$ too, and $W'$ is smooth. This concludes the proof of Th.~\ref{sharp32}.
\qedhere
\end{prg}
\end{proof}

We conclude this section with some technical lemmas that follow from our previous results, and are needed in the sequel.
\begin{lemma}[refined version of \cite{small}, Lemma 7.2]\label{torino}
Let $X$ be a smooth Fano $4$-fold with $\rho_X\geq 7$, and $D\subset X$ a fixed prime divisor of type $(3,0)^Q$. Suppose that there are two fixed prime divisors $E_1,E_2$, of type $(3,2)$, both adjacent to $D$. 
 Then  $X$ has a rational contraction onto a $3$-fold.
  \end{lemma}
  \begin{proof}
    Note that $X$ is not isomorphic to a product of surfaces, because it has a fixed prime divisor of type $(3,0)^Q$, hence $\delta_X\leq 1$ by Th.~\ref{starting}. Then
    \cite[proof of Lemma 7.2]{small} shows that either
 $\N(E_i,X)\subsetneq\N(X)$ for some $i\in\{1,2\}$, or $X$ has
   a rational contraction onto a $3$-fold, thus the statement follows from Th.~\ref{sharp32}.
    \end{proof}
\begin{lemma}[refined version of \cite{small}, Lemma 5.13]\label{lori}
  Let $X$ be a smooth Fano $4$-fold with $\rho_X\geq 7$, not isomorphic to a product of surfaces, and $f\colon X\dasharrow S$ a rational contraction onto a surface with $\rho_S=1$. Suppose that there is a unique prime divisor in $X$ contracted to a point by $f$. Then   there is a fixed prime divisor $E\subset X$ of type $(3,2)$ such that $\N(E,X)\subsetneq\N(X)$.
\end{lemma}  
\begin{proof}
  As in \cite[proof of Lemma 5.13]{small} we see that there is a special rational contraction $g\colon X\dasharrow T$ with $\dim T=\rho_T=2$.  We have $\rho_X-\rho_T\geq 5$ while $d_g\leq 4$ by Th.~\ref{scala} (see Def.-Rem.~\ref{dimfiber}), hence   $g$ is not quasi-elementary, 
  and the statement follows from [\emph{ibid.}, Lemma 5.12].
 \end{proof}
 \begin{proposition}[refined version of \cite{small}, Prop.~7.1]\label{sabri}
    Let $X$ be a smooth Fano $4$-fold with $\rho_X\geq 7$, $D$ a fixed prime divisor of type $(3,1)^{\sm}$ or $(3,0)^Q$, and $X\dasharrow \w{X}\stackrel{\sigma}{\to} Y$ the associated contraction as in Th.-Def.~\ref{fixed}. Suppose that $Y$ contains a nef prime divisor $H$ covered by a family of rational curves of anticanonical degree one, and such that $H\cap\sigma(\Exc\sigma)=\emptyset$.
    Then   one of the following holds:
    \begin{enumerate}[$(i)$]
     \item 
       $X$ has     a (regular) contraction onto a $3$-fold, that sends $D$ to a point;
     \item there is a fixed prime divisor $E\subset X$ of type $(3,2)$ such that $\N(E,X)\subsetneq\N(X)$.
       \end{enumerate}
  \end{proposition}  
  \begin{proof}
    We note that since $X$ contains a fixed prime divisor of type $(3,1)^{\sm}$ or $(3,0)^Q$, $X$ is not a product of surfaces, and $\delta_X\leq 1$ by Th.~\ref{starting}.
    
    We follow \cite[proof of Prop.~7.1]{small}, which shows the existence of a commutative diagram:
    $$\xymatrix{X\ar@{-->}[r]^{\ph}\ar[dr]_h&{\w{X}}\ar[r]^{\sigma}&Y\ar[dl]^g\\
      &Z&
      }$$
where $g$ and $h$ are contractions, and $h(D)=g(\sigma(\Exc(\sigma)))=\{\pt\}$.

    We show that $g$ and $h$ are of fiber type.
    If $D$ is $(3,1)^{\sm}$, then $Y$ is a smooth Fano $4$-fold with $\rho_Y\geq 6$, and $\delta_Y\leq 2$ by Lemma \ref{deltaY}. Now if $[C]\in\mov(Y)$, then $g$ is of fiber type. Otherwise as in [\emph{ibid.}, proof of Prop.~7.1] we see that $[C]$ generates an extremal ray of type $(3,2)$ of $\NE(Y)$, with locus $E_1$, and there is another
 fixed prime divisor $E_2\subset Y$, of type $(3,2)$, such that $E_2\cdot C>0$ and  $[C_{E_2}]\in\NE(g)$. Then $E_1\cdot C_{E_2}>0$ \cite[Lemma 2.2(b)]{cdue}, hence $[C]+[C_{E_2}]\in\mov(Y)$ \cite[Lemma 4.6 and its proof]{fibrations}, and again $g$    
 is of fiber type.
The case where 
$D$ is $(3,0)^Q$ is treated in \cite[proof of Prop.~7.1]{small}.

Therefore $g$ and $h$ are of fiber type, and
 [\emph{ibid.}, proof of Prop.~7.1] shows that $\dim Z>1$.
If $\dim Z=3$ we have $(i)$, while 
if $\dim Z=2$, then by [\emph{ibid.}, Lemma 5.14] we have
$\rho_Z=1$ and $D$ is the unique prime divisor contracted to a point by $h$. Hence Lemma \ref{lori}  gives $(ii)$.
   \end{proof} 
    \section{Fano $4$-folds with no small contractions and blow-ups of cubic $4$-folds}\label{cubic}
\noindent In this section we consider Fano $4$-folds with no small contraction, and show the following result (Th.~\ref{finalintro} from the Introduction).
\begin{thm}\label{final}
   Let $X$ be a smooth Fano $4$-fold with $\rho_X\geq 7$, not isomorphic to a product of surfaces, and without small elementary contractions.
             
   Then $\rho_X\leq 9$ and there are a cubic $4$-fold $Z\subset \pr^5$ with at most isolated ordinary double points, and $s=\rho_X-1$ distinct planes  $A_1,\dotsc,A_s\subset Z$ intersecting pairwise in a point, such that $X$ is obtained from $Z$ by blowing-up one plane $A_i$ and successively the transforms of the other ones.\footnote{The resulting $4$-fold does not depend on the order of the blow-ups,  see Lemma \ref{blowuporder}.}

      We have $\Sing(Z)\subset\cup_iA_i$;\footnote{The blow-up of the surfaces $A_i$ resolves the ordinary double points, see Rem~\ref{ODP}.} for each $i< j$ the point $A_i\cap A_j$ is in $Z_{\reg}$, and $A_i\cap A_j\cap A_k=\emptyset$ for $i<j<k$. 
    \end{thm}
    Let us note that products of del Pezzo surfaces do not have small elementary contractions. 
     Moreover the condition $\rho_X\geq 7$ is necessary in Th.~\ref{final}, as the following example shows.
     \begin{example}
Let $X$ be a Fano $4$-fold with $\rho_X=6$ and $\delta_X=3$ which is not a product of surfaces; there are 9 such families, see Ex.~\ref{es1}. Then $X$ has no small elementary contraction. Indeed there is a quasi-elementary contraction $f\colon X\to S$ where $S$ is a surface with $\rho_S=2$; the cone $\NE(f)$ is described in \cite[\S 6]{delta3}, and it does not contain small extremal rays. On the other hand, if $\NE(X)$ had a small extremal ray $R\not\subset\NE(f)$, then $f$ should be finite on $\Lo(R)$ which is a union of exceptional planes (see Lemma \ref{kawamata}), but this is impossible because $\rho_S=2$.
\end{example}
As remarked in the Introduction, we do not know whether a Fano $4$-fold as in Th.~\ref{final} exists, see \S\ref{excubics} and Question \ref{questioncubic}. See also \S\ref{excubics} for the geometry of the blow-up of a smooth cubic $4$-fold along two planes intersecting in a point; in particular we show that it is Fano.
\begin{remark}\label{conescubic}
  Let $X$ be as in Th.~\ref{final}, and consider   the blow-up $f\colon X\to Y$ of the last surface $S\subset Y$. The proof of Th.~\ref{final}  (in particular Cor.~\ref{important}) will show that $Y$ is Fano with at most isolated ordinary double points lying in $S$, and $S$ is a smooth del Pezzo surface with $\rho_S=\rho_Y=\rho_X-1\in\{6,7,8\}$. Moreover $\NE(S)\cong\NE(Y)$ under the natural map given by the inclusion $S\hookrightarrow Y$, and $\Eff(Y)\cong \NE(S)$ under the restriction. Every conic bundle $S\to\pr^1$ extends to a contraction $Y\to\pr^2$ sending $S$ to a line, and every birational map $S\to\pr^2$ extends to a blow-up map $Y\to Z'$, where $Z'$ is a cubic $4$-fold, sending $S$ to a plane.
\end{remark}  
\begin{prg}[\em Overview of the proof of Th.~\ref{final}]\label{overviewfinal}
  We show first of all that if $X\to T$ is a non-trivial contraction of fiber type, then $T\cong\pr^2$ (Lemma \ref{pasqua}). This relies on both Th.~\ref{scala}, Th.~\ref{starting}, and \cite{3folds}.

  Then we consider birational contractions of $X$.
For the case of an elementary contraction $f\colon X\to Y$,
  it is shown in \cite{32} that $f$ is  the blow-up of a smooth del Pezzo surface $S\subset Y$ with $K_S=K_{Y|S}$, and $Y$ is Fano with at most isolated ordinary double points (Prop.~\ref{cono}). 

  We first generalize this description to every birational contraction $\sigma\colon X\to Z$ (Prop.~\ref{rex}): $\sigma$ is the blow-up of smooth del Pezzo surfaces $A_1,\dotsc,A_{\rho_X-\rho_Z}\subset Z$ that intersect pairwise transversally 
in (at most) finitely many points,  and $Z$ is Fano with at most isolated ordinary double points. Moreover for every $i=1,\dotsc,\rho_X-\rho_Z$  we have
  $K_{A_i}=K_{Z|A_i}$,
  and if $E_i\subset X$ is the exceptional divisor over $A_i$,
  it follows from Th.~\ref{sharp32} that $\N(E_i,X)=\N(X)$ and hence $\N(A_i,Z)=\N(Z)$, so that $\rho_{A_i}\geq\rho_Z$. This also implies that $E_i\cap E_j\neq\emptyset$ and hence $A_i\cap A_j\neq\emptyset$ for every $i<j$. 

  Let us fix $i$ and factor $\sigma$ as $X\stackrel{\sigma_i}{\to} Y_i\stackrel{h_i}{\to}Z$, where $h_i$ blows-up all the surfaces $A_j$ for $j\neq i$, $\rho_{Y_i}=\rho_X-1$, and $\sigma_i$ blows-up the transform $S_i\subset Y_i$ of $A_i$. Then $h_{i|S_i}\colon S_i\to A_i$ blows-up all the points where $A_i$ intersects the other surfaces $A_j$, $j\neq i$. On the other hand, $S_i$ is del Pezzo, thus $\rho_{S_i}\leq 9$. This allows to bound both $\rho_{A_i}$ and the cardinality of $A_i\cap A_j$ (Rem.~\ref{remconti}), and is used repeatedly.

  Then we show that, if for some $i\in\{1,\dotsc,\rho_X-\rho_Z\}$ we have $\rho_{A_i}=\rho_Z\geq 3$, then $\NE(A_i)\cong\NE(Z)$, and one can extend contractions from $A_i$ to $Z$ (Cor.~\ref{important}).

  We use all these preliminary results to prove that 
 there exists a birational contraction $X\to Z$ with $\rho_Z=1$ and such that one of the blown-up surfaces, say $A_1\subset Z$, also has $\rho_{A_1}=1$
 (Lemma \ref{dolonne}).
 With this we can finally show Th.~\ref{final}, let us give an outline.
 
   We have $A_1\cong\pr^2$, $K_{Z|A_1}=K_{A_1}$, and we show that the restriction $\Pic(Z)\to\Pic(A_1)$ is an isomorphism. This implies that $Z$ is Fano with index $3$, namely a del Pezzo $4$-fold, and these $4$-folds are classified. Let $H\in\Pic(Z)$ be such that $-K_Z=3H$, and set $d:=H^4$; then $h^0(Z,-K_Z)=15d+10$ (see \ref{Z}).

  Next we show that there are two other blown-up surfaces, say $A_2$ and $A_3$, that intersect $A_1$  only in one point, trasversally. If $Z'\to Z$ is the blow-up of $A_2$ and $A_3$, and $A_1'\subset Z'$ is the transform of $A_1$, we have $A_1'\cong\Bl_{2\mskip1mu\pts}\pr^2$ and $\rho_{A_1'}=\rho_{Z'}=3$ (see \ref{spezia}), and hence $\NE(A_1')\cong\NE(Z')$ by the previous part of the proof; we also compute that $h^0(Z',-K_{Z'})\in\{15d-10,\dotsc,15d-6\}$. Then we extend the contraction $A_1'\to \pr^1\times\pr^1$ to an elementary birational contraction $Z'\to W$ which is again the blow-up of a surface; $W$ is a Fano $4$-fold with $\rho_W=2$ and index $2$. By studying the contractions of $W$, we show that $W$ is a double cover of $\pr^2\times\pr^2$ with branch  divisor of degree $(2,2)$ (see \ref{malpensa}). This gives $h^0(W,-K_W)=45$, and allows to compute that $h^0(Z',-K_{Z'})=36$ and that $d=3$, so that $Z$ is a cubic $4$-fold, each $A_i$ is a plane, and $A_i\cap A_j$ is just one point for every $i<j$. 
We also get $h^0(Z,-K_Z)=55$ and we can compute $h^0(X,-K_X)$ in terms of the  number of blown-up planes, that are $\rho_X-1$. Finally, using that $h^0(X,-K_X)\geq 0$, we get $\rho_X\leq 9$ (see \ref{atene}).
\end{prg}
           \begin{lemma}\label{pasqua}
             Let $X$ be a smooth Fano $4$-fold with $\rho_X\geq 7$, not isomorphic to a product of surfaces, and without small elementary contractions.
             
Let $f\colon X\dasharrow Y$ be a non-trivial rational contraction of fiber type. Then $f$ is regular and equidimensional, and $Y\cong\pr^2$.
             \end{lemma}
             \begin{proof}
           Since $X$ has no small elementary contraction, $f$ is regular.    Let $F\subset X$ be a general fiber; we have $\dim\N(F,X)=d_f\leq 4$ by Th.~\ref{scala} (see Def.-Rem.~\ref{dimfiber}). In particular $F$ cannot be a divisor because $\delta_X\leq 1$ (Th.~\ref{starting}) and $\rho_X\geq 7$, hence $\dim Y>1$.

               Suppose by contradiction that $\dim Y=3$. Then by Prop.~\ref{calanques} $f$ factors through a special rational contraction $f'\colon X\dasharrow Y'$ with $\dim Y'=3$. As above $f'$ is regular, and $\rho_{Y'}\leq 3$ by Lemma \ref{coop}. On the other hand $\rho_X-\rho_{Y'}\leq 2$ by \cite[Lemma 3.12]{3folds}, and we have a contradiction.

               We conclude that $\dim Y=2$. As above we can factor $f$ as $X\stackrel{f'}{\to} Y'\to Y$ where $f'$ is equidimensional and $\dim Y'=2$; in particular $Y'$ is a smooth rational surface (Lemma \ref{specialsurf}). If $\rho_{Y'}>1$, then there is a morphism $Y'\to\pr^1$, and the composition $X\to\pr^1$ contradicts the previous part of the proof. Therefore $\rho_{Y'}=1$ and $Y\cong Y'\cong\pr^2$.
             \end{proof}
               \begin{corollary}\label{FBI}
                 Let $X$ be a smooth Fano $4$-fold with $\rho_X\geq 7$, not isomorphic to a product of surfaces, and without small elementary contractions. Let $E\subset X$ be a fixed prime divisor. Then $E$ is of type $(3,2)$ and 
$\N(E,X)=\N(X)$.
               \end{corollary}
               \begin{proof}
 Since $X$ has no small   elementary contraction,   $E$ must be of type $(3,2)$           by Th.~-Def.\ref{fixed}$(c)$. Then $\N(E,X)=\N(X)$
               by Th.~\ref{sharp32} and Lemma \ref{pasqua}.
                 \end{proof}
           \begin{proposition}[refinement of \cite{32}, Th.~4.1]\label{cono}
             Let $X$ be a smooth Fano $4$-fold with $\rho_X\geq 7$, not isomorphic to a product of surfaces, and without small elementary contractions.
             
Let $f\colon X\to Y$ be an elementary contraction.
Then
 $f$ is of type $(3,2)$,
 $Y$ is Fano with at most isolated ordinary double points, and $S:=f(\Exc(f))\subset Y$ is a smooth del Pezzo surface with $K_S=K_{Y|S}$. 
\end{proposition}
\begin{proof}
The statement follows from \cite[Th.~4.1]{32}. Note that this reference
has the assumption $\rho_X\geq 8$, but this is used only to show that a contraction $g\colon X\to Z$ with $\rho_X-\rho_Z\leq 3$ cannot be of fiber type. In our setting this follows from Lemma \ref{pasqua}, and the same proof applies under the assumption $\rho_X\geq 7$.
Finally $Y$ is Fano by Th.-Def.~\ref{fixed}, and 
 has at most isolated ordinary double points by Lemma \ref{singtarget}.
\end{proof}
\begin{proposition}\label{rex}
  Let $X$ be a smooth Fano $4$-fold with $\rho_X\geq 7$, not isomorphic to a product of surfaces, and without small elementary contractions.
  
   Let $\sigma\colon X\to Z$ be a birational contraction; then $\NE(\sigma)=R_1+\cdots+R_s$ with $R_i$ extremal ray of type $(3,2)$ of $\NE(X)$ for every $i=1,\dotsc,s$, and $s=\rho_X-\rho_Z$. Set $E_i:=\Lo(R_i)$.
   The following holds:
   \begin{enumerate}[$(a)$]
   \item   $Z$ is Fano with at most isolated ordinary double points, and $A_i:=\sigma(E_{i})$ is a smooth del Pezzo surface with $K_{A_i}=K_{Z|A_i}$ and $\N(A_i,Z)=\N(Z)$ for every $i=1,\dotsc,s$;
\item
    for every $i\neq j$ the surfaces $A_i$ and $A_j$ intersect transversally in finitely many points where $Z$ is regular, while $A_i\cap A_j\cap A_k=\emptyset$ if $i<j<k$. Each singular point of $Z$ is contained in some  $A_i$. Moreover $\Exc(\sigma)=E_{1}\cup\cdots\cup E_{s}$ and $\sigma$ is obtained
    by blowing-up one surface $A_i$ and successively the transforms of the other ones, in any order;
  \item every elementary contraction of $Z$ is either of type $(3,2)$, or $Z\to\pr^2$.
    \end{enumerate}
\end{proposition}
\begin{proof}
  \begin{prg}\label{Si}
 By Prop.~\ref{cono} each $R_i$ is of type $(3,2)$, and
if $\sigma_i\colon X\to Y_i$ is the contraction of $R_i$, then $Y_i$ has at most isolated ordinary double points, and $S_i:=\sigma_i(E_i)$ is a smooth del Pezzo surface with $K_{S_i}=K_{Y_i|S_i}$. 

By \cite[Lemma 2.5]{32} we also have   $E_{i}\cdot R_j=0$ for every $i\neq j$. In particular $[C_{E_1}],\dotsc,[C_{E_s}]\in\N(X)$ are linearly independent, and 
   $\NE(\sigma)$ is a simplicial cone of dimension $s=\rho_X-\rho_Z$.
\end{prg}
\begin{prg}\label{NE}
   Clearly $E_i\subset\Exc(\sigma)$ for every $i=1,\dotsc,s$. Conversely if $C\subset X$ is an irreducible curve with $[C]\in\NE(\sigma)$,
we have $C\equiv\sum_{j=1}^s\lambda_j C_{E_j}$ with $\lambda_j\in\Q_{\geq 0}$.
Then for every $i=1,\dotsc,s$ we have $E_i\cdot C=-\lambda_i\leq 0$, and either $C\subset E_i$, or $\lambda_i=0$. Therefore 
either $[C]\in R_i$ for some $i\in\{1,\dotsc,s\}$, or $C\subset E_i\cap E_j$ for some $i\neq j$. In particular  $\Exc(\sigma)=E_1\cup\cdots \cup E_{s}$.

Moreover, on the open subset $X\smallsetminus\cup_{j\neq i}E_j$, $\sigma$ coincides with $\sigma_i$, therefore $\dim A_i=2$,  $A_i\smallsetminus   \cup_{j\neq i}A_j$ is smooth, and $Z$ has at most isolated ordinary double points at $A_i\smallsetminus   \cup_{j\neq i}A_j$. We also note that $\N(E_i,X)=\N(X)$ by Cor.~\ref{FBI}, hence $\N(A_i,Z)=\sigma_*(\N(E_i,X))=\N(Z)$, for every $i=1,\dotsc,s$.
\end{prg}
\begin{prg}\label{CUS}
  We show that $\sigma$ factors as the composition of $s$ elementary divisorial contractions of type $(3,2)$ with exceptional divisors $E_1,\dotsc,E_s$, in arbitrary order, and $Z$ is $\Q$-factorial. We proceed by induction on $s$, the statement being clear for $s=1$. Consider a facet of $\NE(\sigma)$, up to renumbering we can assume that it is $R_1+\cdots+R_{s-1}$, and factor $\sigma$ as:
  $$\xymatrix{X\ar@/^1pc/[rr]^{\sigma}\ar[r]_{\sigma'}&{Z'}\ar[r]_f&Z    }$$
  where $\NE(\sigma')=R_1+\cdots+R_{s-1}$. Then the statement holds for $\sigma'$ by induction, $Z'$ is $\Q$-factorial, and $f$ is an elementary birational contraction with exceptional locus $\sigma'(E_s)$, therefore $f$ is divisorial and $Z$ is $\Q$-factorial. Moreover $f$ is of type $(3,2)$ because $\dim A_s=2$. This gives the statement.
\end{prg}
\begin{prg}
 If $g\colon Z\to W$ is an elementary contraction, we can consider the composition $g\circ\sigma\colon X\to W$. If $g$ is of fiber type, then  $W\cong\pr^2$ by Lemma \ref{pasqua}. If instead $g$ is birational, then it is of type $(3,2)$ by \ref{CUS}. This gives $(c)$.
\end{prg}
\begin{prg}\label{caprera}
  Let $i,j\in\{1,\dotsc,s\}$ with $i<j$. We have $E_{i}\cap E_{j}\neq\emptyset$ by Cor.~\ref{FBI} and Rem.~\ref{cassis}. Moreover every connected component
  of  $E_{i}\cap E_{j}$ is isomorphic to $\pr^1\times\pr^1$ with the rulings being curves in $R_i$ and $R_j$  (see \cite[Lemma 4.15]{small} and \cite[Prop.~3.4]{32}). Therefore if $C\subset X$ is an irreducible curve contained in $E_i\cap E_j$, we have $[C]\in R_i+R_j$. Together with \ref{NE}, this implies that for every irreducible curve $C\subset X$ contracted by $\sigma$ we have  $[C]\in R_i+R_j$ for some $i<j$.
  
Let $\sigma_{ij}\colon X\to Y_{ij}$ be the contraction of $R_i+R_j$, and
set  $B_i:=\sigma_{ij}(E_i)$, $B_j:=\sigma_{ij}(E_j)$.
The behaviour of $\sigma_{ij}$ is analyzed in  \cite[Prop.~3.4 and its proof]{32}, where it is shown that:
\begin{enumerate}
  \item $\dim(B_i\cap B_j)=0$ and $E_{i}\cap E_{j}=\sigma_{ij}^{-1}(B_i\cap B_j)$;
\item
  if $B_i\cap B_j=\{y_1,\dotsc,y_m\}$, then  $B_i$, $B_j$, and $Y_{ij}$ are smooth at each $y_a$, and the fiber $T_a:=\sigma_{ij}^{-1}(y_a)$ is an isolated $2$-dimensional fiber of $\sigma_{ij}$ and a connected component  of $E_i\cap E_j$;
   \item  if 
we factor $\sigma_{ij}$ as
$$X\stackrel{\sigma_i}{\la}Y_i\stackrel{g}{\la} Y_{ij}$$
where $g$ is an elementary contraction of type $(3,2)$ with exceptional divisor $E_j':=\sigma_i(E_j)$, then $(E_j')_{|S_i}=C_1+\cdots+C_m$ (recall that
$S_i=\sigma_i(E_i)\subset Y_i$, see \ref{Si})
where $C_1,\dotsc,C_m$ are pairwise disjoint $(-1)$-curves in $S_i$, and fibers of $g$; moreover $C_a=\sigma_i(T_a)$ and $C_a\subset (Y_i)_{\reg}$ for every $a=1,\dotsc,m$.
    \end{enumerate}
 In particular, locally around $y_1,\dotsc,y_m$, $g$ is just the blow-up of the smooth surface $B_j$ in the smooth $4$-fold $Y_{ij}$, thus the restriction $g_{|S_i}\colon S_i\to B_i$ is the blow-up of $B_i$ at the scheme intersection $B_i\cap B_j$. Since $S_i$ is smooth (see \ref{Si}), this intersection must be reduced.
\end{prg}
\begin{prg}\label{torquay}
Let $k\in\{1,\dotsc,s\}$, $k\neq i,j$, and let $T$ be a connected component of $E_i\cap E_j$; note that $\N(T,X)=\R R_i\oplus\R R_j$ is $2$-dimensional. If $T\cap E_k\neq\emptyset$, then $T\subset E_k$, because $E_k\cdot R_i=E_k\cdot R_j=0$. Then the contraction $\sigma_k$ of $R_k$ must be finite on $T$, and $\sigma_k(T)=S_k\subset Y_k$. 
This gives
$\N(S_k,Y_k)=(\sigma_k)_*(\N(T,X))$, $\dim\N(S_k,Y_k)\leq 2$,
and $\dim\N(E_k,X)\leq 1+\dim\N(S_k,Y_k)\leq 3$
(see Lemma \ref{easy}), a contradiction. Hence $E_i\cap E_j\cap E_k=\emptyset$ and $A_i\cap A_j\cap A_k=\emptyset$.

On the open subset $X\smallsetminus\cup_{k\neq i,j}E_k$, $\sigma$ coincides with  $\sigma_{ij}$, and $A_i\cap A_j\subset\sigma(X\smallsetminus\cup_{k\neq i,j}E_k)$. We conclude from \ref{caprera}
 that every
 $A_i$ is smooth,  $A_i$ and $A_j$ intersect transversally at finitely many points where $Z$ is smooth, and $Z$ has at most isolated ordinary double points contained in some $A_i$; in particular $Z$ has locally factorial, terminal singularities (see Rem.~\ref{ODP}). We have shown $(b)$.
\end{prg}
\begin{prg}
  We have $-K_X+E_1+\cdots+E_s=\sigma^*(-K_Z)$. If $R$ is an extremal ray of $\NE(X)$ different from $R_1,\dotsc,R_s$, then $E_i\cdot R\geq 0$ for every $i=1,\dotsc,s$ by  \cite[Lemma 2.3]{32}, thus $-K_X+E_1+\cdots+E_s$ is nef and $(-K_X+E_1+\cdots+E_s)^{\perp}\cap\NE(X)=\NE(\sigma)$. This implies that $-K_Z$ is ample, namely $Z$ is Fano.
\end{prg}
\begin{prg}
 Fix $i\in\{1,\dotsc,s\}$. We have a factorization:
 \stepcounter{thm}
 \begin{equation}\label{luminy} \xymatrix{X\ar@/^1pc/[rr]^{\sigma}\ar[r]_{\sigma_i}&{Y_i}\ar[r]_{h_i}&Z;   }
   \end{equation}
 set $E'_j:=\sigma_i(E_j)\subset Y_i$ for every $j\neq i$. Then  $h_{i|S_i}\colon S_i\to A_i$ is a blow-up of smooth points, with exceptional locus $\cup_{j\neq i}E'_{j|S_i}$, which is a union of pairwise disjoint $(-1)$-curves $\Gamma_a$ (see \ref{caprera}(3)). We have $K_{S_i}=(h_{i|S_i})^*K_{A_i}+\sum_a\Gamma_a$ and $K_{Y_i}=h_i^*K_Z+\sum_{j\neq i} E_j'$; then $K_{S_i}=K_{Y_i|S_i}$ (see \ref{Si})  implies that $(h_{i|S_i})^*K_{A_i}=(h_{i|S_i})^*K_{Z|A_i}$ and finally that $K_{A_i}=K_{Z|A_i}$, so we have $(a)$.\qedhere
\end{prg}
\end{proof}
\begin{remark}\label{antsections}
 In the setting of Prop.~\ref{rex}, let $g\colon Z\to W$ be an elementary birational contraction.
 Then  $B:=g(\Exc(g))\subset W$ is a smooth del Pezzo surface and
 $$h^0(Z,-K_Z)=h^0(W,-K_W)-11+\rho_B.$$
 Indeed by Prop.~\ref{rex} both $Z$ and $W$ are Fano with locally factorial and terminal singularities, so that $h^i(Z,-K_Z)=0$ for $i>0$ and $h^0(Z,-K_Z)=\chi(Z,-K_Z)$, and similarly for $W$. Moreover, still by Prop.~\ref{rex} applied to $g\circ \sigma\colon X\to W$, we have that  $B$ is a smooth del Pezzo surface with $K_B=K_{W|B}$, hence the formula follows from Lemma \ref{blowupformula} and $\chi(B,-K_B)=h^0(B,-K_B)=11-\rho_B$.
               \end{remark}
\begin{remark}\label{remconti}
  In the setting of Prop.~\ref{rex}, set $m_{ij}:=\#(A_i\cap A_j)$ for every $i,j=1,\dotsc,\rho_X-\rho_Z$, $i\neq j$; note that $m_{ij}$ is also the number of connected components of $E_i\cap E_j$ in $X$.

Let $i\in \{1,\dotsc,\rho_X-\rho_Z\}$. The blow-up of $A_i$ at the points of  intersection with $\cup_{j\neq i}A_j$ is the del Pezzo surface $S_i\subset Y_i$ as in \eqref{luminy}, and $\rho_{S_i}\leq 9$, therefore:
  \stepcounter{thm}
  \begin{equation}\label{conti}
    9\geq\rho_{A_i}+\sum_{\substack{j=1\\j\neq i}}^{\rho_X-\rho_Z} m_{ij}\geq \rho_{A_i}+\rho_X-\rho_Z-1.
    \end{equation}
  \end{remark}
    \begin{corollary}\label{important}
    In the setting of Prop.~\ref{rex}, suppose that $\rho_Z\geq 3$ and  $\rho_{A_i}=\rho_Z$ for some $i\in\{1,\dotsc,\rho_X-\rho_Z\}$. Let $\iota\colon A_i\hookrightarrow Z$ be the inclusion.
    Then $\iota_*\colon \N(A_i)\to\N(Z)$ is an isomorphism such that $\iota_*(\NE(A_i))=\NE(Z)$. For every $(-1)$-curve $C\subset A_i$ there exists a (unique) extremal ray $R$ of $\NE(Z)$ such that    $[C]\in R$; moreover $R$ is of type $(3,2)$ and  $E_{R|A_i}=C$, where $E_R:=\Lo(R)$.
  \end{corollary}
  \begin{proof}
We have $K_{A_i}=K_{Z|A_i}$ and  $\N(A_i,Z)=\N(Z)$ by  Prop.~\ref{rex}$(a)$. In particular $\iota_*\colon \N(A_i)\to\N(Z)$ is surjective, thus it is an isomorphism because $\rho_{A_i}=\rho_Z$.
    
We show that  $\iota_*(\NE(A_i))=\NE(Z)$. The inclusion
$\subseteq$
 is clear. Conversely, let $R$ be an extremal ray of $\NE(Z)$. Then $R$ is of type $(3,2)$ by Prop.~\ref{rex}$(c)$ and because $\rho_Z\geq 3$. By the proof of Prop.~\ref{rex} (more precisely \ref{caprera} and \ref{torquay}) we have $E_R\cap A_i\neq\emptyset$ and $E_{R|A_i}=C_1+\cdots+C_m$ with $C_1,\dotsc,C_m$ pairwise disjoint $(-1)$-curves with $[C_j]\in R$, so that $R\in\iota_*(\NE(A_i))$, and we conclude that $\iota_*(\NE(A_i))=\NE(Z)$. Moreover  $-K_Z\cdot C_j=-K_{A_i}\cdot C_j=1$ for every $j$, which implies that $C_j\equiv C_1$ in $Z$. Since $\iota_*$ is injective, we get that $C_j\equiv C_1$ in $A_i$, therefore $C_j=C_1$ and $m=1$, because these are $(-1)$-curves.
  \end{proof}
  \begin{lemma}\label{pasqua2}
             Let $X$ be a smooth Fano $4$-fold with $\rho_X\geq 7$, not isomorphic to a product of surfaces, and without small elementary contractions.
            If there is a contraction
$f\colon X\to \pr^2$, then  $d_f\leq 3$.
\end{lemma}
\noindent We recall that $d_f=\dim\N(F,X)$ for a general fiber $F\subset X$ of $f$, see Def.-Rem.~\ref{dimfiber}.
             \begin{proof}
               Let us  factor $f$ as $X\stackrel{\sigma}{\to} Z \stackrel{g}{\to}\pr^2$, with $\rho_Z=2$ and $g$ elementary; let $F\subset X$ be a general fiber of $f$ and $F_Z:=\sigma(F)\subset Z$ a general fiber of $g$.
               
               Then $\sigma$ is birational by Lemma \ref{pasqua},
$g$ is elementary of fiber type, and $\dim\N(F_Z,Z)=1$. 
               We apply Prop.~\ref{rex} to $\sigma$;
consider the surfaces $A_1,\dotsc,A_{\rho_X-2}\subset Z$ blown-up by $\sigma$.
Up to renumbering
               let $A_1,\dotsc,A_r$ be those dominating $\pr^2$ under $g$, with $r\in\{0,\dotsc,\rho_X-2\}$.
               Then
               $$d_f=\dim\N(F,X)=r+\dim\N(F_Z,Z)=r+1.$$ 
               
  Suppose by contradiction that $d_f\geq 4$. Then by Th.~\ref{scala} we have $d_f=4$, $r=3$, and  $F$  is a del Pezzo surface of degree one.
               Let us further factor $f$ as
               $$X\stackrel{\sigma''}{\la}Z'\stackrel{\sigma'}{\la} Z\stackrel{g}{\la}\pr^2$$
               where $\sigma'$ is the blow-up of  $A_1,A_2,A_3\subset Z$, $\rho_{Z'}=5$,  and set $F':=\sigma''(F)\subset Z'$. Then the surfaces blown-up by $\sigma''$ do not dominate $\pr^2$, hence $F'$ is contained in the open subset where $\sigma''$ is an isomorphism, so that $F'\cong F$ and $F'$ is a del Pezzo surface of degree one. 
               By considering the anticanonical pencil in $F'$, we see that $Z'$ is covered by a proper family $V$ of irreducible curves of anticanonical degree one (recall that $Z'$ is Gorenstein Fano by Prop.~\ref{rex}).
             
           We have    $\rho_{Z'}<\rho_X$, hence $\sigma''$ is not an isomorphism;  
               let $B\subset Z'$ be a surface blown-up by $\sigma''$. By
      Prop.~\ref{rex}      we know that  $B$ is a smooth del Pezzo surface with $K_{Z'|B}=K_{B}$.  
               Let $p\in B$ be a general point. There is an irreducible curve $\Gamma$ of the family $V$ such that $p\in\Gamma$. Since $-K_{Z'}\cdot\Gamma=1$, we must have $\Gamma\subset B$  (see \cite[Prop.~2.10]{3folds}); moreover
     $-K_{B}\cdot\Gamma=-K_{Z'}\cdot\Gamma=1$.

               We conclude that $B$ is covered by curves of anticanonical degree one, hence $\rho_{B}=9$. On the other hand  \eqref{conti} applied to $\sigma''$  and $A_i=B$ gives
               $9\geq\rho_{B}+\rho_X-\rho_{Z'}-1=\rho_X+3$,  a contradiction.
             \end{proof}
\begin{lemma}\label{dolonne}
  Let $X$ be a smooth Fano $4$-fold with $\rho_X\geq 7$, not isomorphic to a product of surfaces, and without small elementary contractions.

  Then there exists a birational contraction $\sigma\colon X\to Z$ with
$\rho_Z=1$ and such that $\rho_{A_i}=1$ for some $i\in\{1,\dotsc,\rho_X-1\}$
  (notation as in Prop.~\ref{rex}).
\end{lemma}
  \begin{proof}\begin{prg}\label{prel}
  Let  $\tau\colon X\to W$ be a contraction with $\rho_W=2$; it is birational by Lemma \ref{pasqua}, and Prop.~\ref{rex} applies; in particolar $W$ is a Fano $4$-fold with at most ordinary double points as singularities, and $\tau$ blows-up $\rho_X-2\geq 5$ surfaces $B_1,\dotsc,B_{\rho_X-2}\subset W$.  For every $i=1,\dotsc,\rho_X-2$, \eqref{conti} gives $9\geq \rho_{B_i}+\rho_X-3\geq \rho_{B_i}+4$, therefore $\rho_{B_i}\leq 5$.

  Let us consider
  the two elementary contractions  $g_j\colon W\to Z_j$ of $W$, with
  $j\in\{1,2\}$, and set $f_j:=g_j\circ\tau\colon X\to Z_j$.

If $g_j$ is of fiber type, then
by Lemmas \ref{pasqua} and \ref{pasqua2} we have  $Z_j\cong\pr^2$, $f_j$ is equidimensional, and $d_{f_j}\leq 3$; moreover as in the proof of Lemma \ref{pasqua2} we see that at most two surfaces $B_i$ can dominate $\pr^2$ under $g_j$.

\medskip

If $g_j$ is birational, then it must be of type $(3,2)$ by Prop.~\ref{rex}.
Set $E_j:=\Exc(g_j)\subset W$ and $A_j:=g_j(E_j)\subset Z_j$.
Suppose that $\rho_{A_j}>1$; we show that there are at most two indices $i\in\{1,\dotsc,\rho_X-2\}$ such that  $E_{j|B_i}$ is reducible.

 We apply Rem.~\ref{remconti} to $f_j=g_j\circ\tau\colon X\to Z_j$ and the surface $A_j$; the other surfaces blown-up by $f_j$ are $g_j(B_1),\dotsc,g_j(B_{\rho_X-2})$. Set $m_{i}:=\#(g_j(B_i)\cap A_j)$, and note that this is also the number of components of $E_{j|B_i}$ in $W$.  Then \eqref{conti} gives $9\geq\rho_{A_j}+\sum_{i=1}^{\rho_X-2}m_{i}\geq 2+
 \sum_{i}m_{i}$, therefore $\sum_{i=1}^{\rho_X-2}m_{i}\leq 7$ and the sum has $\rho_X-2\geq 5$ terms, so
 we see that there are at most two indices $i\in\{1,\dotsc,\rho_X-2\}$ such that $m_{i}>1$, namely such that $E_{j|B_i}$ is reducible.
 \end{prg}
 \begin{prg}\label{fibertype}
   Suppose that both $g_1$ and $g_2$ are of fiber type.
   
Since $\rho_X-2\geq 5$, by \ref{prel} there is at least  one surface, say $B_1\subset W$, such that $g_j(B_1)\subsetneq\pr^2$ for $j=1,2$. 
    On the other hand $g_j(B_1)=f_j(E_1)$ cannot be a point because $f_j$ is equidimensional, therefore $g_j(B_1)$ is a curve in $\pr^2$. Since $B_1$ is a smooth del Pezzo surface, $g_{j|B_1}$ factors through a conic bundle $h_j\colon B_1\to\pr^1$, such that $h_1$ and $h_2$ define a finite map $h\colon B_1\to\pr^1\times\pr^1$. Since $\rho_{B_1}\leq 5$, \cite[Prop.~7.2]{druelNefEff} implies that $B_1\cong\pr^1\times\pr^1$.

 Similarly as in \ref{prel}, using \eqref{conti} and $\rho_{B_1}=2$,
    we see that there is at least one index $i\in\{2,\dotsc,\rho_X-2\}$  such that  $\#(B_1\cap B_i)=1$; up to renumbering we can assume $i=2$.
Consider the factorization of $\tau$ given by:
  $$X\stackrel{\tau''}{\la} W'\stackrel{\tau'}{\la} W$$
where 
$\tau'$ is the blow-up of $B_2$, and let $B_1'\subset W'$ be the transform of $B_1$. Then $B_1'\cong\Bl_{2\mskip1mu\pts}\pr^2$, so that $\rho_{B_1'}=\rho_{W'}=3$, and we can apply Cor.~\ref{important}.
In particular, let $C_1,C_2\subset B_1'$ be the $(-1)$-curves contracted by $B_1'\to\pr^2$; then for $i=1,2$ the class $[C_i]\in\NE(W')$ generates an extremal ray $R_i$ of type $(3,2)$, with locus $G_i\subset W'$ such that $G_{i|B_1'}=C_i$.
Moreover
 $R_1+R_2$ is a face of $\NE(W')$, because $\NE(W')\cong\NE(B_1')$. 

 Let $h\colon W'\to T$ be the contraction with $\NE(h)=R_1+R_2$, so that $\rho_T=1$.
   We have $G_1\cdot C_2=G_{1|B_1'}\cdot_{B_1'}C_2=C_1\cdot_{B_1'}C_2=0$, and similarly $G_2\cdot C_1=0$. Then it is easy to see that $\Exc(h)\subset G_1\cup G_2$, and $h$ is birational. Moreover $A_T:=h(B_1')\subset T$ is one of the surfaces blown-up by the birational contraction $h\circ\tau''\colon X\to T$, and $\rho_{A_T}=1$ by construction, so we get the statement.
\end{prg}  
\begin{prg}
  Suppose now that $g_1$ and $g_2$ are both birational; recall that $E_j=\Exc(g_j)\subset W$ and $A_j=g_j(E_j)\subset Z_j$ for $j=1,2$. We show that $\rho_{A_j}=1$ for some $j\in\{1,2\}$, which gives the statement.

By contradiction, suppose that $\rho_{A_j}>1$ for $j=1,2$.
Then, by \ref{prel}, there are at 
 most $4$ indices $i\in\{1,\dotsc,\rho_X-2\}$ such that $E_{j|B_i}$ is reducible for some $j\in\{1,2\}$, and since $\rho_X-2\geq 5$,
 there exists an index $i$, say $i=1$, such that  $E_{j|B_1}$ is irreducible
for $j=1,2$. Then
  $E_{j|B_1}=C_j$ is a $(-1)$-curve in $B_1$, with $[C_j]\in\NE(g_j)$ (see \ref{caprera}(3)).

We have $\NE(W)=\NE(g_1)+\NE(g_2)$ and $E_1\cdot \NE(g_1)<0$, hence $E_1\cdot \NE(g_2)>0$. Then $C_1\cdot_{B_1} C_2=E_{1|B_1}\cdot_{B_1} C_2=E_1\cdot C_2>0$, and since
$B_1$ is a smooth del Pezzo surface with $\rho_{B_1}\leq 5$ (see \ref{prel}), it is easy to see that it must be $C_1\cdot_{B_1} C_2=1$.
Conversely this implies that $E_1\cdot C_2=1$, and similarly 
$E_2\cdot C_1=1$. This gives $(E_1+E_2)\cdot C_1=(E_1+E_2)\cdot C_2=0$, which is impossible because $\rho_W=2$, $[C_1]$ and $[C_2]$ generate $\N(W)$, and we would get $E_1+E_2\equiv 0$.
\end{prg}
\begin{prg}
  Finally suppose that $g_1\colon W\to Z_1$ is birational and $g_2\colon W\to\pr^2$ is of fiber type.  Set $E:=\Exc(g_1)\subset W$ and $A:=g_1(E)\subset Z_1$. We have $\NE(W)=\NE(g_1)+\NE(g_2)$, and $E\cdot \NE(g_1)<0$, hence $E\cdot\NE(g_2)>0$.

  If $\rho_{A}=1$ we get the statement, thus let us assume that $\rho_A>1$. Then by \ref{prel} there at most two indices $i\in\{1,\dotsc,\rho_X-2\}$ such that $E_{|B_i}$ is reducible, and at  most two indices $i\in\{1,\dotsc,\rho_X-2\}$ such that $B_i$ dominates $\pr^2$ under $g_2$. Since $\rho_X-2\geq 5$, there is at least one index $i$, say $i=1$, such that $g_2(B_1)\subsetneq\pr^2$ and $E_{|B_1}=C$ a $(-1)$-curve in $B_1$.

  We show that $\rho_{B_1}=2$. Then the surface $g_1(B_1)\subset Z_1$ has Picard number $1$, and this gives the statement.

  Assume by contradiction that $\rho_{B_1}>2$.
As in \ref{fibertype} we see that $g_2(B_1)\subset\pr^2$ is a curve and 
$g_{2|B_1}$ factors through a conic bundle $h\colon B_1\to\pr^1$; since
$\rho_{B_1}>2$, $h$ 
 has at least one reducible fiber $\Gamma_1+\Gamma_2$, where $\Gamma_i$ are $(-1)$-curves with $\Gamma_1\cdot_{B_1}\Gamma_2=1$.
 Let $i\in\{1,2\}$. Note that in $W$ we have $[\Gamma_i]\in\NE(g_2)$ and hence $E\cdot \Gamma_i>0$, therefore  $C\cdot_{B_1}\Gamma_i=E_{|B_1}\cdot_{B_1}\Gamma_i=E\cdot\Gamma_i>0$.

  However it is impossible to have such a configuration $C,\Gamma_1,\Gamma_2$ of $(-1)$-curves in $B_1$, because $\rho_{B_1}\leq 5$ (see \ref{prel}). Indeed, consider a birational map $\ph\colon B_1\to\pr^2$ that contracts $C$ to a point $p\in\pr^2$. The $(-1)$-curves of $B_1$ are given by the exceptional curves for $\ph$, and the transforms of the lines through two blown-up points. 
  Since $C\cdot\Gamma_i>0$, $\ph(\Gamma_i)$ must be a line $\overline{pq_i}$
  where $q_i$  is another blown-up point, but then the two lines meet only at $p$ and $\Gamma_1$ and $\Gamma_2$ should be disjoint in $B_1$.
  This concludes the proof of Lemma \ref{dolonne}.\qedhere
\end{prg}
\end{proof}  
\begin{proof}[Proof of Th.~\ref{final}]
  By Lemma \ref{dolonne} there exists a birational contraction $\sigma\colon X\to Z$ with $\rho_Z=1$ and such that  $\rho_{A_i}=1$ for some $i\in\{1,\dotsc,\rho_X-1\}$ (notation as in Prop.~\ref{rex}). 
  Up to renumbering we can assume that $\rho_{A_1}=1$, so that $A_1\cong\pr^2$.
  Similarly as before (see the proof of Lemma \ref{dolonne}), using \eqref{conti} we see
that there are at least two indexes $j\in\{2,\dotsc,\rho_X-1\}$ with  $\#(A_1\cap A_j)=1$; up to renumbering we can assume that $m_{12}=m_{13}=1$.
\begin{prg}\label{spezia}
  Consider the factorization of $\sigma$ given by:
  $$X\stackrel{\sigma''}{\la} Z'\stackrel{\sigma'}{\la} Z$$
  where $\rho_{Z'}=3$ and $\sigma'$ is the blow-up of $A_2$ and $A_3$. Let $A_1'\subset Z'$ be the transform of $A_1\subset Z$. Then  $\sigma'_{|A_1'}\colon A_1'\to A_1$ blows-up two points, so that $A_1'\cong\Bl_{2\mskip1mu\pts}\pr^2$ and $\rho_{A_1'}=\rho_{Z'}=3$. For $i=2,3$ we denote by $C_i\subset A_1'$ the $(-1)$-curve over the point $A_1\cap A_i$, and $C_0\subset A_1'$ the $(-1)$-curve not contracted by $\sigma'$; recall that $\NE(A_1')=\langle [C_0], [C_2],[C_3]\rangle$.

Let  $\iota\colon A'_1\hookrightarrow Z'$ be the inclusion.
  By Cor.~\ref{important} $\iota_*\colon\N(A'_1)\to\N(Z')$ is an isomorphism that induces an isomorphism between $\NE(A_1')$ and $\NE(Z')$, hence
  $\NE(Z')$ is a simplicial $3$-dimensional cone generated by $[C_0],[C_2],[C_3]$.
  
For $i=0,2,3$ let $R_i$ be the extremal ray of $\NE(Z')$ generated by $[C_i]$ (so that $\NE(\sigma')=R_2+R_3$); by Cor.~\ref{important} each $R_i$ is of type $(3,2)$, with locus $E_i\subset Z'$ such that $E_{i|A_1'}=C_i$.

 This implies that restriction 
  $r'\colon \Pic(Z')\to\Pic(A'_1)$ is an isomorphism. Indeed $r'$ is injective because $\iota_*$ is surjective, and it is surjective 
  because $\Pic(A'_1)$ is generated by the classes of $C_0,C_2,C_3$, which are restrictions of divisors in $Z'$.
  \end{prg}
\begin{prg}\label{birat3}
  The restriction $\Pic(Z)\to\Pic(A_1)$ is an isomorphism.
\begin{proof}
 We have a commutative diagram:
  $$\xymatrix{{\Pic(Z')}\ar[r]^{r'}&{\Pic(A'_1)}\\
{\Pic(Z)}\ar[u]^{(\sigma')^*}\ar[r]^{r}&{\Pic(A_1)}\ar[u]_{(\sigma'_{|A_1'})^*}
}$$
where $r'$ is an isomorphism and $(\sigma')^*$ and $(\sigma'_{|A_1'})^*$ are injective, so that $r$ is injective too.

If $L\in\Pic(A_1)$, consider $(\sigma'_{|A_1'})^*(L)\in\Pic(A_1')$. Since $r'$ is surjective, there exists $M\in\Pic(Z')$ such that $M_{|A_1'}=(\sigma'_{|A_1'})^*(L)$. We have $M\cdot C_i=M_{|A_1'}\cdot C_i=0$ for $i=2,3$, so that $M=(\sigma')^*(M_2)$ for some $M_2\in\Pic(Z)$, and
$(\sigma'_{|A_1'})^*(M_{2|A_1})=(\sigma')^*(M_2)_{|A_1'}=M_{|A_1'}=(\sigma'_{|A_1'})^*(L)$, hence
$M_{2|A_1}=L$ and $r$ is surjective.
\end{proof}
\end{prg}
\begin{prg}\label{Z}
  By Prop.~\ref{rex} $Z$ is a Fano $4$-fold with at most locally factorial, isolated ordinary double points, $\rho_Z=1$, and $K_{A_1}=K_{Z|A_1}$. Since $A_1\cong\pr^2$, \ref{birat3} implies that
   $Z$ has index $3$, namely 
 $Z$ is a del Pezzo variety, see \cite[\S 3.2]{fanoEMS}. Let $d\in\mathbb{N}$ be its degree, \emph{i.e.} $d=H^4$ where $H$ is the ample generator of $\Pic(Z)$ and $-K_Z=3H$. We have $\chi(Z,H)=h^0(Z,H)=d+3$ [\emph{ibid.}, Rem.~3.2.2(ii)], $\chi(Z,-H)=\chi(Z,-2H)=0$, and $\chi(Z,\ol_Z)=1$, which yields $\chi(Z,tH)=\frac{1}{24}(t+1)(t+2)(dt^2+3dt+12)$, and finally $h^0(Z,-K_Z)=15 d+10$.
\end{prg}
\begin{prg}\label{Z'}
  The birational map $\sigma'\colon Z'\to Z$ factors as two blow-ups $Z'\to Z''\to Z$ with center first $A_2\subset Z$ and then the transform $A_3''\subset Z''$ of $A_3\subset Z$. By Prop.~\ref{rex}, $A_2$ is a smooth del Pezzo surface with $-K_{A_2}=-K_{Z|A_2}=3H_{|A_2}$, so that $A_2$ has again index $3$ and $A_2\cong\pr^2$.  Rem.~\ref{antsections} gives 
    $h^0(Z'',-K_{Z''})=h^0(Z,-K_Z)-10=15d$.
    
By applying \eqref{conti} to $X\to Z''$ and $A_3''\subset Z''$ we get $\rho_{A_3''}\leq 10-(\rho_X-\rho_{Z''})=12-\rho_X\leq 5$, thus 
again by Rem.~\ref{antsections} we get
\stepcounter{thm}
  \begin{equation}\label{Z'formula}
    h^0(Z',-K_{Z'})=15d-11+\rho_{A_3''}\in\{15d-10,\dotsc,15d-6\}.
    \end{equation}
\end{prg}
\begin{prg}\label{WW}
  Fix $j\in\{2,3\}$ and consider the contraction $g_j\colon Z'\to Z_j$ with $\NE(g_j)=R_0+R_j$ (see \ref{spezia}), so that $\rho_{Z_j}=1$. In $Z'$ we have $E_{j}\cdot C_0=E_{j|A_1'}\cdot_{A_1'} C_0=C_j\cdot C_0=1$, and similarly $E_{0}\cdot C_j=1$. This implies that $g_j$ must be of fiber type. Indeed if $g_j$ is birational, we can apply Prop.~\ref{rex} to $g_j\circ\sigma''\colon X\to Z_j$, and we get extremal rays $\w{R}_0$, $\w{R}_j$ in $\NE(g_j\circ\sigma'')$, with $(\sigma'')_*(\w{R}_0)=R_0$ and   $(\sigma'')_*(\w{R}_j)=R_j$, and loci $\w{E}_0=(\sigma'')^*(E_0)$, $\w{E}_j=(\sigma'')^*(E_j)$ in $X$. Then we get $\w{E}_j\cdot\w{R}_0>0$ by the projection formula, but it is shown in the proof of Prop.~\ref{rex} that $\w{E}_j\cdot\w{R}_0=0$. 

Therefore $g_j$ is of fiber type, and again
  by considering the composition $g_j\circ\sigma''\colon X\to Z_j$, we get $Z_j\cong\pr^2$ by Lemma \ref{pasqua}.

Let us also note that $(E_0+E_j)\cdot R_0=(E_0+E_j)\cdot R_j=0$, so that $\ol_{Z'}(E_0+E_j)=g_j^*\ol_{\pr^2}(a_j)$ for some $a_j\in\mathbb{N}$.

We have $(E_0+E_j)_{|A_1'}=C_0+C_j$ fiber of a conic bundle $h_j\colon A_1'\to\pr^1$, therefore
$g_j(A_1')\subset \pr^2$ is a curve, and we have a factorization
 $$\xymatrix{ {A_1'}\ar@/^1pc/[rr]^{g_{j|A_1'}}\ar[r]_{h_j}&{\pr^1}\ar[r]_<<<<{\nu_j}&{g_j(A_1')}
 }$$
 where $\nu_j$ is a finite map. Then $h_j^*\ol_{\pr^1}(1)=C_0+C_j=(E_0+E_j)_{|A_1'}=h_j^*\nu_j^*\ol_{\pr^2}(a_j)_{|g_j(A_1')}$, which gives
$\ol_{\pr^1}(1)=\nu_j^*\ol_{\pr^2}(a_j)_{|g_j(A_1')}$.
We conclude that $a_j=1$, $g_j(A_1')$ is a line in $\pr^2$,
 $\nu_j$ is an isomorphism, and $g_{j|A_1'}=h_j$.
  \end{prg}
\begin{prg}\label{malpensa}
  Let now
  $f\colon Z'\to W$ be the contraction of $R_0$; we have $\rho_W=2$, and $f$ is of type $(3,2)$ by Prop.~\ref{rex}$(c)$. Then 
Prop.~\ref{rex} applies to
  the composition
  $f\circ \sigma''\colon X\to W$; in particular $W$ is a Fano $4$-fold with at most locally factorial, isolated ordinary double points.
Set $A_W:=f(A_1')\subset W$; then $A_W$ is one of the surfaces blown-up by 
$f\circ\sigma''$, hence it is a smooth del Pezzo surface with $K_{A_W}=K_{W|A_W}$.

Moreover $g_j\colon Z'\to\pr^2$ factors through $f$, so that  
 $W$ has two elementary contractions 
 $\pi_j\colon W\to\pr^2$ for $j=1,2$, that induce a finite map $\pi\colon W\to\pr^2\times\pr^2$.

 The restriction $f_{|A_1'}$ contracts the ``middle'' $(-1)$-curve $C_0$ of $A_1'\cong\Bl_{2\mskip1mu\pts}\pr^2$, and $A_W\cong 
\pr^1\times\pr^1$. Moreover by \ref{WW} for $j=2,3$ the image $\pi_j(A_W)\subset\pr^2$ is a line, and $\pi_{j|A_W}\colon A_W\to\pr^1$ are the two projections.

As in \ref{spezia} we see that
 the restriction $\Pic(W)\to\Pic(A_W)$ is an isomorphism, and $K_{W|A_W}=K_{A_W}$, therefore $W$ has index two. Moreover  $K_{A_W}=\ol_{\pr^1\times\pr^1}(-2,-2)$ implies that $K_W=\pi_1^*\ol_{\pr^2}(-2)+
 \pi_2^*\ol_{\pr^2}(-2)$.
 
 Consider the general fiber $F_1\subset W$ of $\pi_1\colon W\to\pr^2$. It is a smooth del Pezzo surface with $K_{F_1}=K_{W|F_1}$, thus $F_1$ has index two and $F_1\cong\pr^1\times\pr^1$. Notice that $\deg\pi=\deg(\pi_{2|F_1}\colon F_1\to\pr^2)$. We have
 $$\ol_{\pr^1\times\pr^1}(-2,-2)=K_{F_1}=K_{W|F_1}=\bigl(\pi_1^*\ol_{\pr^2}(-2)+
 \pi_2^*\ol_{\pr^2}(-2)\bigr)_{|F_1}=(\pi_{2|F_1})^*\ol_{\pr^2}(-2),$$
 hence $(\pi_{2|F_1})^*\ol_{\pr^2}(1)=\ol_{\pr^1\times\pr^1}(1,1)$, and we conclude that $\deg \pi=2$. Finally $K_W=\pi^*\ol_{\pr^2\times\pr^2}(-2,-2)=\pi^*K_{\pr^2\times\pr^2}+R$ with $R\sim\pi^*\ol_{\pr^2\times\pr^2}(1,1)$, and
 $W$ is a double cover of $\pr^2\times\pr^2$ with branch locus a divisor of degree $(2,2)$.
We have  $\pi_*\ol_W=\ol_{\pr^2\times\pr^2}\oplus\ol_{\pr^2\times\pr^2}(-1,-1)$, and $$\pi_*\ol_W(-K_W)=\pi_*\ol_W\otimes\ol_{\pr^2\times\pr^2}(2,2)=\ol_{\pr^2\times\pr^2}(2,2)\oplus\ol_{\pr^2\times\pr^2}(1,1),$$
therefore $h^0(W,-K_W)=45$.
\end{prg}
\begin{prg}\label{atene}
  The blow-up map $f\colon Z'\to W$ has center a smooth del Pezzo surface of index two (by Prop.~\ref{rex} and because $W$ has index two), therefore $h^0(Z',-K_{Z'})=h^0(W,-K_W)-9=36$ by Rem.~\ref{antsections}.
Together with \eqref{Z'formula}, this implies that $d=3$ and $Z\subset \pr^5$ is a cubic $4$-fold, see \cite[Th.~3.3.1]{fanoEMS}; moreover $A_i\subset Z$ is a plane for every $i=1,\dotsc,\rho_X-1$.

The intersection $A_i\cap A_j$ is finite, therefore  $A_i$ and $A_j$ intersect only in one point, namely $m_{ij}=1$ for every $i\neq j$ (see Rem.~\ref{remconti}).
Moreover $A_i\cap A_j\cap A_k=\emptyset$ if $i<j<k$ by Prop.~\ref{rex}.

Let us write $\rho:=\rho_X$ for simplicity.
 The map $X\to Z$ factors a sequence of $\rho-1$ blow-ups:
 $$X=X_\rho\stackrel{\sigma_{\rho-1}}{\la}X_{\rho-1}{\la}\cdots\cdots
 {\la} 
 X_{2}\stackrel{\sigma_{1}}{\la} X_{1}=Z,$$
 where, for $i=1,\dotsc,\rho-1$, $\sigma_{i}\colon X_{i+1}\to X_{i}$ blows-up a smooth del Pezzo surface $S_i\subset X_i$  with $K_{X_{i|S_i}}=K_{S_i}$, given by the transform of $A_i\subset Z$. The induced map $S_i\to A_i$ blows-up $\sum_{j=1}^{i-1}m_{ij}=i-1$ points, therefore $\rho_{S_i}=i$, and
 Rem.~\ref{antsections} gives $$h^0(X_{i+1},-K_{X_{i+1}})
=h^0(X_{i},-K_{X_{i}})-11+i,$$
 and $h^0(X_{1},-K_{X_{1}})=55$.
 This gives $h^0(X,-K_X)=\frac{1}{2}(\rho^2-23\rho+132)$; on the other hand $h^0(X,-K_X)\geq 2$ (see  Th.~\ref{h0}), which implies that $\rho_X\leq 9$. This concludes the proof of Th.~\ref{final}.\qedhere
\end{prg}
\end{proof}
\section{Fano $4$-folds with a  divisor of type $(3,0)^{\sm}$}\label{sec30}
\noindent Let $X$ be a Fano $4$-fold with $\rho_X\geq 7$, that is not a product of surfaces; we consider now the case where $X$ has some small elementary contraction.  The following remark shows that such $X$ must contain a fixed prime divisor  of type $(3,0)^{\sm}$, $(3,1)^{\sm}$, or $(3,0)^Q$.
 \begin{remark}\label{united}
Let $X$ be a smooth Fano $4$-fold with $\rho_X\geq 7$, having a small elementary contraction $f\colon X\to Y$. We have $\NE(f)\not\subset\mov(X)=\Eff(X)^{\vee}$, therefore there exists a one-dimensional face $\tau$ of $\Eff(X)$ such that $\tau\cdot\NE(f)<0$. By Cor.~\ref{summary} $\tau$ is a fixed face, hence $\tau=\R_{\geq 0}[D]$ for some  fixed prime divisor $D\subset X$, and $D\cdot\NE(f)<0$. In particular $D$ contains $\Lo(f)$ which is a union of exceptional planes (Lemma \ref{kawamata}), and $D$ cannot be of type $(3,2)$ (Lemma \ref{dim32}). We conclude that $D$ must be of type $(3,0)^{\sm}$, $(3,1)^{\sm}$, or $(3,0)^Q$ (see Th.-Def.~\ref{fixed}).
          \end{remark}

In this section we consider the case where $X$ has a fixed prime divisor of type $(3,0)^{\sm}$, and show the following result (Th.~\ref{30intro} from the Introduction).
\begin{thm}\label{30}
Let $X$ be a smooth Fano $4$-fold with $\rho_X\geq 7$ having a fixed prime divisor of type $(3,0)^{\sm}$. Then one of the following holds:
  \begin{enumerate}[$(i)$]
  \item  $X$ has an elementary rational contraction onto a $3$-fold and $\rho_X\leq 9$;
  \item $\rho_X=7$,
$h^0(X,-K_X)\in\{2,\dotsc,15\}$,
and there is a SQM $X\dasharrow X'$ such that $X'=\Bl_{6\mskip1mu\pts}Y$, $Y$ a smooth Fano $4$-fold with $\rho_Y=1$ and $h^0(Y,-K_Y)=h^0(X,-K_X)+90\in\{92,\dotsc, 105\}$. Moreover every fixed prime divisor of $X$ is of type $(3,0)^{\sm}$.
    \end{enumerate}
\end{thm}
\begin{prg}[\em Overview of the proof of Th.~\ref{30}]\label{overview30}
The proof follows the same strategy as in \cite{blowup}, where the bound $\rho_X\leq 12$ is shown. Let us note that, in case $(i)$ where $X$ has an elementary rational contraction onto a $3$-fold, the bound $\rho_X\leq 9$ follows from Th.~\ref{CS}.

We treat separately the two cases where every fixed prime divisor is of type $(3,0)^{\sm}$ (Prop.~\ref{all30}), and where there are also fixed prime divisors of other types (Prop.~\ref{one30}).

In the first case, there are a smooth Fano $4$-fold $Y$ and a SQM $X\dasharrow X'$ such that $X'=\Bl_{\pts}Y$, and $Y$ does not have divisorial elementary rational contractions. We compute that $h^0(Y,-K_Y)=h^0(X,-K_X)+15(\rho_X-\rho_Y)$.  

If $Y$ has an elementary rational contraction onto a $3$-fold, then we lift this contraction to $X$, and get $(i)$.
If $\rho_Y=1$, then  either $Y\cong\pr^4$ and we get again $(i)$ (see \S\ref{fanomodel}), or we get $h^0(Y,-K_Y)\leq 105$ from \cite{blowup}, which implies
$(ii)$ (see \ref{rhoY1}).

Otherwise, we show that we can reduce to the situation where
 $\rho_Y=2$ and $Y$ has two distinct elementary rational contractions of fiber type $Y\dasharrow B_i$ where $B_i$ is $\pr^1$ or $\pr^2$, for $i=1,2$; moreover $\rho_X-\rho_Y\geq 5$ gives 
$h^0(Y,-K_Y)\geq 75$.
We show that this case does not happen, by studying explicitly the geometry of $Y$.
 This is the longest part of the proof (see \ref{dimZ2} - \ref{verylast}).

\medskip

Consider now the case where $X$ has a fixed prime divisor of type $(3,0)^{\sm}$ and also some fixed prime divisor of another type. We know that $\Eff(X)$ is generated by classes of fixed prime divisors, and that every $2$-dimensional face is fixed (Cor.~\ref{summary}). Then there must a  fixed prime divisor $D$ of type $(3,0)^{\sm}$ and a fixed prime divisor $B$, not of type $(3,0)^{\sm}$, that are adjacent. We consider the contraction $X\dasharrow Y$ of $D$, so that $Y$ is a smooth Fano $4$-fold, and the transform $B_Y\subset Y$ is still a fixed prime divisor. By analysing 
 the possibilities for $B_Y$, we show that in each case there is a fixed prime divisor $E\subset X$, of type $(3,2)$, such that $D\cap E=\emptyset$. Then we get $(i)$ by Th.~\ref{sharp32}.
\end{prg}
\begin{proposition}[refinement of \cite{blowup}, Th.~4.3]\label{all30}
  Let $X$ be a smooth Fano $4$-fold with $\rho_X\geq 7$, and suppose that every fixed prime divisor of $X$ is of type $(3,0)^{\sm}$. Then the statement of Th.~\ref{30} holds.
  \end{proposition}
  \begin{proof}
    Let us first note that, if $X$ has an elementary rational contraction onto a $3$-fold, then $\rho_X\leq 9$ by Th.~\ref{CS}, and we get $(i)$.

    We follow  \cite[proof of Lemma 4.22]{blowup}, where it is shown that there is a diagram
    $$X\stackrel{\ph}{\dasharrow} Y\stackrel{f}{\la} Z$$
    where $\ph$ is a sequence of flips and divisorial elementary  contractions, and $f$ is an elementary contraction of fiber type.

    Up to increasing the number of flips and divisorial contractions in $\ph$, we can also assume that $Y$ has no divisorial elementary  rational contractions, hence $Y$ contains no fixed prime divisors (see Rem.~\ref{fixedprime}). 

    By [\emph{ibid.}, Lemma 4.19], up to SQM we can assume that   $Y$ is a smooth Fano $4$-fold, and  there is a resolution
    $\ph'\colon X'\to Y$
    of $\ph$ such that $\ph'$ 
 is  the blow-up of $Y$ in $r:=\rho_X-\rho_Y$ distinct points. Then we have:
    \stepcounter{thm}
    \begin{equation}\label{sections}
     2\leq h^0(X,-K_X)=h^0(Y,-K_Y)-15r,
   \end{equation}
   by Th.~\ref{h0} and [\emph{ibid.}, Prop.~3.3].

   If $\dim Z=3$, then by [\emph{ibid.}, Lemma 4.21] we have $(i)$.
    \begin{prg}\label{rhoY1}
      Suppose that $\rho_Y=1$ and $Z=\{\pt\}$.
      If $Y\cong\pr^4$,
      then  $X$ is the Fano model of $\Bl_{\pts}\pr^4$ and has an elementary rational contraction onto $\Bl_{\pts}\pr^3$ (see \S\ref{fanomodel}), thus we have $(i)$.
      If instead $Y\not\cong\pr^4$, 
 as in [\emph{ibid.}, proof of Th.~3.2, p.~380] we see that $h^0(Y,-K_Y)\leq 105$, which together with \eqref{sections} implies $(ii)$.
\end{prg}
  \begin{prg}\label{dimZ2}
   Suppose that $\dim Z=2$, and note that $Z$ is smooth (Lemma \ref{specialsurf}).   
    We notice that
    $Z$ cannot have divisorial  elementary contractions, otherwise the pullback under $f$ of the exceptional divisor would be a fixed prime divisor in $Y$, against our reductions.
Therefore $Z$ is a minimal smooth rational surface  with no 
divisorial elementary contraction, and we conclude that either $Z\cong\pr^2$ and $\rho_Y=2$, or $Z\cong\pr^1\times\pr^1$ and $\rho_Y=3$. We also note that if $\dim Z=1$, then $Z\cong\pr^1$ and $\rho_Y=2$.
\end{prg}
  \begin{prg}\label{avocado}
   Suppose that $Z\cong\pr^1\times\pr^1$, and consider the blow-up
of the first point $W:=\Bl_{p_1}Y\to Y$ in $\ph'\colon X'\to Y$. The composition $W\to\pr^1\times\pr^1$ is not equidimensional, and by Prop.~\ref{calanques} it factors as
    $$\xymatrix{{W}\ar[r]\ar@{-->}[d]&Y\ar[d]\\
     S\ar[r]&{\pr^1\times\pr^1}
    }$$
    where $W\dasharrow S$ is an elementary rational contraction, $S$ is a smooth surface (see Lemma \ref{specialsurf}), and $S\to\pr^1\times\pr^1$ is the blow-up of a smooth point, thus $S\cong\Bl_{2\mskip1mu\pts}\pr^2$.

    Consider a $(-1)$-curve of $S$ contracted by $S \to \pr^2$, and its pullback $G$ in $W$. This is a fixed
prime divisor, and up to to replacing  $W$ with a SQM, we can assume that $G = \Exc(g)$ for some
divisorial elementary contraction $g \colon W \to Y'$ (see Rem.~\ref{fixedprime}). Then $g$ must be the blow-up of a smooth point (by \cite[Lemma 4.19(2)]{blowup} applied to the composition $X\dasharrow Y'$), and there is an elementary rational contraction $Y'\dasharrow \mathbb{F}_1$.
$$\xymatrix{{W}\ar[r]^g\ar@{-->}[d]&{Y'}\ar@{-->}[d]\\
S\ar[r]&{\mathbb{F}_1}
}$$
Now we consider the blow-up $\mathbb{F}_1\to\pr^2$ and the composition $Y'\dasharrow \pr^2$; it factors again as an elementary divisorial rational contraction $Y'\dasharrow Y''$ and an elementary contraction $Y''\to\pr^2$. Finally, by replacing $Y$ with $Y''$,
 and we can assume that $\rho_Y=2$ and $Z\cong\pr^2$.
\end{prg}
  \begin{prg}
    We are left with the possibility that $\rho_Y=2$ and $Y$ has two distinct elementary rational contractions of fiber type $Y\dasharrow B_i$  where $B_i$ is $\pr^1$ or $\pr^2$, for $i=1,2$, so that $\Mov(Y)=\Eff(Y)$ and the two one-dimensional faces of this cone are the pullbacks of $\Nef(B_i)$, $i=1,2$.
    We are going to show that
   this contradicts $\rho_X\geq 7$.
   \end{prg}
 \begin{prg}\label{flixbus}
Suppose that $Y$ has index $>1$. Then $Y$ cannot have small elementary contractions (see Lemma \ref{kawamata}), therefore 
both maps $Y\to B_i$ must be regular. Since each map must be finite on the fibers of the other one, such fibers are at most $2$-dimensional, and we conclude that $Y$ has two elementary contractions onto $\pr^2$.

Fano $4$-folds with index $>1$ and Picard number $>1$ are classified, see \cite[Cor.~3.1.15, Th.~3.3.1, Th.~7.2.15]{fanoEMS}. An inspection of the classification shows that the only cases where $\rho_Y=2$ and $Y$ has  two elementary contractions onto $\pr^2$ are $\pr^2\times\pr^2$ and 
a double cover of $\pr^2\times\pr^2$ ramified over a divisor of degree $(2,2)$.

On the other hand $Y$ cannot be covered by a family of curves of anticanonical degree $3$ by \cite[Lemma 3.12]{blowup}, and this excludes $\pr^2\times\pr^2$. In the other case
we have $h^0(Y,-K_Y)=45$ (see \ref{malpensa}), and by \eqref{sections} this implies that $r\leq 2$ and $\rho_X\leq 4$, a contradiction.
\end{prg}
  \begin{prg}\label{duo}
We suppose from now on that $Y$ has index one.

For $i=1,2$ let us consider a $K$-negative resolution 
 $f_i\colon Y_i\to B_i$ of $Y\dasharrow B_i$;
we have a diagram:
$$\xymatrix{{Y_1}\ar[d]_{f_1}&Y\ar@{-->}[l]_{\psi_1}\ar@{-->}[r]^{\psi_2}&{Y_2}\ar[d]^{f_2}\\
{B_1}&&{B_2}
}$$
where $\psi_i$ is a SQM.
Fix $i\in\{1,2\}$ and let $D_i\subset Y_i$ be a general fiber of $f_i$ if $B_i=\pr^1$, the pullback of a general line if $B_i=\pr^2$.
Let $\w{D}_i\subset Y$ be the transform of $D_i$.
\end{prg}
  \begin{prg}
    We notice that
    the general fiber $F_i$ of $f_i$ cannot have a covering family of rational curves of anticanonical degree $3$.
    Otherwise, $Y_i$ would have a covering family $V$ of rational curves, of anticanonical degree $3$, contracted to points by $f_i$. However this is excluded as in \cite[4.24]{blowup}.

    This implies that $D_i$ contains an irreducible curve $\Gamma_i$, contracted by $f_i$,  with $c_i:=-K_{D_i}\cdot \Gamma_i\in\{1,2,4\}$ \cite[Cor.~IV.1.15]{kollar}, and $c_i\in\{1,2\}$ if $B_i=\pr^2$.  Notice that $D_i\cdot\Gamma_i=0$ and $-K_{Y_i}\cdot\Gamma_i=c_i$.

    Since $F_i$ is Fano, it is rationally connected.
Consider $f_i\colon Y_i\to B_i$ if $B_i=\pr^1$, or
$f_{i|D_i}\colon D_i\to\pr^1$ 
 if $B_i=\pr^2$. These are families of rationally connected varieties over a curve, and 
by \cite{GraberHarrisStarr} they have a section, namely
there exists $C_i\subset Y_i$ such that  $D_i\cdot C_i=1$.
    \end{prg}
  \begin{prg}\label{colorado}
We have $\Eff(Y)=\langle[\w{D}_1],[\w{D}_2]\rangle$, so there exist positive integers $n,a_1,a_2$ such that
$$n(-K_Y)=a_1\w{D}_1+a_2\w{D}_2,$$
and we can assume that $\gcd(n,a_1,a_2)=1$. 
Now let $\widehat{D}_2\subset Y_1$ be the transform of $D_2$. In $Y_1$ we have 
\stepcounter{thm}
\begin{equation}\label{Y1}
n(-K_{Y_1})=a_1D_1+a_2\widehat{D}_2,
\end{equation}
 and intersecting with $C_1$ we get 
$$n(-K_{Y_1}\cdot C_1)=a_1+a_2\widehat{D}_2\cdot C_1.$$
This implies that $\gcd(a_2,n)$ divides $a_1$, and hence that $\gcd(a_2,n)=1$. Similarly we get that $\gcd(a_1,n)=1$.

Set $d:=\gcd(a_1,a_2)$ and  $a_i:=da_i'$ for $i=1,2$. 
We have
$$n(-K_Y)=d(a'_1\w{D}_1+a'_2\w{D}_2)$$
and $\gcd(n,d)=1$. This implies that $d=1$, because $Y$ has index $1$. 
Now intersecting with $\Gamma_1$ in $Y_1$ we get from \eqref{Y1}:
$$nc_1=a_2\widehat{D}_2\cdot\Gamma_1,$$
which implies that $a_2$ divides $c_1$ and hence that $a_2\in\{1,2,4\}$, and $a_2\in\{1,2\}$ if $B_1=\pr^2$.
Similarly we deduce that $a_1\in\{1,2,4\}$.

Since $\gcd(a_1,a_2)=1$, 
up to exchanging $Y_1$ and $Y_2$ we can assume that $a_1=1$, so that \eqref{Y1} becomes:
\stepcounter{thm}
\begin{equation}\label{Y2}
  n(-K_{Y_1})=D_1+a_2\widehat{D}_2.
\end{equation}  
\end{prg}
  \begin{prg}
Suppose that $B_1\cong\pr^1$. Then $D_1$ is a smooth Fano $3$-fold and $h^0(D_1,-K_{Y_1|D_1})=h^0(D_1,-K_{D_1})=\frac{1}{2}(-K_{D_1})^3+3\leq 35$ \cite[Cor.~7.1.2]{fanoEMS}. Moreover by \eqref{Y2}:
$$-K_{Y_1}-D_1=a_2\widehat{D}_2-(n-1)(-K_{Y_1}),$$
and we know that $h^0(Y_1,-K_{Y_1})=h^0(Y,-K_Y)>0$ by Th.~\ref{h0};
thus we get
\begin{gather*}
h^0(Y_1,-K_{Y_1}-D_1)=h^0\bigl(Y_1,a_2\widehat{D}_2-(n-1)(-K_{Y_1})\bigr)\leq 
h^0(Y_1,a_2\widehat{D}_2)
=h^0(Y_2,a_2D_2)\\=h^0\bigl(Y_2,f_2^*\ma{O}_{B_2}(a_2)\bigr)=h^0\bigl(B_2,\ma{O}_{B_2}(a_2)\bigr)\leq
h^0\bigl(\pr^2,\ma{O}_{\pr^2}(4)\bigr)=15.\end{gather*}

Now the exact sequence
$$0\la\ma{O}_{Y_1}(-K_{Y_1}-D_1)\la\ma{O}_{Y_1}(-K_{Y_1})\la\ma{O}_{D_1}(-K_{Y_1|D_1})\la 0$$
yields 
$$h^0(Y,-K_Y)=h^0(Y_1,-K_{Y_1})\leq h^0(Y_1,-K_{Y_1}-D_1)+h^0(D_1,-K_{Y_1|D_1})\leq 15+35=50,$$
that together with \eqref{sections} implies that $r\leq 3$ and $\rho_X\leq 5$, a contradiction.
\end{prg}
  \begin{prg}\label{verylast}
    Suppose now that $B_1\cong\pr^2$, and
consider the blow-up
of the first point
$\Bl_{p_1}Y\to Y$ in $\ph'\colon X'\to Y$. By \cite[Lemma 4.19(1)]{blowup}, applied to the map $X\dasharrow \Bl_{p_1}Y$, we see that there is a smooth Fano $4$-fold $W$, with $\rho_{W}=3$, and a SQM $\Bl_{p_1}Y\dasharrow W$. Moreover there is a SQM of $X$ which is obtained by blowing-up $W$ at $\rho_X-3$ distinct points, so that using again [\emph{ibid.}, Prop.~3.3]
we have:
 \stepcounter{thm}
 \begin{equation}\label{sections2}h^0(W,-K_{W})=h^0(X,-K_X)+15(\rho_X-3).\end{equation}
 
 Since $p_1$ cannot belong to an exceptional plane [\emph{ibid.}, Lemma 4.4(3)], we see that $p_1\in\dom(\psi_i)$ for $i=1,2$ (see \ref{duo}); we still denote by $p_1$ its image in $Y_i$. Let
 $$\sigma_i\colon W_i\la Y_i$$ be the blow-up of $p_1$, and $E_i\subset W_i$ the exceptional divisor.
Then $W_i$ is another SQM of $\Bl_{p_1}Y$ and $W$.

Consider the composition $f_1\circ\sigma_1\colon W_1\to\pr^2$. As in \ref{avocado}, we see that this map must factor through an elementary rational contraction $W_1\dasharrow \mathbb{F}_1$, where $h\colon \mathbb{F}_1\to \pr^2$ is the blow-up of the point $f_1(p_1)$. 
$$\xymatrix{ {W_1}\ar@{-->}[d]\ar[r]^{\sigma_1}&{Y_1}\ar[d]^{f_1}\\
 {\mathbb{F}_1}\ar[r]^h&{\pr^2}
  }$$
  Recall that $D_1=f_1^*(l)$ where  $l\subset\pr^2$ is a general line (see \ref{duo}).  Using \eqref{Y2} we get:
  \stepcounter{thm}
  \begin{equation}\label{summer}
  \begin{split}
  -nK_{W_1} = & \ \sigma_1^*(-nK_{Y_1})-3nE_1=\sigma_1^*\bigl(D_1+a_2\wi{D}_2\bigr)-3nE_1\\
 = &\ \sigma_1^*\bigl(f_1^*(l)\bigr)+a_2\sigma_1^*(\wi{D}_2)-3nE_1.
  \end{split}
 \end{equation}

  Let $\alpha\colon W_1'\to\mathbb{F}_1$ be a $K$-negative resolution of $W_1\dasharrow\mathbb{F}_1$, $\wi{D}_2'\subset W_1'$ the transform of $\sigma_1^*(\wi{D}_2)$, and
$E_1'\subset W_1'$ the transform of $E_1=\Exc(\sigma_1)$. 
We have
  $E_1'=\alpha^*(e)$ where  $e\subset \mathbb{F}_1$
  is the $(-1)$-curve. Let also
  $\pi\colon\mathbb{F}_1\to\pr^1$ be the $\pr^1$-bundle; we have $h^*(l)=f+e$ where $f\subset \mathbb{F}_1$ is a general fiber of $\pi$. Finally let us consider $F=\alpha^*(f)$  a general fiber of $\pi\circ\alpha\colon W_1'\to\pr^1$.
$$\xymatrix{{W_1'}\ar@{-->}[r]\ar[dr]^{\alpha}& {W_1}\ar@{-->}[d]\ar[r]^{\sigma_1}&{Y_1}\ar[d]^{f_1}\\
{\pr^1}& {\mathbb{F}_1}\ar[l]_{\pi}\ar[r]^h&{\pr^2}
}$$

We note that $F$ is a smooth Fano $3$-fold. Indeed there is a $K$-negative resolution $W_1''\to\pr^1$ of $\pi\circ\alpha$, and the indeterminacy locus of the SQM $W_1'\dasharrow W_1''$ is contracted to points by $\alpha$, therefore it does not meet $F$. Hence $F$ is isomorphic to a general fiber of  $W_1''\to\pr^1$, which is Fano. We also note that $\alpha(F)=f$, so that $\rho_F\geq 2$.
  
In $W_1'$ we get, from \eqref{summer}: 
 \begin{align*}  -nK_{W_1'}=& \ \alpha^*\bigl(h^*(l)\bigr)+a_2\wi{D}_2'-3nE_1'=
                              \alpha^*(f+e)+a_2\wi{D}_2'-3nE_1'\\
   = & \  F+a_2\wi{D}_2'-(3n-1)E_1'.
 \end{align*}
 Recall that $\ol_{Y_2}(D_2)=f_2^*\ol_{B_2}(1)$ (see \ref{duo}), and
 that $a_2\in\{1,2\}$ because $B_1\cong\pr^2$ (see \ref{colorado}). 
  Therefore we have:
  \begin{gather*}
    h^0(W_1',-K_{W_1'}-F)= h^0\bigl(W_1',a_2\wi{D}_2'-(3n-1)E_1'-(n-1)(-K_{W_1'})\bigr)\\ \leq  h^0(W_1',a_2\wi{D}_2')
    =h^0\bigl(W_2,\sigma_2^*(a_2D_2)\bigr)
  =h^0\bigl(W_2,\sigma_2^*\bigl(f_2^*\ol_{B_2}(a_2)\bigr)\bigr)\\ =h^0(B_2,\ol_{B_2}(a_2))\leq  h^0(\pr^2,\ol(2))=6.
\end{gather*}
Moreover $-K_{W_1'|F}=-K_F$ and $h^0(F,-K_F)=\frac{1}{2}(-K_F)^3+3\leq 34$ by classification \cite[Ch.~12]{fanoEMS}, and finally we get $h^0(W,-K_W)=h^0(W_1',-K_{W_1'})\leq  h^0(W_1',-K_{W_1'}-F)+h^0(F,-K_F)\leq 40$. Together with \eqref{sections2}
this implies that $\rho_X\leq 5$, again a contradiction. This concludes the proof of Prop.~\ref{all30}.\qedhere
\end{prg}
\end{proof}
\begin{proposition}[refinement of \cite{blowup}, Th.~5.40]\label{one30}
Let $X$ be a smooth Fano $4$-fold with $\rho_X\geq 7$ having both a fixed prime divisor
of type $(3,0)^{\sm}$ and a fixed prime divisor not of type $(3,0)^{\sm}$.

Then $X$ contains a fixed prime divisor $D$
of type $(3,0)^{\sm}$ and
a fixed prime divisor $E$ of type $(3,2)$ such that $D\cap E=\emptyset$; in particular $\N(E,X)\subsetneq\N(X)$. 
\end{proposition}
\begin{proof}
  We proceed as in \cite[proof of Th.~5.40]{blowup}.
Since $X$ has a fixed prime divisor of type $(3,0)^{\sm}$, $X$ is not isomorphic to a product of surfaces, hence $\delta_X\leq 1$ (Th.~\ref{starting}) and every face of $\Eff(X)$ of dimension $\leq 2$ is fixed (Cor.~\ref{summary}).

Since every one-dimensional face of $\Eff(X)$ is generated by a fixed prime divisor, there must a $2$-dimensional face $\langle [D], [B]\rangle$ such that 
 $D$ is
of type $(3,0)^{\sm}$ and
$B$ is not. This face is fixed, hence
$D$ and $B$ are adjacent.

Let $X\dasharrow \w{X}\stackrel{\sigma}{\to} Y$ be the contraction of $D$ as in Th.-Def.~\ref{fixed}$(a)$, so that $X\dasharrow \w{X}$ is a SQM, $Y$ is a smooth Fano $4$-fold with $\rho_Y\geq 6$, and $\sigma$ is the blow-up of a point $p\in Y$ with exceptional divisor the transform of $D$; we also have $\delta_Y\leq 2$ by Lemma \ref{deltaY}, so that Th.-Def.~\ref{fixed} applies to $Y$ too.

  We note that since every $2$-dimensional face of $\Eff(X)$ is fixed, the same holds for every one-dimensional face of $\Eff(Y)$, and the transform $B_Y\subset Y$ of $B$ is a fixed prime divisor of $Y$ (see \cite[Lemma 2.21]{blowup}). Moreover $p\not\in B_Y$ by [\emph{ibid.}, Lemma 5.41], and $B_Y$ is not of type $(3,0)^{\sm}$ by [\emph{ibid.}, Lemma 4.14].
  
  Let $C_1,\dotsc,C_s\subset Y$ be the images of the exceptional lines of $\w{X}$;  then $C_i$ cannot meet any irreducible curve of anticanonical degree one in $Y$, for every $i=1,\dotsc,s$ [\emph{ibid.}, Lemma 3.11(1)]. 

  If $B_Y$ is of type $(3,2)$, then $B_Y$ is covered by irreducible curves of anticanonical degree one, therefore $B_Y\cap C_i=\emptyset$ for   every $i=1,\dotsc,s$. Then
$B_Y$
is  contained in the domain of the birational map $Y\dasharrow X$, therefore $B\cong B_Y$ and $B$ is of type $(3,2)$ too. Moreover  $B\cap D=\emptyset$, and we have the statement.

  \medskip

Suppose that $B_Y$ is of type $(3,0)^Q$. Note that $Y$ cannot have a covering family of curves of anticanonical degree $3$ (see [\emph{ibid.}, Lemma 3.12]), hence by [\emph{ibid.}, Prop.~5.18] $Y$ must contain a prime divisor $E$ covered by curves of anticanonical degree one. As before, we have $C_i\cap E=\emptyset$ for every $i=1,\dotsc,s$. Moreover, as in [\emph{ibid.}, 5.44] we see that $[C_i]\in\mov(Y)$ for every $i=1,\dotsc,s$. Thus $\tau:=E^{\perp}\cap\mov(Y)$ is a face of $\mov(Y)$ containing $[C_1],\dotsc,[C_s]$.

By [\emph{ibid.}, Lemma 5.7] we have
$$\dim\tau\geq\dim\R([C_1],\dotsc,[C_s])=\dim\N(D,X)-1\geq\rho_X-2=\rho_Y-1.$$
Therefore $\tau$ is a facet of $\mov(Y)$, and dually $\R_{\geq 0}[E]$ is a one-dimensional face of $\Eff(Y)=\mov(Y)^{\vee}$. Hence $E$ is a fixed divisor, it must be of type $(3,2)$ [\emph{ibid.}, Lemma 2.18], and as in the previous part of the proof we see that the transform of $E$ in $X$ is a fixed prime divisor of type $(3,2)$ disjoint from $D$.

\medskip

Finally, if $B_Y$ is of type $(3,1)^{\sm}$, there is a SQM $\ph\colon Y\dasharrow \w{Y}$ and a smooth Fano $4$-fold $Z$ such that $\w{Y}$ is the blow-up of $Z$ along a smooth irreducible curve $C\subset Z$, with exceptional divisor the transform $B_{\w{Y}}\subset\w{Y}$ of $B_Y$. Let $\alpha$ be a one-dimensional face of $\Eff(Z)$ with $\alpha\cdot C>0$. Then $\alpha\not\subset\Mov(Z)$ by [\emph{ibid.}, Lemma 5.17], so that $\alpha=\R_{\geq 0}[G]$ where $G\subset Z$ is a fixed prime divisor; let  $G_{\w{Y}}\subset\w{Y}$, $G_Y\subset Y$, and $G_X\subset X$ be its transforms. 
By \cite[Lemmas 4.2 and 4.4]{small}, $\langle [D],[B],[G_X]\rangle$ is a fixed face of $\Eff(X)$, 
such that $G_X$ is adjacent to both $D$ and $B$ .
Since $G\cdot C>0$, we have $G_{\w{Y}}\cap B_{\w{Y}}\neq\emptyset$, thus $G_Y\cap B_Y$ is non-empty and is not a union of exceptional planes (recall that $\dim (\w{Y}\smallsetminus \dom\ph)=1$ by Lemma \ref{SQMFano}). This implies that also 
$G_X\cap B$   is non-empty and is not a union of exceptional planes. Then  $G_X$ must be of type $(3,2)$ by \cite[Lemma 4.13(b)]{small}, hence $G_X\cap D=\emptyset$ by Lemma \ref{30sm}, and we get the statement.
\end{proof} 
\begin{proof}[Proof of Th.~\ref{30}]
  If every fixed prime divisor of $X$ is of type $(3,0)^{\sm}$, we apply Prop.~\ref{all30}. Otherwise  by  Prop.~\ref{one30} there are a fixed prime divisor $E$ of type $(3,2)$, and another fixed prime divisor $D$, such that
$D\cap E=\emptyset$. Then we have $\N(E,X)\subsetneq\N(X)$ (see Rem.~\ref{cassis}), and we get $(i)$
by Th.~\ref{sharp32}.
\end{proof}  
\section{Fano $4$-folds with a divisor of type $(3,1)^{\sm}$ or $(3,0)^Q$}\label{secsmall}
\noindent In this section we conclude our study of Fano $4$-folds with $\rho\geq 7$ having a small elementary contraction, by considering the case where $X$ has a fixed prime divisor of type $(3,1)^{\sm}$ or $(3,0)^Q$ (Th.~\ref{31}). Then we prove Cor.~\ref{explicitfinal} characterizing Fano $4$-folds with $\rho\geq 9$ that are not products of surfaces; this result implies
Theorems \ref{sharpbound} and \ref{explicit} from the Introduction. Finally we also give an application to Fano $4$-folds with $\rho =8$ (Cor.~\ref{last}).
  \begin{thm}\label{31}
  Let $X$ be a smooth Fano $4$-fold  with $\rho_X\geq 9$ having a fixed prime divisor of type $(3,1)^{\sm}$ or $(3,0)^Q$. Then 
  $X$ has a rational contraction onto a $3$-fold and $\rho_X=9$. 
\end{thm}
Note that in the setting of the theorem, it is enough to show that $X$ has a rational contraction onto a $3$-fold, because then $\rho_X=9$ follows from Th.~\ref{CS}.

To prove Th.~\ref{31}, we follow \cite[proofs of Th.~8.1 and Th.~9.1]{small}, where in the same setting the bound $\rho_X\leq 12$ is shown. Thanks to the results in the previous sections, in particular Cor.~\ref{summary}, Th.~\ref{zucca}, and Th.~\ref{sharp32}, and also to Th.~\ref{starting}, we can improve the bound to $9$  and simplify the proofs.
We first treat the case where $X$ has a fixed prime divisor of type $(3,1)^{\sm}$, and then the case where $X$ has a fixed prime divisor of type $(3,0)^{Q}$ (and no fixed prime divisors of type $(3,1)^{\sm}$).
Since the proofs from \cite{small} are quite technical, instead of reporting them in full we will 
point out  which modifications are needed to get our  statement; for the case of a divisor of type $(3,1)^{\sm}$ we will also give a highlight of the main points of the proof.
\begin{proof}[Proof of Th.~\ref{31}, case $(3,1)^{\sm}$]
  We follow \cite[proof of Th.~8.1]{small}. 

  Let us first assume that: \emph{$X$ is a smooth Fano $4$-fold with $\rho_X\geq 8$, having a fixed prime divisor $D$ of type $(3,1)^{\sm}$, and
  with no rational contraction onto a $3$-fold.}
  
Since $X$ has a fixed prime divisor of type $(3,1)^{\sm}$, $X$ is not isomorphic to a product of surfaces.
  Then we have the following properties:
  \begin{enumerate}[$\bullet$]
    \item $\delta_X\leq 1$ (Th.~\ref{starting});
    \item 
  $\N(E,X)=\N(X)$ for every fixed prime divisor $E\subset X$ of type $(3,2)$
  (Th.~\ref{sharp32});
\item every face of $\Eff(X)$ of dimension $\leq 3$ is fixed (Cor.~\ref{summary});
  \item $X$ has no fixed prime divisor of type $(3,0)^{\sm}$  (Prop.~\ref{one30}).
\end{enumerate}

We consider the contraction $X\dasharrow \w{X}\stackrel{\sigma}{\to} Y$ of $D$ as in Th.-Def.~\ref{fixed}, so that $Y$ is a smooth Fano $4$-fold with $\rho_Y\geq 7$ and $\sigma$ is the blow-up of a smooth irreducible curve $C\subset Y$. Then we have:
\begin{enumerate}[$(a)$]
  \item $Y$ has no rational contraction onto a $3$-fold;
    \item 
  $\N(E,Y)=\N(Y)$ for every fixed prime divisor $E\subset Y$ of type $(3,2)$;
\item $Y$ has no fixed prime divisor of type $(3,0)^{\sm}$;
  \item $Y$ is not isomorphic to a product of surfaces;
  \item $\delta_Y\leq 1$.
  \end{enumerate}
Here $(a)$, $(b)$, and $(c)$ follow from the same properties of $X$ as shown in 
 \cite[8.2]{small}, and $(d)$ and $(e)$ are proved in [\emph{ibid.}, 8.3].
We also have:
 \begin{enumerate}[$(a)$]\setcounter{enumi}{5}
    \item  every face of $\Eff(Y)$ of dimension $\leq 2$ is fixed (Cor.~\ref{summary}).
    \end{enumerate}
    
As in [\emph{ibid.}, 8.5], we introduce the following notation: if $C$ belongs to a family of rational curves of anticanonical degree one with locus a divisor, we denote this divisor by $E_0$. Then we have:
  \begin{enumerate}[$(a)$]\setcounter{enumi}{6}
  \item
    if $Y$ contains a nef prime divisor $H$ covered by a family $V$ of rational curves of anticanonical degree one, then $H=E_0$ and $[C]\equiv[V]$.
     \end{enumerate}
     Indeed we must have $H\cap C\neq\emptyset$ by Prop.~\ref{sabri}, $(a)$, and $(b)$. Then $C$ is a member of the family $V$ by [\emph{ibid.}, Lemma 4.21$(a)$], so that $H=E_0$.


    As in [\emph{ibid.}, 8.6 - 8.7], we denote by $E_1,\dotsc,E_r$ the set of fixed prime divisors of $Y$ of type $(3,2)$ and different from $E_0$.
It is shown in [\emph{ibid.}, 8.12 - 8.16] that $r\geq \rho_X-5$, and that  for every $i=1,\dotsc,r$ the transform $\w{E}_i\subset X$ of $E_i$ is a fixed prime divisor of type $(3,2)$ having intersection zero with $C_D$.
 Note that  [\emph{ibid.}, 8.9 and 8.10] are not needed because of $(f)$, and [\emph{ibid.}, 8.11] is replaced by $(g)$. Moreover [\emph{ibid.}, 8.18] shows that every fixed prime divisor of $Y$ is of type $(3,1)^{\sm}$ or $(3,2)$ (and both types do occur), hence by  [\emph{ibid.}, Lemma 4.22] every fixed prime divisor of $X$, adjacent to $D$, is still  of type $(3,1)^{\sm}$ or $(3,2)$. We get the following intermediate result.
 \begin{proposition}[refinement of \cite{small}, Prop.~8.17]\label{vivina}
Let $X$ be a smooth Fano $4$-fold with $\rho_X\geq 8$ and $D\subset X$ a fixed prime divisor of type $(3,1)^{\sm}$. Then  one of the following holds:
  \begin{enumerate}[$(i)$]
  \item $X$ has a rational contraction onto a $3$-fold;
    \item there are at least $\rho_X-5$ fixed prime divisors of type $(3,2)$ adjacent to $D$ and having intersection zero with $C_D$; moreover every fixed prime divisor adjacent to $D$ is of type $(3,1)^{\sm}$ or $(3,2)$.
    \end{enumerate}
\end{proposition}

We assume from now on that $\rho_X\geq 9$, namely that \emph{$X$ is a smooth Fano $4$-fold with $\rho_X\geq 9$, having a fixed prime divisor $D$ of type $(3,1)^{\sm}$, and
  with no rational contraction onto a $3$-fold.} We show that this gives a contradiction.

Notice that $\rho_Y\geq 8$, in particular
every face of $\Eff(Y)$ of dimension $\leq 3$ is fixed, by
Cor.~\ref{summary} and $(d)$; this replaces [\emph{ibid.}, 8.24].

Let $B\subset Y$ be a 
  fixed prime divisor
 of type $(3,1)^{\sm}$; by $(a)$  and Prop.~\ref{vivina} applied to $Y$ and $B$,  we see that 
 $B$ is adjacent to some divisor among $E_1,\dotsc,E_r$.

 Using this,  [\emph{ibid.}, 8.26 - 8.27]
 shows that there exist a fixed prime divisor $B_1\subset Y$ of type $(3,1)^{\sm}$, and $i\in\{1,\dotsc,r\}$, say $i=1$, such that $B_1+E_1$ is movable and not big, so that the linear system $m|B_1+E_1|$ for $m\gg 0$ defines a rational contraction of fiber type $f\colon Y\dasharrow Z$. By [\emph{ibid.}, Lemma 5.11], $(a)$, and $(b)$, we have $Z\cong\pr^2$ and $f$ is special. 

 We note that $C$ cannot be contracted to a point by $f$, otherwise  the composition $X\dasharrow \pr^2$ has a unique prime divisor (precisely $D$) contracted to a point, and this contradicts Lemma \ref{lori}. Thus $C$ is not contracted to a point,
   and as in [\emph{ibid.}, proof of 8.9], this implies that the general fiber of $f$ is $\pr^2$, hence the minimal face $\tau_f$ of $\Eff(Y)$ containing $f^*\Nef(\pr^2)$ is a facet (see Def.-Rem.~\ref{dimfiber}).

   Then [\emph{ibid.}, 8.29] shows that there are  irreducible curves $\Gamma_1,\dotsc,\Gamma_{\rho_Y-2}\subset\pr^2$, and fixed prime divisors $B_2,\dotsc,B_{\rho_Y-2}\subset Y$ of type $(3,1)^{\sm}$, such that (up to renumbering) $f^*(\Gamma_i)=B_i+E_i$ for $i=1,\dotsc,\rho_Y-2$, and
   $$\tau_f=\langle[B_1],\dotsc,[B_{\rho_Y-2}], [E_1],\dotsc,[E_{\rho_Y-2}]\rangle.$$
   Moreover $\langle [B_1],\dotsc,[B_{\rho_Y-4}], [E_{\rho_Y-3}],[E_{\rho_Y-2}]\rangle$ is a face of $\tau_f$ of dimension $\rho_Y-2$, hence there exists a  facet  $\eta$ of $\Eff(Y)$ such that $$\langle [B_1],\dotsc,[B_{\rho_Y-4}], [E_{\rho_Y-3}],[E_{\rho_Y-2}]\rangle=\tau_f\cap\eta$$
   (see [\emph{ibid.}, 8.32]).

   If $P\subset Y$ is a fixed prime divisor such that $[P]\in\eta\smallsetminus\tau_f$, [\emph{ibid.}, 8.33 - 8.35] shows that, up to renumbering, $P+B_i$ is movable for $i=1,\dotsc,\rho_Y-6$. Let $\eta_1$ be the minimal face of $\eta$ containing $[(\rho_Y-6)P+B_1+\cdots+B_{\rho_Y-6}]$, so that $\eta_1$ contains $[P],[B_1],\dotsc,[B_{\rho_Y-6}]$, and $[P+B_i]\in\eta_1\cap\Mov(Y)$ for every $i=1,\dotsc,\rho_Y-6$; in particular $\eta_1\cap\Mov(Y)$ is a face of $\Mov(Y)$ of dimension $\geq \rho_Y-6$. As in Rem.~\ref{zumba}, we take a cone $\alpha\in\MCD(Y)$ such that $\alpha\subset\eta_1\cap\Mov(Y)$ and $\dim\alpha=\dim(\eta_1\cap\Mov(Y))$, and consider the rational contraction  $h\colon Y\dasharrow S$  such that $\alpha=h^*(\Nef(S))$ (see [\emph{ibid.}, 8.36]).

 Then $h$ is of fiber type and special, by Lemma \ref{eubea}; moreover $\rho_S=\dim(\eta_1\cap\Mov(Y))\geq \rho_Y-6\geq 2$, hence $S\not\cong\pr^1$, and by $(a)$ we have $\dim S=2$.
 By   [\emph{ibid.}, Lemma 5.12] and $(b)$, $h$ must be quasi-elementary. 

 Finally Th.~\ref{zucca}, $(a)$, and $(d)$ imply that $h$ is elementary, therefore
$\rho_S=\rho_Y-1$, $\eta_1=\eta$, and $h^*(\Nu(S))=\R\eta$. On the other hand $[E_{\rho_Y-2}]\in\eta$, and by  [\emph{ibid.}, Lemma 5.11] and $(b)$ this gives a contradiction.
This concludes the first part of the proof of Th.~\ref{31}.
\qedhere
    \end{proof}
    \begin{lemma}[refinement of \cite{small}, Lemma 9.2]\label{9.2}
      Let $X$ be a smooth Fano $4$-fold with $\rho_X\geq 7$ and let $D,E$ be two adjacent fixed prime divisors of type $(3,0)^Q$ in $X$. Let $L\subset E$ be an exceptional plane such that $L\cap D=\emptyset$. 
Then one of the following holds:
  \begin{enumerate}[$(i)$]
  \item $X$ has a rational contraction onto a $3$-fold;
    \item there exists an exceptional plane $M\subset D$ such that $C_M+C_E\equiv C_D+C_L$ and $D\cdot C_M=-1$.
    \end{enumerate}
       \end{lemma} 
       \begin{proof} Note that $X$ is not a product of surfaces, because it contains fixed prime divisors of type $(3,0)^Q$, hence $\delta_X\leq 1$ by Th.~\ref{starting}.
The same proof as the one of \cite[Lemma 9.2]{small} applies, just using Th.~\ref{sharp32} instead of [\emph{ibid.}, Th.~4.8], and Prop.~\ref{sabri} and Th.~\ref{sharp32} instead of [\emph{ibid.}, Prop.~7.1].
\end{proof}
 \begin{lemma}[refinement of \cite{small}, Lemma 9.3]\label{9.3}
   Let $X$ be a smooth Fano $4$-fold with $\rho_X\geq 7$ and let $D,E$ be two adjacent fixed prime divisors of type $(3,0)^Q$ in $X$. Suppose that one of the following holds:
   \begin{enumerate}[$(a)$]
   \item every fixed prime divisor is of type $(3,0)^Q$;
   \item there is a fixed prime divisor $F$, of type $(3,2)$, such that $\langle [D],[E],[F]\rangle$ is a fixed face of $\Eff(X)$.
       \end{enumerate}
Then one of the following holds:
  \begin{enumerate}[$(i)$]
  \item $X$ has a rational contraction onto a $3$-fold;
  \item there exist
$L_1,\dotsc,L_{\rho_X-2}\subset E$ exceptional planes, disjoint from $D$, such that $[C_E],[C_{L_1}],\dotsc,[C_{L_{\rho_X-2}}]$ is a basis of $D^{\perp}$. In case $(b)$, we can moreover assume that $C_E\equiv C_F+C_{L_1}$ and $[C_E],[C_F],[C_{L_2}],\dotsc,[C_{L_{\rho_X-2}}]$ is a basis of $D^{\perp}$.
    \end{enumerate}
       \end{lemma} 
       \begin{proof}
         Note that $X$ is not a product of surfaces, because it contains fixed prime divisors of type $(3,0)^Q$, hence $\delta_X\leq 1$ by Th.~\ref{starting}.
The same proof as the one of \cite[Lemma 9.3]{small} applies, just using Lemma \ref{torino} instead of [\emph{ibid.}, Lemma 7.2], and Th.~\ref{sharp32} instead of [\emph{ibid.}, Th.~4.8].
         \end{proof}
          \begin{proof}[Proof of Th.~\ref{31}, case $(3,0)^Q$]
Let $X$ be a smooth Fano $4$-fold with $\rho_X\geq 9$ having a fixed prime divisor of type $(3,0)^Q$.
            Note that $X$ is not a product of surfaces,  therefore $\delta_X\leq 1$ by Th.~\ref{starting}.

         If  $X$ has a rational contraction onto a $3$-fold,  then $\rho_X=9$ by Th.~\ref{CS}.
         We assume that  $X$ has no rational contraction onto a $3$-fold, and reach a contradiction.
         By Th.~\ref{30} and by the first part of the proof, every fixed prime divisor of $X$ is of type $(3,0)^Q$ or $(3,2)$.
         
         The same proof as \cite[proof of Th.~9.1]{small} works, with the following modifications.
            \begin{enumerate}[$\bullet$]
            \item By Cor.~\ref{summary}, every face of $\Eff(X)$ of dimension $\leq 4$ is fixed; this simplifies [\emph{ibid.}, 9.4 and 9.6]. 
            \item In [\emph{ibid.}, 9.4 and 9.5] we apply Lemma \ref{torino} instead of  [\emph{ibid.}, Lemma 7.2].
     \item In [\emph{ibid.}, 9.7] we apply Th.~\ref{sharp32} instead of  [\emph{ibid.}, Th.~4.8].           
\item
            In [\emph{ibid.}, 9.12 and 9.13] we apply Lemma \ref{9.3} instead of [\emph{ibid.}, Lemma 9.3].
\item
  In [\emph{ibid.}, 9.14 and 9.24] we apply Lemma \ref{9.2} instead of [\emph{ibid.}, Lemma 9.2].
\item
  In [\emph{ibid.}, 9.20] we get a contradiction, because we are assuming $\rho_X\geq 9$.
  \qedhere
          \end{enumerate}
        \end{proof}
        \begin{corollary}\label{explicitfinal}
  Let $X$ be a smooth Fano $4$-fold with $\rho_X\geq 9$, not isomorphic to a product of surfaces.
Then $\rho_X=9$ and one of the following holds:
\begin{enumerate}[$(i)$]
\item $X$ has an elementary rational contraction onto a $3$-fold;
\item  $X$ is the blow-up of $W$ along a normal surface $S$,
  where $W$ is the Fano model of $\Bl_{7\mskip1mu\pts}\pr^4$, and $S\subset W$ is the transform of a cubic scroll or
  a cone over a twisted cubic in
$\pr^4$, containing the blown-up points;
\item  $X$ is a blow-up of a cubic $4$-fold as in Th.~\ref{finalintro}.
     \end{enumerate}
   \end{corollary}
  Note that  this result implies
Theorems \ref{sharpbound} and \ref{explicit}.
\begin{proof}  
  If $X$ has no small elementary contractions, then by Th.~\ref{finalintro} we have $\rho_X=9$ and $(iii)$.
  
  If instead $X$ has a small elementary contraction, then 
 by Rem.~\ref{united} $X$ has a fixed prime divisor of type $(3,0)^{\sm}$, 
$(3,1)^{\sm}$, or $(3,0)^Q$.
If $X$ has a fixed prime divisor of type $(3,0)^{\sm}$, then we apply Th.~\ref{30} and get $\rho_X= 9$ and $(i)$.

 Finally, if $X$ has a fixed prime divisor of type $(3,1)^{\sm}$ or $(3,0)^Q$, by Th.~\ref{31} we get that $\rho_X=9$ and $X$ has a rational contraction onto a $3$-fold. Then we apply Th.~\ref{CS}, and get $(i)$ or $(ii)$.
\end{proof}
 \begin{corollary}\label{last}
          Let $X$ be a smooth Fano $4$-fold with $\rho_X=8$. Then one of the following holds:
          \begin{enumerate}[$(i)$]
          \item $X$ has a rational contraction onto a $3$-fold;
          \item every fixed prime divisor of $X$ is of type $(3,1)^{\sm}$ or $(3,2)$;
            \item every fixed prime divisor of $X$ is of type $(3,0)^Q$ or $(3,2)$, and two divisors of type $(3,2)$ cannot be adjacent.
\end{enumerate}
\end{corollary}
\begin{proof}
  Let us assume that $(i)$ does not hold, namely $X$ has no rational contraction onto a $3$-fold; in particular $X$ is not isomorphic to a product of surfaces.
By Th.~\ref{30}, 
$X$ has no fixed prime divisors of type $(3,0)^{\sm}$, hence
every fixed prime divisor of $X$ is of type $(3,2)$, $(3,1)^{\sm}$, or $(3,0)^Q$
(see Th.-Def.~\ref{fixed}).  By Prop.~\ref{vivina}, in $X$ a divisor of type $(3,1)^{\sm}$ and one of type $(3,0)^Q$ cannot be adjacent. Moreover by Lemma \ref{torino} a divisor of type $(3,0)^Q$ can be adjacent to at most one divisor of type $(3,2)$. Let us also note that every face of $\Eff(X)$ of dimension $\leq 3$ is fixed, by Cor.~\ref{summary}, and hence simplicial (see \cite[Lemma 4.2]{small}).

Let $E$ be a fixed prime divisor of type $(3,2)$, and suppose that $E$ is adjacent to a fixed prime divisor $D$ of type $(3,0)^Q$. Then
every $3$-dimensional face of $\Eff(X)$ containing $\langle [E],[D]\rangle$ must be $\langle  [E],[D],[D']\rangle$ with $D'$ of type $(3,0)^Q$. This implies that every fixed prime divisor adjacent to $E$ must be of type $(3,0)^Q$.

Consider now a sequence of fixed prime divisors $D_1,\dotsc,D_r$ such that $D_i$ is adjacent to $D_{i+1}$ for every $i=1,\dotsc,r-1$,  with $D_1$ of type $(3,0)^{Q}$. Then $D_2$ can be either $(3,0)^Q$ or $(3,2)$, and in this last case, $D_3$ must be $(3,0)^Q$. Proceeding in this way, we see that no $D_i$ can be of type $(3,1)^{sm}$. Since $\Eff(X)$ is connected, we must have $(ii)$ or $(iii)$.
  \end{proof}       
  \section{Examples}\label{secexamples}
  \subsection{Fano models of the blow-ups of $\pr^4$ in points}\label{fanomodel}
\noindent Let $X'$ be the blow-up of $\pr^4$ at $r$ general points, with $r\leq 8$. Then there is a SQM $X'\dasharrow X$ such that $X$ is smooth and Fano, see \cite[Ex.~7.2 and references therein]{3folds}; we have
 $\rho_X=1+r\leq 9$.

 If $r\geq 1$, there is an elementary rational contraction $f\colon X\dasharrow Y:=\Bl_{r-1\mskip2mu\pts}\pr^3$.

 Moreover (for $r\geq 2$) there is also 
an
elementary rational contraction $\ph\colon Y\dasharrow S:=\Bl_{r-2\mskip2mu\pts}\pr^2$, and the composition $g:=\ph\circ f\colon X\dasharrow S$ is a quasi-elementary rational contraction with $d_g=\rho_X-\rho_S=2$.

Finally, by composing $g$ with a conic bundle $S\to\pr^1$ (for $r\geq 3$), we get a rational contraction $h\colon X\dasharrow\pr^1$ with $d_h=3$.

 We also note that every fixed prime divisor of $X$ is of type $(3,0)^{\sm}$.
\subsection{Two examples with $\rho=7$ from \cite{3folds}}\label{newex}
\noindent Let $W$ be the Fano model of $\Bl_{q_1,\dotsc,q_5}\pr^4$ (see \S\ref{fanomodel}), and for $i=1,2$ let
$S_i\subset W$ be the transform of the following surface $A_i\subset\pr^4$:
\begin{enumerate}[$\bullet$]
\item $A_1$ is a general cubic rational normal scroll containing $q_1,\dotsc,q_5$;
\item $A_2$ is a sextic K3 surface with $\Sing(A_2)=\{q_1,\dotsc,q_5\}$ and having ordinary double points at each $q_j$; more precisely
$A_2 =Q\cap M$ where $Q$ is a general quadric hypersurface containing $q_1,\dotsc,q_5$, and $M$ is a general cubic hypersurface with ordinary double points at $q_1,\dotsc,q_5$. 
\end{enumerate}  
 Let $\sigma_i\colon X_i\to W$ be the blow-up along $S_i$;
then $X_i$ is a smooth Fano $4$-fold with $\rho_{X_i}=7$ 
(see \cite[\S7.2 and \S7.3]{3folds}).

There is an elementary rational contraction $\psi\colon W\dasharrow Y:=
\Bl_{4\mskip1mu\pts}\pr^3$, and the composition $f_i:=\psi\circ \sigma_i\colon X_i\dasharrow Y$
is a
special rational contraction with $\rho_{X_i}-\rho_{Y}=2$.

Moreover there is an elementary rational contraction $\ph\colon Y\dasharrow S:=\Bl_{3\mskip1mu\pts}\pr^2$, and
the composition $g_i:=\ph\circ f_i\colon X_i\dasharrow S$ is a  quasi-elementary rational contraction with  $d_{g_i}=\rho_{X_i}-\rho_S=3$.

Finally, by composing $g_i$ with a conic bundle $S\to\pr^1$, we get a rational contraction $h_i\colon X_i\dasharrow\pr^1$ with $d_{h_i}=4$.

We also note that $X_i$ contains (at least) six fixed prime divisors of type $(3,2)$: one is $\Exc(\sigma_i)$, and the others are the transforms in $X_i$ of the cones in $\pr^4$ over $A_i$ with vertex $q_j$, for $j=1,\dotsc,5$.

Moreover, for $j=1,\dotsc,5$, let $D_j\subset X_i$ be the transform of the exceptional divisor over $q_j$ in $\Bl_{q_1,\dotsc,q_5}\pr^4\to\pr^4$. Then $D_j$ is a fixed prime divisor 
 of type $(3,1)^{\sm}$ in $X_1$, of type $(3,0)^Q$ in $X_2$  (see [\emph{ibid.}, Section 5, in particular Lemmas 5.39, 5.46, 5.68, 5.74]).
\subsection{A family with an elementary rational contraction onto a surface}
\label{Y}
 \noindent Let $p_1,\dotsc,p_m\in\pr^2$ be general points, $C_1',\dotsc,C_m'\subset\pr^2$ general lines,  $C_i:=\{p_i\}\times C_i'\subset\pr^2\times\pr^2$, and $h\colon \pr^2\times\pr^2\to\pr^2$ the first projection.
  
  Let $\beta\colon\wi{X}\to\pr^2\times\pr^2$ be the blow-up of $C_1,\dotsc,C_m$,   $\wi{D}_i\subset\wi{X}$ the exceptional divisor over $C_i$, and $L_i\subset \wi{X}$ the transform of the fiber $h^{-1}(p_i)$, containing $C_i$.

 Since $\ma{N}_{C_i/\pr^2\times\pr^2}\cong\ol_{\pr^1}(1)\oplus\ol_{\pr^1}^{\oplus 2}$, we have $\wi{D}_i\cong\pr_{\pr^1}(\ol\oplus\ol(1)^{\oplus 2})$, and $\wi{D}_i\cap L_i$ is the flopping curve in $\wi{D}_i$.
 Moreover  $L_i\cong\pr^2$ and  $\ma{N}_{L_i/\wi{X}}\cong\ol_{\pr^2}(-1)^{\oplus 2}$, namely $L_i$ is an exceptional plane in $\wi{X}$.

 Let us consider the composition $h\circ\beta\colon\wi{X}\to\pr^2$.
 The class of every curve in $\wi{D}_i\cup L_i$ is in $\langle [C_{L_i}]\rangle+\NE(\beta)$, so that
 $$\NE(h\circ\beta)=\NE(\beta)+\langle [C_{L_1}],\dotsc,[C_{L_m}]\rangle,$$
and $h\circ\beta$ is $K$-negative.
One can also check that $-K_{\wi{X}}-\sum_{i=1}^m\wi{D}_i$ has positive intersection with every non-zero class in $\NE(\beta)$, while it has zero intersection with every $C_{L_i}$, so that $\langle [C_{L_1}],\dotsc,[C_{L_m}]\rangle$ is an $m$-dimensional face of $\NE(h\circ\beta)$, whose contraction is small.
There is a SQM $\xi\colon\wi{X}\dasharrow\w{X}$, relative to $\beta\circ h$:
$$\xymatrix{
  {\wi{X}}\ar[d]_{\beta}\ar@{-->}[r]^{\xi}& {\w{X}}\ar[d]_{\tilde h}\ar[dr]^g&\\
  {\pr^2\times\pr^2}\ar[r]^<<<<h&{\pr^2}&S\ar[l]_{\alpha}
  }$$
  where $\w{X}$ is smooth, the indeterminacy locus of $\xi$ is $L_1\cup\cdots\cup L_m$, and that of $\xi^{-1}$ is the disjoint union of $m$ exceptional lines $\ell_1,\dotsc,\ell_m$.

  Let $D_i\subset\w{X}$ be the transform of $\wi{D}_i$; we have
$\ell_i\subset D_i$ and $D_i=(\tilde{h})^{-1}(p_i)$.
   The restriction $(\xi)_{|\wi{D}_i}\colon   \wi{D}_i\dasharrow D_i$ is a flop, we still have $D_i\cong \pr_{\pr^1}(\ol\oplus\ol(1)^{\oplus 2})$, but now the flopping curve in $D_i$ is the exceptional line $\ell_i$ (see Fig.~\ref{figura_Y}).
   It is not difficult to check that the transform $\Gamma\subset\w{X}$ of a line in a general fiber of $h\colon\pr^2\times\pr^2\to\pr^2$ is numerically equivalent to a line in a fiber of the $\pr^2$-bundle $D_i\to\pr^1$, and that
   $$\NE(\tilde{h})=\langle[\Gamma],[\ell_1],\dotsc,[\ell_m]\rangle.$$
Let $\alpha\colon S\to\pr^2$ be the blow-up of $p_1,\dotsc,p_m$. Then $\tilde h$ factors as $\alpha\circ g$ where $g\colon \w{X}\to S$ is a $\pr^2$-bundle, and $g(D_i)\subset S$ is the exceptional curve over $p_i$.

In conclusion, the composition $g\circ\xi\colon \wi{X}\dasharrow S$ is an elementary rational contraction onto a surface. We have $\rho_{\wi{X}}=m+2$.

\begin{figure}\caption{The flip over the point $p_i\in\pr^2$, in \S\ref{Y}.}\label{figura_Y}

\bigskip
  
  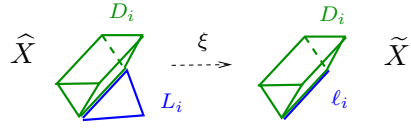
\end{figure}
 
\bigskip

\noindent {\bf The case $m\leq 3$.\ }
If $m\leq 3$, then all these varieties are toric, and they can be described explicitly using Batyrev's
language of primitive collections and primitive relations \cite{bat91}, and Sato's description of how primitive relations change after smooth toric blow-ups and blow-downs \cite[\S 4]{sato}.
Let us consider in detail each case.

For $m=1$, $X:=\wi{X}$ is Fano with $\rho_X=3$; it is $G_5$ in Batyrev's classification of toric Fano $4$-folds \cite{bat2}.

For $m=2$, $\wi{X}$ is not Fano: it contains one exceptional line
$\ell$   whose image in $\pr^2$ is the line $\overline{p_1p_2}$. More precisely,  $\ell$ is the transform of the unique curve $\overline{p_1p_2}\times\{q\}\subset\pr^2\times\pr^2$ that intersects both $C_1$ and $C_2$ (namely $q=C_1'\cap C_2'$).
The class $[\ell]$ generates a small extremal ray of $\NE(\wi{X})$;  the associated flip $\wi{X}\dasharrow X$ gives a smooth toric Fano $4$-fold $X$ with $\rho_X=4$, which is $Z_2$ in \cite{bat2}.
$$\xymatrix{
 X& {\wi{X}}\ar[d]_{\beta}\ar@{-->}[l]\ar@{-->}[r]^{\xi}& {\w{X}}\ar[d]_{\tilde h}\ar[dr]^g&\\
&  {\pr^2\times\pr^2}\ar[r]^h&{\pr^2}&S\ar[l]_{\alpha}
  }$$

Finally let us consider the case $m=3$. Now
$\wi{X}$ contains three exceptional lines $\ell_{ij}$, whose images in $\pr^2$ are the lines $\overline{p_ip_j}$, for $1\leq i<j\leq 3$. As before,  $\ell_{ij}$ is the transform of the curve $\overline{p_ip_j}\times\{q_{ij}\}\subset\pr^2\times\pr^2$ with $q_{ij}=C_i'\cap C_j'$.
Again one can consider the flips of these three exceptional lines $\wi{X}\dasharrow X$, and get a smooth toric Fano $4$-fold $X$ with $\rho_X=5$; this is Sato's example \cite[Ex.~4.7]{sato}.  

Therefore for $m\in\{1,2,3\}$ this construction gives a Fano $4$-fold $X$ with $\rho_X=m+2\in\{3,4,5\}$ and an elementary rational contraction $X\dasharrow S$ onto a surface.

\bigskip

\noindent{\bf The case $m\geq 4$.\ } 
\begin{lemma}
If there exists a SQM $\wi{X}\dasharrow X$ such that $X$ is a smooth Fano $4$-fold, then $m\leq 6$.
\end{lemma}
\begin{proof}
Assume that $X$ exists. We have
 $h^0(X,-K_X)=\chi(X,-K_X)$ because $X$ is Fano, and $\chi(X,-K_X)=\chi(\wi{X},-K_{\wi{X}})$ by \cite[Cor.~3.10]{blowup}. Since $C_i\cong\pr^1$ and $-K_{\pr^2\times\pr^2}\cdot C_i=3$ for every $i=1,\dotsc,m$, by [\emph{ibid.}, Lemma 3.8] we get $$0\leq h^0(X,-K_X)=\chi(\wi{X},-K_{\wi{X}})=\chi(\pr^2\times\pr^2,\ol(3,3))-15m=100-15m$$
 which implies that $m\leq 6$.
\end{proof}
  \begin{question}\label{questionY}
For $m\in\{4,5,6\}$, does there exist a SQM $\wi{X}\dasharrow X$ such that $X$ is a smooth Fano $4$-fold?
  \end{question} 
  \noindent This would give a Fano $4$-fold $X$ with $\rho_X=m+2\in\{6,7,8\}$ and an elementary rational contraction $X\dasharrow S$ onto a surface. This construction shows up in the proof of Th.~\ref{zucca}, see \ref{points}.
 \subsection{A family with $\rho=3$ and a quasi-elementary  contraction onto $\pr^2$}\label{exX}
\noindent Let $A\subset\pr^2\times\pr^2$ be a surface obtained as a complete intersection of two general divisors, $B$ of degree $(2,1)$ and $D$ of degree $(0,2)$; then $A$ is a del Pezzo surface with $K_A^2=2$.
Let $\sigma\colon X\to\pr^2\times\pr^2$ be the blow-up along $A$; we show that $X$ is Fano. This example is inspired by the proof of Th.~\ref{zucca}, see \ref{tauZ}.

 We have
  $$K_X^4=(K_{\pr^2\times\pr^2})^4-3(K_{{\pr^2\times\pr^2}|A})^2-2K_{A}\cdot K_{{\pr^2\times\pr^2}|A}+c_2(\ma{N}_{A/\pr^2\times\pr^2})-K_A^2$$
(see \cite[Lemma 2.4]{3folds}). One can compute that $(K_{\pr^2\times\pr^2})^4=486$, $(K_{{\pr^2\times\pr^2}|A})^2=90$,  $K_{A}\cdot K_{{\pr^2\times\pr^2}|A}=18$,  and $c_2(\ma{N}_{A/\pr^2\times\pr^2})=16$,
 thus $K_X^4=194$.

Let $E\subset X$ be the exceptional divisor, and $\w{B}\subset X$ the transform of $B$; note that $\sigma_{|\w{B}}\colon\w{B}\to B$ is an isomorphism. By adjunction we have
\begin{gather*}
-K_{X|\w{B}}=-K_{\w{B}}+\w{B}_{|\w{B}}=-K_{\w{B}}+(\sigma^*B-E)_{|\w{B}}
=\sigma_{|\w{B}}^*(-K_B+B_{|B}-A)\\
=\sigma_{|\w{B}}^*\bigl((-K_{\pr^2\times\pr^2}-D)_{|B}\bigr)
\end{gather*}
and $-K_{\pr^2\times\pr^2}-D$ is ample, thus $-K_{X|\w{B}}$ is ample.

Let $C\subset X$ be an irreducible curve; we show that $-K_X\cdot C>0$. This is clear if $C\subset\w{B}$ or if $\sigma(C)=\{\pt\}$. Otherwise, let
$H\in |\ol_{\pr^2\times\pr^2}(1,2)|$, so that $\w{B}+\sigma^*H\in|-K_X|$. Then
$\w{B}\cdot C\geq 0$ and $\sigma^*(H)\cdot C>0$, hence $-K_X\cdot C>0$.

We conclude that $-K_X$ is strictly nef and big, hence ample by the base-point-free theorem. Here are the invariants of $X$:
$$K_X^4=194,\ \rho_X=3,\ b_4(X)=h^{2,2}(X)=11,\ b_3(X)=0,\ h^0(X,-K_X)=45.$$

\begin{figure}[b]\caption{A section of the cone of effective divisors of $X$ in \S\ref{exX}.}\label{figuraexX}

\bigskip
  
  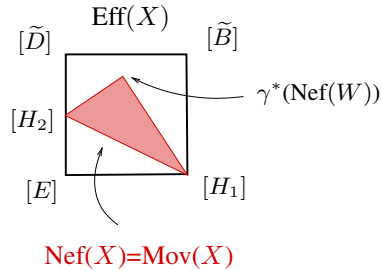
\end{figure}

Let $\pi_i\colon \pr^2\times\pr^2\to\pr^2$ be the two projections, and $f_i:=\pi_i\circ\sigma\colon X\to\pr^2$, $i=1,2$.

Since $\pi_1(A)=\pr^2$ and $(\pi_1)_{|A}$ has degree $2$, $f_1$ is a quasi-elementary contraction with general fiber $F\cong\Bl_{2\mskip1mu\pts}\pr^2$, and $d_{f_1}=\dim\N(F,X)=2$.

We also have $\w{B}\cong B$ and $f_{1|\w{B}}$ is a $\pr^1$-bundle. If $C_{\w{B}}\subset\w{B}$ is a fiber of such $\pr^1$-bundle, we have $-K_X\cdot C_{\w{B}}=1$ and $\w{B}\cdot  C_{\w{B}}=-1$, and $\w{B}$ is the exceptional divisor of a divisorial elementary contraction $\alpha\colon X\to Y$ of type $(3,2)^{\sm}$, such that $\NE(f_1)=\NE(\sigma)+\NE(\alpha)$.
$$\xymatrix{X\ar[r]^<<<<<{\sigma}\ar[d]_{\alpha}\ar[dr]_{f_1}&{\pr^2\times\pr^2}\ar[d]^{\pi_1}\\
Y\ar[r]&{\pr^2}
  }$$
  Here $Y$ is a smooth Fano $4$-fold with $\rho_Y=2$, and the map $Y\to\pr^2$ is a quadric bundle.
  
  On the other hand, since $\pi_2(A)\subset\pr^2$ is a conic, $f_2\colon X\to\pr^2$ is equidimensional but not quasi-elementary; the general fiber is $\pr^2$. Moreover $f_2^*(\pi_2(A))=\w{D}+E$ where $\w{D}\subset X$ is the transform of $D$, and $\w{D}\cong\pr^2\times\pr^1$ with normal bundle $\ma{N}_{\w{D}/X}\cong\ol_{\pr^2\times\pr^1}(-2,2)$. In fact $\w{D}$ is the exceptional divisor of a divisorial elementary contraction $\beta\colon X\to Z$ of type $(3,1)$, where $Z$ is singular along the curve $\beta(\w{D})$.

Finally we note that since $\w{B}\cap\w{D}=\emptyset$, one can check that $\NE(\alpha)+\NE(\beta)$ is a face of $\NE(X)$, with a birational contraction $\gamma\colon X\to W$ where $\rho_W=1$.
 If $H_i=f_i^*\ol_{\pr^2}(1)$, we have $\Nef(X)=\Mov(X)=\langle H_1, H_2, \gamma^*\Nef W\rangle$ and $\Eff(X)=\langle H_1, \w{B}, \w{D}, E\rangle$, see Fig.~\ref{figuraexX}.
\subsection{Blow-ups of $\pr^4$ along planes}
\noindent Let $\sigma\colon X_s\to\pr^4$ be the blow-up of $s$ general planes $A_1,\dotsc,A_s\subset\pr^4$, in the following sense: we blow-up one plane $A_i$ and successively the transforms of the other ones.
 Note that each pair of planes intersect pairwise in a point, and $\sigma$ does not depend on the order of the blow-ups,  see Lemma \ref{blowuporder}. We will refer to such a map simply as the blow-up of $A_1,\dotsc,A_s$.

We have $\rho_{X_s}=1+s$, and for $s=1,2$ the $4$-fold $X_s$ is toric and Fano:  $X_1\cong\pr_{\pr^1}(\ol^{\oplus 3}\oplus\ol(1))$, and $X_2$  is $D_{17}$ in \cite{bat2}.

Let $\pi\colon X_1\to\pr^1$ be the  $\pr^3$-bundle, and $\sigma'\colon X_s\to X_1$ the remaining blow-ups. Let also $\w{A}_i\subset X_1$ be the transform of $A_i$ for $i=2,\dotsc,s$. Then $\w{A}_i\cong\mathbb{F}_1$, $\pi_{|\w{A}_i}$ is the $\pr^1$-bundle, and $\w{A}_i$ intersects the general fiber of $\pi$ in a line. The composition $\pi\circ\sigma'\colon X_s\to\pr^1$ is a quasi-elementary contraction with general fiber the blow-up $F_s$ of $\pr^3$ along $s-1$ general lines.
If $X_s$ is Fano, then $F_s$ must be Fano too;
by classification $F_s$ is Fano if and only if $s\leq 4$ (see \cite[Ch.~12]{fanoEMS}), hence we get the following.
\begin{lemma}
If $X_s$ is Fano, then $s\leq 4$ and $\rho_{X_s}\leq 5$.
\end{lemma}  
{\small
\begin{table}[h]
 $\begin{array}{||c|c|c|c|c|c|c||}
   \hline\hline
   
s   & \rho_{X_s} & K_{X_s}^4 & b_4(X_s)=h^{2,2}(X_s) & b_3(X_s)& \chi(X_s,-K_{X_s}) & \text{Fano}\\
      \hline\hline

         0 & 1 & 625 & 1 & 0 & 126 & \pr^4 \\

        \hline

       1 & 2 & 512 & 2  & 0 & 105 & \text{yes, toric}\\

       \hline

     2 & 3 & 405 & 4 & 0 & 85  & \text{yes, toric}\\
       
  \hline

      3 & 4 & 304 & 7  & 0 & 66 & ?\\

        \hline

      4  & 5 & 209 & 11  & 0 & 48 & ?\\
      
\hline\hline
  \end{array}$

\bigskip
  
  \caption{Invariants of the blow-ups of $\pr^4$ along $s$ planes }\label{t1}
\end{table}}
\begin{question}\label{questionP4}
  Is $X_s$ Fano for  $s\in\{3,4\}$? 
\end{question}
\subsection{Blow-ups of quadrics along planes}\label{quadrics}
\noindent Let $Q\subset\pr^5$ be a smooth quadric. Recall that $Q$ contains two families of planes, and distinct planes in the same family intersect pairwise in a point. Let $A_1,\dotsc,A_s\subset Q$ be general planes in the same family, and $\sigma\colon X_s\to Q$ the blow-up of $A_1,\dotsc,A_s$ (see Lemma \ref{blowuporder}), so that $\rho_{X_s}=1+s$.

It is shown in 
\cite{manivel4fold} that $X_5$ is Fano; this implies that $X_s$ is Fano for every $s\leq 5$.
We have  $X_1\cong\pr_{\pr^2}(T_{\pr^2}\oplus\ol_{\pr^2}(2))$ (see \cite[Theorem on first page]{szurekwisnnagoya}), and we also note that $X_2$ is \cite[Fano 3-2]{FTT} and $X_3$ is \cite[Fano 4-2]{FTT}.

For $s\geq 9$ we have $\rho_{X_s}\geq 10$, hence $X_s$ cannot be Fano by Th.~\ref{sharpbound}.
 \begin{question}\label{questionquadric}
    Is $X_s$ Fano for $s\in\{6,7,8\}$?
  \end{question}
  
 For $i,j\in\{1,\dotsc,s\}$, $i<j$, let   $H_{ij}\subset\pr^5$ be the hyperplane spanned by  $A_i$ and $A_j$; then $H_{ij|Q}$ is
 a cone over a smooth quadric surface.
Moreover let $E_i\subset X_s$ be the exceptional divisor over $A_i\subset Q$, for $i=1,\dotsc,s$, and $H\subset X_s$ the pullback of a general hyperplane section.
 
  \smallskip

\noindent {\bf The case $s=2$.\ }
Let $\sigma_1\colon X_1\to Q$ be the blow-up of $A_1$, and $\sigma_2\colon X_2\to X_1$ the blow-up of the transform $\w{A}_2\subset X_1$ of $A_2$.
Moreover let $D_{12}'\subset X_1$ and
$D_{12}\subset X_2$  be the transforms of $H_{12|Q}$.

\begin{figure}[b]\caption{A section of the effective cone of $X_2$, blow-up of a quadric}\label{figuraX2}

\bigskip
  
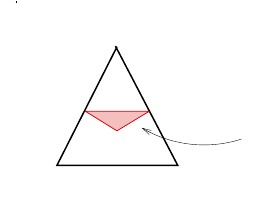
\end{figure}

Under the $\pr^2$-bundle $\pi_1\colon X_1\to\pr^2$, the image of $\w{A}_2\subset X_1$ is a line $\Gamma\subset\pr^2$, and $\pi_1^{-1}(\Gamma)=D'_{12}$. The composition $f_1:=\pi_1\circ\sigma_2\colon X_2\to\pr^2$ is an equidimensional contraction with $f_1^{-1}(\Gamma)=E_2\cup D_{12}$, both fixed prime divisors, and $\tau_{f_1}=\langle [E_2],[D_{12}]\rangle$ (see Def.-Rem.~\ref{dimfiber} and Rem.~\ref{tauf}).

Symmetrically, we can also factor $\sigma\colon X_2\to Q$ as $X_2\stackrel{\alpha}{\to} X_1'\to Q$ where $X_1'\to Q$ is the blow-up of $A_2$, and $\alpha\colon X_2\to X_1'$ the blow-up of the transform of $A_1$. We have a $\pr^2$-bundle $\pi_2\colon X_1'\to\pr^2$, and the composition $f_2:=\pi_2\circ\alpha\colon X_2\to\pr^2$ is an equidimensional contraction with  $\tau_{f_2}=\langle [E_1],[D_{12}]\rangle$.

See Fig.~\ref{figuraX2} for a section of $\Eff(X_2)$; we have $H\equiv D_{12}+H_1+H_2$, $f_i^*\ol_{\pr^2(1)}\equiv H-E_i$ for $i=1,2$, and $-K_{X_2}=4H-E_1-E_2$.

{\small$$\xymatrix{
& {X_1'}\ar[r]^{\pi_2}\ar[dl]&{\pr^2}\\
Q  &{X_2}\ar[u]^{\alpha}\ar[l]_{\sigma}\ar[dr]_{f_1}\ar[ur]^{f_2}\ar[d]_{\sigma_2}\ar[r]^<<<<<{\tau}&{\pr^2\times\pr^2}\ar[u]\ar[d]\\
&{X_1}\ar[lu]^{\sigma_1}\ar[r]_{\pi_1}&{\pr^2}
}$$}

The restriction $\sigma_{1|D'_{12}}\colon D'_{12}\to H_{12|Q}$ is a small resolution of the quadric cone, and $D'_{12}\cong \pr_{\pr^1}(\ol\oplus\ol(1)^{\oplus 2})$. Then $D'_{12}$ is smooth and contains $\w{A}_2$, so that $D_{12}\cong D'_{12}$, and if $C_{D_{12}}\subset D_{12}$ is a line in a fiber of the $\pr^2$-bundle, we have $D_{12}\cdot C_{D_{12}}=-1$. In fact $D_{12}$ is the exceptional divisor of a blow-up $\tau\colon X_2\to \pr^2\times\pr^2$ along a smooth rational curve $C$, which is a complete intersection of divisors of degrees $(1,0),(0,1),(1,1)$.

This can be seen via the embedding $X_2\subset P_2$ where $P_2\to\pr^5$ is the blow-up of $A_1$ and then of  $\w{A}_2$;
$P_2$ is a toric variety,
 and it can be described explicitly using Batyrev's
language of primitive collections and primitive relations \cite{bat91}, and Sato's description of how primitive relations change after smooth blow-ups and blow-downs \cite[\S 4]{sato}.
There is a blow-up $\tau_P\colon P_2\to Y:=\pr_{\pr^2}(\ol\oplus\ol(1)^{\oplus 3})$ along a surface $S\cong\pr^1\times\pr^1$, with exceptional divisor the transform in $P_2$ of the hyperplane $H_{12}\subset\pr^5$ spanned by $A_1$ and $A_2$. We have $X_2\cdot\NE(\tau_P)=0$, $\tau_{P|X_2}=\tau$, $\tau_P(X_2)\in|\ol_{Y}(1)|$, and $\tau_P(X_2)\cong\pr^2\times\pr^2$. Moreover the curve $C:=S\cap\tau_P(X_2)\subset\pr^2\times\pr^2$ is the complete intersection of the divisors $\tau(E_1)\in\ol_{\pr^2\times\pr^2}(1,0)$, $\tau(E_2)\in\ol_{\pr^2\times\pr^2}(0,1)$, and a third divisor in $|\ol_{\pr^2\times\pr^2}(1,1)|$.

\bigskip

\noindent {\bf The case $s=3$.\ }
For $i,j\in\{1,2,3\}$, $i<j$, let $D_{ij}\subset X_3$ be the transform of $H_{ij|Q}$.
Moreover
let $\Lambda\subset\pr^5$ be the plane  spanned by the three points $A_1\cap A_2$, $A_1\cap A_3$, and $A_2\cap A_3$; we note that $\Lambda$ is contained in $Q$ and also in every hyperplane $H_{ij}$.
Let $L\subset X_3$ be the transform of $\Lambda$; then $L$ is an exceptional plane contained in $D_{ij}$ for every $i,j$, and $[C_L]$ generates a small extremal ray of $\NE(X_3)$ with locus $L$.

Let us consider the flip $X_3\dasharrow Y$ of $\R_{\geq 0}[C_L]$. We claim that $Y$ is the blow-up of $\pr^4$ at three general lines, with exceptional divisors (the transforms of) $D_{12}$, $D_{13}$, $D_{23}$ (see \cite{PP21} for a description of $\NE(Y)$ and $\Eff(Y)$). The exceptional line $\ell\subset Y$ given by the flip is the transform of the unique line in $\pr^4$ intersecting the three blown-up lines. Moreover the exceptional divisors $E_1,E_2,E_3$ of $\sigma\colon X_3\to Q$ are the transforms of the hyperplanes in $\pr^4$ spanned by two  blown-up lines.

This can be seen again by consider the toric variety $P_3$ obtained by blowing-up $P_2$ along the transform $\w{A}_3$ of the third plane $A_3\subset\pr^5$; we have $X_3\subset P_3$, and $[C_L]$ still generates a small extremal ray of $\NE(P_3)$, with locus $L$, $K_{P_3}\cdot C_L=0$, and $X_3\cdot C_L=-1$. The flop of this small extremal ray yields a variety isomorphic to $P_3$:
$$\pr^5\longleftarrow P_3\dasharrow P_3\stackrel{\psi}{\la} \pr^5.$$
Then  $\psi\colon P_3\to \pr^5$ is again a sequence of three blow-up of planes, with exceptional divisors the transforms of first $H_{12}$, then $H_{13}$, and finally $H_{23}$. The image of $X_3$ is a hyperplane in $\pr^5$, and $\psi$ blows it up in three disjoint lines.

\begin{figure}\caption{A section of the effective cone of $X_3$, blow-up of a quadric}\label{figuraX3}

\bigskip
  
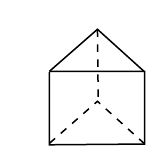
\end{figure}

As in the case $r=2$,
for $i\in\{1,2,3\}$ there is an equidimensional contraction $f_i\colon X_3\to\pr^2$ given by the $\pr^2$-bundle obtained blowing-up $Q$ along $A_i$.
We have $E_2+D_{12}\equiv E_3+D_{13}\equiv f_1^*\ol_{\pr^2}(1)$ and $\tau_{f_1}=\langle [E_2],[E_3],[D_{12}],[D_{13}]\rangle$, and similarly for $f_2$ and $f_3$. See Fig.~\ref{figuraX3} for a section of $\Eff(X_3)$.

 \bigskip

 For $s\geq 3$, we still get $\binom{s}{2}$ fixed prime divisors $D_{ij}\subset X_s$ of type $(3,1)^{\sm}$, given by the transforms of $H_{ij|Q}$, for $i,j\in\{1,\dotsc,s\}$, $i<j$. Moreover for $i,j,h\in\{1,\dotsc,s\}$, $i<j<h$,
  the plane $\Lambda_{ijh}\subset\pr^5$ spanned by the three points $A_i\cap A_j$, $A_i\cap A_h$, and $A_j\cap A_h$, is contained in $Q$ and in $H_{ij},H_{ih},H_{jh}$; its transform $L_{ijh}\subset D_{ij}\cap D_{ih}\cap D_{jh}$ is an exceptional plane.

{\small
\begin{table}[h]
 $\begin{array}{||c|c|c|c|c|c|c||}
   \hline\hline
   
s   & \rho_{X_s} & K_{K_s}^4 & b_4(X_s)=h^{2,2}(X_s) & b_3(X_s)& \chi(X_s,-K_{X_s}) & \text{Fano}\\
      \hline\hline

         0 & 1 & 512 & 2 & 0 & 105 & Q \\

        \hline

       1 & 2 & 432 & 3  & 0 & 90 & \text{yes} \\

       \hline

     2 & 3 & 358 & 5 & 0 & 76  & \text{yes}\\
       
  \hline

      3 & 4 & 290 & 8  & 0 & 63 & \text{yes}\\

        \hline

    4  & 5 & 228 & 12  & 0 & 51 & \text{yes}\\

     \hline

       5 & 6 & 172 & 17  & 0 & 40 & \text{yes} \\

       \hline

     6 & 7 & 122 & 23 & 0 & 30 & ? \\
   
  \hline

      7 & 8 & 78 & 30  & 0 & 21& ? \\

        \hline

      8  & 9 & 40 & 38  & 0 & 13 & ?\\
      
\hline\hline
  \end{array}$

\bigskip
  
  \caption{Invariants of the blow-ups of a quadric $4$-fold along $s$ planes}\label{t2}
\end{table}}
\subsection{Blow-ups of cubic $4$-folds along planes}\label{excubics}
\noindent Let  $Z\subset\pr^5$ be a smooth cubic $4$-fold containing $s$ planes
$A_1,\dotsc,A_s\subset Z$ that intersect pairwise in a point, and $A_i\cap A_j\cap A_k=\emptyset$ for $i<j<k$.

Let $X_s\to Z$ the blow-up of $A_1,\dotsc,A_s$ (see Lemma \ref{blowuporder});
then $X_s$ is a smooth projective $4$-fold with $\rho_{X_s}=1+s$, compare Th.~\ref{final} and Rem.~\ref{conescubic}.

For $s=1$, $X_1$ has two elementary contractions, the blow-up $X_1\to Z$, and a quadric bundle $X_1\to\pr^2$, so that $X_1$ is Fano. For $s=2$, we show below that $X_2$ is still Fano. 
On the other hand
for $s\geq 9$ we have $\rho_{X_s}\geq 10$, hence $X_s$ cannot be Fano by Th.~\ref{sharpbound}.  
  \begin{question}\label{questioncubic}
    For which values of $s\in\{3,\dotsc,8\}$ is $X_s$ Fano?
\end{question}

{\small\begin{table}[h]
 $\begin{array}{||c|c|c|c|c|c|c|c||}
   \hline\hline
   
s   & \rho_{X_s} & K_{K_s}^4 & h^{2,2}(X_s) & h^{1,3}(X_s) &b_3(X_s)& \chi(X_s,-K_{X_s}) & \text{Fano}\\
      \hline\hline

         0 & 1 & 243 & 21 & 1 & 0 & 55 & Z \\

        \hline

       1 & 2 & 192 & 22  & 1 & 0 & 45 & \text{yes}\\

       \hline

     2 & 3 & 147 & 24 & 1 & 0 & 36  & \text{yes}\\

  \hline

      3 & 4 & 108 &27 &  1  & 0 & 28 & ?\\

        \hline

    4  & 5 & 75 & 31 &  1 & 0 & 21 & ?\\

     \hline

       5 & 6 & 48 & 36 & 1  & 0 & 15 & ?\\

       \hline

     6 & 7 & 27 & 42 & 1  & 0 & 10  & ?\\

  \hline

      7 & 8 & 12 & 49 & 1  & 0 & 6 & ?\\

        \hline

      8  & 9 & 3 & 57 & 1  & 0 & 3 & ?\\
      
\hline\hline
  \end{array}$

\bigskip
  
  \caption{Invariants of the blow-ups of a cubic $4$-fold along $s$ planes}\label{t3}
\end{table}}

\noindent {\bf The case $s=2$.\ } We keep the same notation as in \S\ref{quadrics}. We still have $X_2\subset P_2$, where $P_2$ is the toric $5$-fold obtained by blowing-up $\pr^5$ along first $A_1$ and then $\w{A}_2$. One can check that $-K_{P_2}-X_2$ is nef and big on $P_2$, and $(-K_{P_2}-X_2)^{\perp}\cap\NE(P_2)=R$ where $R$ is a small extremal ray with $K_{P_2}\cdot R= X_2\cdot R=0$. The locus of $R$ is a surface $T\cong\pr^2$, thus either $T\subset X_2$, or $T\cap X_2=\emptyset$.

Let us consider the  blow-up $P_1\to\pr^5$ of $A_1$. The surface $T\subset P_2$ is the transform of the fiber $F\subset P_1$ of the blow-up over the point $A_1\cap A_2\in\pr^5$. Then in $P_1$ we have $X_1\cap F=\w{A}_2\cap F$ and this is a line in $F\cong\pr^2$; therefore we see that $X_2$ can intersect $T$ at most in a curve, and we conclude that $X_2\cap T=\emptyset$. This shows that $-K_{X_2}=(-K_{P_2}-X_2)_{|X_2}$ is ample, and $X_2$ is Fano.

Let us also consider the blow-up $\tau_P\colon P_2\to Y=\pr_{\pr^2}(\ol\oplus\ol(1)^{\oplus 3})$ as in \S\ref{quadrics}, and its restriction $\tau:=\tau_{P|X_2}\colon X_2\to W:=\tau_P(X_2)$.
If $C\subset P_2$ is a line in a general non-trivial fiber $F_{\tau}\cong\pr^2$ of $\tau_P$, we have
 $X_2\cdot C=1$, therefore $X_2$ intersects $F_{\tau}$ in a line, and $W$ contains the surface $S\cong\pr^1\times\pr^1$ blown-up by $\tau_P$.  Thus we see that $\tau$ is an elementary divisorial contraction of type $(3,2)$ with exceptional divisor $D_{12}$, the transform of $H_{12|Z}$.

The tautological class $\ol_Y(1)$ is nef and big, and it is zero on a small extremal ray $R_Y$ of $\NE(Y)$ with locus $T_Y:=\tau_P(T)$, which is the negative section of the $\pr^3$-bundle $Y\to\pr^2$. In $P_2$ we have $T\cap X_2=\emptyset$ and $T\cap\Exc(\tau_P)=\emptyset$, thus $W\cap T_Y=\emptyset$ in $Y$.
Moreover
$W\in|\ol_Y(2)|$, $-K_Y\sim \ol_Y(4)$, and again we see that $-K_{W}=\ol_Y(2)_{|W}$ is ample, so that $W$ is Fano of index $2$. 
It is not difficult to see that $W$ is a double cover of $\pr^2\times\pr^2$ with branch divisor of degree $(2,2)$ (compare \ref{malpensa}), and it has two elementary contractions $W\to\pr^2$ that are quadric bundles. By composing with $\tau$, we get two equidimensional contractions $f_i\colon X_2\to\pr^2$, $i=1,2$. Fig.~\ref{figuraX2} still gives a section of $\Eff(X_2)$; here $H$ is the pullback of a general hyperplane section of $Z$, and $-K_Z=3H-E_1-E_2$.

\bigskip

\bigskip

\small

\providecommand{\noop}[1]{}
\providecommand{\bysame}{\leavevmode\hbox to3em{\hrulefill}\thinspace}
\providecommand{\MR}{\relax\ifhmode\unskip\space\fi MR }
\providecommand{\MRhref}[2]{%
  \href{http://www.ams.org/mathscinet-getitem?mr=#1}{#2}
}
\providecommand{\href}[2]{#2}


\end{document}